\documentclass[11pt,letterpaper]{amsart}

\usepackage{amsthm}
\usepackage{amsfonts}         
\usepackage{amsmath}
\usepackage{amssymb}
\usepackage{breqn}
\usepackage{bm}
\usepackage{url}
\usepackage{esint}
\usepackage{fancyhdr}

\newcommand{\N}{\mathcal{N}}
\newcommand{\YY}{\mathbb{Y}}

\newcommand{\ZZ}{\mathbb{Z}}
\newcommand{\RR}{\mathbb{R}}
\newcommand{\CC}{\mathbb{C}}
\newcommand{\QQ}{\mathbb{Q}}

\newcommand{\A}{\mathcal{A}}
\newcommand{\B}{\mathcal{B}}
\newcommand{\C}{\mathcal{C}}
\newcommand{\XX}{\mathbb{X}}

\newcommand{\PP}{\mathbb{P}}

\newcommand{\M}{\mathcal{M}}

\renewcommand{\a}{\mathfrak{a}}
\renewcommand{\b}{\mathfrak{b}}
\renewcommand{\c}{\mathfrak{c}}
\renewcommand{\d}{\mathfrak{d}}

\newcommand{\f}{\mathfrak{f}}

\newcommand{\m}{\mathfrak{m}}
\newcommand{\n}{\mathfrak{n}}
\newcommand{\p}{\mathfrak{p}}

\newcommand{\cond}{\text{cond}}

\theoremstyle{plain}
\newtheorem{theorem}{Theorem}[section]
\newtheorem{lemma}[theorem]{Lemma}
\newtheorem{prop}[theorem]{Proposition}
\newtheorem{cor}[theorem]{Corollary}
\newtheorem{definition}{Definition}

\theoremstyle{remark}
\newtheorem{remark}{Remark}[section]

\numberwithin{equation}{section}
\newcommand{\eps}{\varepsilon}
\renewcommand{\Re}{\text{Re}}
\renewcommand{\Im}{\text{Im}}
\begin{document}
\title{On Gaussian primes in sparse sets}
\author{Jori Merikoski}
\address{Mathematical Institute,
University of Oxford,
Andrew Wiles Building,
Radcliffe Observatory Quarter,
Woodstock Road,
Oxford,
OX2 6GG}
\email{jori.merikoski@maths.ox.ac.uk}
\subjclass[2020]{11N32 primary, 11N36 secondary}

\begin{abstract} 
We show that there exists some $\delta > 0$ such that, for any set of integers $B$ with $|B\cap[1,Y]|\gg Y^{1-\delta}$ for all $Y \gg 1$, there are infinitely many primes of the form $a^2+b^2$ with $b\in B$.
We prove a quasi-explicit formula for the number of primes of the form $a^2+b^2 \leq X$ with $b \in B$ for any $|B|=X^{1/2-\delta}$ with $\delta < 1/10$ and $B \subseteq [\eta X^{1/2},(1-\eta)X^{1/2}] \cap \ZZ$, in 
 terms of zeros of Hecke $L$-functions on $\QQ(i)$. We obtain the expected asymptotic formula for the number of such primes provided that the set $B$ does not have a large subset which consists of multiples of a fixed large integer. In particular, we get an asymptotic formula if $B$ is a sparse subset of primes. For an arbitrary $B$ we obtain a lower bound for the number of primes with a weaker range for $\delta$, by bounding the contribution from potential exceptional characters.
\end{abstract} 

\maketitle
\tableofcontents
\section{Introduction}
The distribution of prime numbers in sparse sets is a central topic in modern analytic number theory. A key motivating question is Landau's fourth problem, which asks if there are infinitely many prime numbers of the form $n^2+1$. This is far beyond the current methods as the set is very sparse -- the number of integers up to $X$ of this form is of order $X^{1/2}$. 

As an approximation to Landau's question much attention has been given to Gaussian prime numbers $p=a^2+b^2$ with $b$ restricted to some specific sparse set $B$. A major breakthrough was achieved by Friedlander and Iwaniec \cite{FI} who proved that there are infinitely many primes of the form $a^2+b^4$, that is, with $B$ being the set of squares. Following this there have been many variants where $b$ is drawn from a sparse set $B$, for instance, the papers of Heath-Brown and Li, Pratt, and the author \cite{hbli,merikoskipoly,pratt}. We also point out the results of Heath-Brown \cite{hb}, Li \cite{li}, and Maynard \cite{maynard} for primes in other polynomial sequences, where Li's result has  the record for the sparsest polynomial sequence with primes, with size of the set being $X^{43/67 + \eps}$.

Notably, all of the above-mentioned results exploit heavily the structure of the  specific sparse set $B$, leaving open the question of what can be said about an arbitrary sparse set $B$. In this direction Fouvry and Iwaniec \cite{fouvryi} proved that one can take $B$ with density $(\log X)^{-C}$ for any $C >0$ and establish an asymptotic formula for the number of primes $a^2+b^2 \leq X$ with $b \in B$. Using the argument of Fouvry and Iwaniec one would be required to improve upon the famous Siegel-Walfisz theorem to reach sparser sets $B$. In comparison, our main result obtains unconditionally a power saving in the density of $B$ for the first time. 
\begin{theorem} \label{qualtheorem}
There is some (computable) $\delta > 0$ such that the following holds. If $B$ is a set of integers with $|B\cap[0,Y]| \gg Y^{1-\delta}$ for all $Y\gg 1$, then there are infinitely many primes of the form $a^2+b^2$ with $b \in B$.
\end{theorem}
This result is a corollary of the following lower bound for the number of such primes, which is weaker than the expected by a factor of $\gg_\eps X^{-\eps}$.
\begin{theorem}  \label{maintheorem}
There is some (computable) $\delta > 0$ such that the following holds for any small $\eta > 0$. For all sufficiently large $X$ and for all $B \subseteq [\eta X^{1/2},(1-\eta)X^{1/2}] \cap \ZZ$ with $|B| \, \geq X^{1/2-\delta}$ we have for any $\eps > 0$
\[
\sum_{p=a^2+b^2 \leq X} \mathbf{1}_B(b) \gg_\eps X^{1/2-\eps}|B|. 
\]
\end{theorem} 
\begin{proof}[Deduction of Theorem \ref{qualtheorem} from \ref{maintheorem}]
 Let $\delta > 0$ be small and  let $B$ be a set of integers with $|B\cap[0,Y]| \gg Y^{1-\delta}$ for all $Y\gg 1$. Then by the pigeonhole principle for any $\eps > 0$ for $Y$ any sufficiently large there is some $Y_1 \in [Y^{1-\delta-\eps},Y]$ such that  $B_1:=B\cap[Y_1/4,Y_1/2]$ satisfies $ |B_1| \geq Y_1^{1-\delta-\eps}$. By Theorem \ref{maintheorem} with $X=Y_1^2$ and a trivial upper bound for $a^2+b^2 \leq X^{1-4\eps }$ we have
 \begin{align*}
     \sum_{X^{1-4\eps } < p=a^2+b^2 \leq X} \mathbf{1}_{B_1}(b) = &     \sum_{p=a^2+b^2 \leq X} \mathbf{1}_{B_1}(b)  -    \sum_{ a^2+b^2 \leq X^{1-4\eps }} \mathbf{1}_{B_1}(b) \\
     \gg_\eps& X^{1/2-\eps} |B_1|.
 \end{align*}
 Therefore, for all large $Y$ there exists a prime number $p=a^2+b^2 \in (Y^{2- 2\delta - 8\eps},Y^2]$ with $b \in B$, so that in particular there are infinitely many primes of the form $a^2+b^2$ with $b \in B$.
\end{proof}

\begin{remark}
A back-of-the-envelope estimate shows that it should be possible to
establish Theorem 1.2 for some $\delta \in (1/20,1/10)$ but we have not checked this as it depends on the  the numerical constants $c_1,c_2,c_3$ in Lemmas \ref{zerofreeregionlemma} and \ref{zerodensitylemma} as well as optimization of Theorem \ref{exceptionaltheorem} -- this would require a separate lengthy optimization similar to the arguments in \cite{hblinnik}.  It is possible to generalize our results to general binary quadratic forms instead of $a^2+b^2$ by adapting ideas from \cite{lsx}.
\end{remark}
\begin{remark}
     The lower bound in Theorem \ref{maintheorem} is the best that can be hoped for in general with current technology. In fact, improving the lower bound  in Theorem \ref{maintheorem} to the correct order of magnitude would imply the non-existence of Siegel zeros by a suitable application of Theorem \ref{asympregulartheorem} below. The implied constant in the lower bound $\gg_\eps$ in Theorem \ref{maintheorem} is ineffective but the result can be made effective with $\eps=\delta$ for some very small $\delta >0$ by using the class number formula, so that Theorem \ref{qualtheorem} is effective in all aspects.
\end{remark}
\begin{remark}
The argument we give works for $\eta = X^{-\eps}$ for some small $ \eps >0$. It is possible to extend the proof of Theorem \ref{maintheorem} to handle sets $B \subseteq [0,X^{1/2}]$. The possibility that $B$ or the variable $a$ is restricted to a narrow interval $[0,X^{1/2-\delta'}]$ for some $\delta' \leq \delta$ adds only a technical problem, namely, all  of the Gaussian primes $b+ia$ counted lie in a very narrow sector.  See Remark \ref{smallBremark} for further details on how to modify our argument. 
\end{remark}

In the case that $B$ is unbiased we are able to get an asymptotic formula similar to \cite{fouvryi} for $\delta < 1/10$.
Let
\[
\rho(d) := |\{\nu \in \ZZ/d \ZZ: \, \nu^2+1 \equiv 0\, (d) \}|
\]
and define
\[
\omega(b) := \prod_{\substack{ p\mid b 
 }}\bigg(1- \frac{\rho(p)}{p}\bigg)^{-1}.
\]
Then the following result is a corollary of our quasi-explicit formula (cf. Theorem \ref{asympregulartheorem}).
\begin{theorem} \label{asymptotictheorem}
Let $\eta > 0$ be small. Let $C''>0$ be large compared to $C'>0$ which is large compared to $C >0$. Let $\lambda_b$ be complex coefficients with $|\lambda_b| \, \leq X^{o(1)}$, supported on  $ [\eta X^{1/2},(1-\eta)X^{1/2}] \cap \ZZ$, and satisfying
\[
\sum_{b} |\lambda_b| \, \geq X^{2/5+\eps}.
\]
 Suppose that for all $(\log X)^{C''} < q \leq X^{\delta+\eta}$ we have 
\begin{align} \label{lambdaassumption}
   \sum_{b \equiv 0 \, (q)} |\lambda_b| \, \leq \frac{1}{(\log X)^{C'}}\sum_{b} |\lambda_b|.
\end{align}
Then
\[
\sum_{a^2+b^2\leq X} \lambda_b \Lambda(a^2+b^2) = \frac{4}{\pi}\sum_{\substack{ a^2+b^2 \leq  X \\  (a,b)=1}}  (\lambda_b\omega(b) +O_{C,C',C''}(|\lambda_b|(\log X)^{-C})).
\]
\end{theorem}
\begin{remark}
  With more work the assumptions that $|\lambda_b| \leq X^{o(1)}$  and $\sum_b |\lambda_b| \, \geq X^{2/5+\eps}$  may be replaced (here and later) by
  \[
  \|\lambda_b\|_1 \geq X^{1/5+\eps} \|\lambda_b\|_2,
  \]
  where $\|\lambda_b\|_p := \left( \sum_{b}|\lambda_b|^p \right)^{1/p}$. By adapting Harman's sieve \cite{harman} one can get a correct order lower bound for  $1/2-\delta$ for some $\delta \in (1/8,1/10)$.
\end{remark}
\begin{remark}
We note that $\rho(p)=1+\chi_4(p)$ with $\chi_4$ being the unique non-trivial character to modulus 4 and
\begin{align} \label{piproduct}
\frac{4}{\pi} = \prod_p \bigg( 1-\frac{\rho(p)}{p}\bigg) \bigg( 1-\frac{1}{p}\bigg)^{-1}. 
\end{align}
Also for all $b$
\[
\sum_{\substack{ a^2+b^2 \sim  X \\  (a,b)=1}}  \omega(b) \sim \sum_{\substack{ a^2+b^2 \sim  X }} \frac{\varphi(b)}{b} \omega(b) = \sum_{\substack{ a^2+b^2 \sim  X }}  \prod_{p|b} \bigg( 1-\frac{\rho(p)}{p}\bigg)^{-1} \bigg( 1-\frac{1}{p}\bigg),
\]
so that the main term is of the same form as in \cite[Theorem 1]{fouvryi}. 
\end{remark}

The assumption \eqref{lambdaassumption} in Theorem \ref{asymptotictheorem} is very mild since $C''$ can be taken to be large compared to $C'$ and thus \eqref{lambdaassumption} applies to most instances that come up in nature. For example, we immediately get 
 as a corollary \cite[Theorem 3]{merikoskipoly} for large $k$ and a weak version of \cite[Theorem 1.1]{pratt} where the base is taken to be large with $B$ being the set of numbers missing one digit.

In particular, we get an asymptotic formula if $B$ is a sparse subset of primes.

\begin{cor}
   Let \[
    B \subseteq [\eta X^{1/2},(1-\eta)X^{1/2}] \cap \PP
    \] 
    with $|B| \geq X^{2/5+\eps}.$
Then for any $C>0$
\[
\sum_{a^2+b^2\leq X}  \mathbf{1}_B(b)\Lambda(a^2+b^2) = \frac{4}{\pi}\sum_{\substack{ a^2+b^2 \leq  X }} \mathbf{1}_B(b)  + O_C\bigg(\frac{X^{1/2}|B|}{(\log X)^C}\bigg)
\]
\end{cor}

The assumption \eqref{lambdaassumption} in Theorem \ref{asymptotictheorem} may be replaced by a wide zero-free region for Hecke $L$-functions (cf. Section \ref{lfunctionsection} for the relevant notations). For the statement we let $\M(u)$ denote the smallest integer $m$ with $u|m$.
\begin{theorem} \label{GRHtheorem}
For all $C > 0$ there is some $C'>0$ such that the following holds. Let $\lambda_b$ be complex coefficients with $|\lambda_b| \, \leq X^{o(1)}$, supported on  $ [\eta X^{1/2},(1-\eta)X^{1/2}] \cap \ZZ$, and satisfying
\[
\sum_{b} |\lambda_b| \, \geq X^{2/5+\eps}.
\]
Assume that Hecke $L$-functions $L(s,\xi_k \chi)$ on $\QQ(i)$ with $|k| \leq X^\eta$ and modulus $\M(u) \leq X^{1/10+\eta}$ have no zeros in the region
\begin{align} \label{zeroassumptionstrong}
  \sigma > 1- \frac{C' \log \log X}{\log X}, \quad |t| \leq  X^\eta.  
\end{align}
Then
\[
\sum_{a^2+b^2\leq X} \lambda_b \Lambda(a^2+b^2) = \frac{4}{\pi}\sum_{\substack{ a^2+b^2 \leq  X \\  (a,b)=1}}  (\lambda_b\omega(b) +O_{C}(|\lambda_b|(\log X)^{-C})).
\]
\end{theorem}
\subsection{Overview of the proof}

Our goal is to estimate
\begin{align} \label{mainsketchthingy}
  \sum_{p=a^2+b^2 \sim X} \mathbf{1}_B(b),  
\end{align}
where $|B| = X^{1/2-\delta}.$   It is immediately apparent that potential Siegel zeros can cause a problem, since if $\chi_1 \in \widehat{(\ZZ/q_1\ZZ)^\times}$ is an exceptional character (necessarily quadratic) to modulus $q_1 \leq X^{\delta}$ and if $B \subseteq q_1 \ZZ$, then for $b \in B$
\[
\mu(a^2+b^2) \approx \chi_1(a^2+b^2) = \chi_1(a^2) = 1,
\]
implying that the sequence $a^2+b^2$ is biased towards numbers with an even number of prime factors.  In this case we would expect the main term for \eqref{mainsketchthingy} to be multiplied essentially by $L(1,\chi_1)$. Note that the function
\[
b+ia \mapsto \chi(a^2+b^2)
\]
defines a Dirichlet character on $(\ZZ[i]/q_1\ZZ[i])^\times$.

Similarly as in the proof of Linnik's theorem \cite{linnik}, our argument splits into two cases depending on whether there is a Siegel zero or not. We choose a small parameter $\eps_1>0$ and we will  take $\delta$ to be small in terms of $\eps_1$.  If there is a zero
\[
\beta_1 \geq 1- \frac{\eps_1}{\log X},
\]
then we can use a similar argument as in \cite[Chapter 24.2]{odc} to give a lower bound for primes of the form $a^2+b^2$, using just Type I information. This works for a certain fixed $\eps_1$ once $\delta>0$ is sufficiently small.

For simplicity let us then assume in this sketch that we are in the situation of Theorem \ref{GRHtheorem}, that is, for Dirichlet characters $\chi$ to moduli $u \in \ZZ[i]$ with $\M(u) \leq X^{ \delta+\eta}$ Hecke $L$-functions on $\QQ(i)$ have no zeros in the wider region \eqref{zeroassumptionstrong}. By Vaughan's identity evaluating \eqref{mainsketchthingy} is reduced to estimating
\begin{align}
 \text{Type I sums}: \quad \quad   \label{typeisketchintro}
    \sum_{d \leq D} \alpha(d) \sum_{\substack{a^2+b^2 \sim X \\ a^2+b^2 \equiv 0 \, (d)}} \mathbf{1}_B(b)
\end{align}
and for $MN=X$
\begin{align}   \label{typeiisketchintro}
 \text{Type II sums}: \quad 
S(\alpha,\beta):= \sum_{\substack{m \sim M  \\ n \sim N}} \alpha(m) \beta(n) \sum_{\substack{mn=a^2+b^2}} \mathbf{1}_B(b),  
\end{align}
where $\alpha,\beta$ denote bounded coefficients.

For the Type I sums \eqref{typeisketchintro}  the argument goes back to \cite{fouvryi} and we get an asymptotic formula for $D=X^{1-\delta-o(1)}$ by applying Poisson summation for the free variable $a$ and the quadratic large sieve (cf. Lemma \ref{quadraticlargesievelemma}). For small $\delta$ the exponent of distribution approaches $1-o(1)$, so that we only need very little parity breaking Type II information.

For the Type II sums \eqref{typeiisketchintro} it suffices to consider the case when $\beta=\mu$, the M\"obius function. Similar to \cite{FI}, by unique factorization in $\QQ(i)$ we essentially have for $w,z \in \ZZ(i)$ 
\[
a^2+b^2 = mn = |\overline{w}z|^2 = |b+ia|^2
\]
and get
\[
S(\alpha,\mu) = \sum_{\substack{|w|^2 \sim M \\ |z|^2 \sim N}} \alpha(|w|^2) \mu(|z|^2) \mathbf{1}_B(\Re(\overline{w}z)).
\]
After using Cauchy-Schwarz similarly to \cite{FI} we essentially need to show cancellation in
\begin{align*}
    \sum_{|z_1|^2,|z_2|^2 \sim N}& \mu(|z_1|^2) \mu(|z_2|^2) \sum_{\substack{b_1,b_2 \in B \\b_1z_2\equiv b_2z_1 \, (\Im(z_2\overline{z_1}))}} 1  \\
    &=: \sum_{|z_1|^2,|z_2|^2 \sim N} \mu(|z_1|^2) \mu(|z_2|^2) T_B(z_1,z_2).
\end{align*}
The goal is to evaluate the sum $T_B(z_1,z_2)$ with a main term $M_B(z_1,z_2)$ and then bound
\begin{align} \label{musumintro}
 \sum_{|z_1|^2,|z_2|^2 \sim N} \mu(|z_1|^2) \mu(|z_2|^2) M_B(z_1,z_2)   
\end{align}
by catching the oscillations from $\mu(|z|^2)$. All of the previous works \cite{FI,hbli,merikoskipoly,pratt} rely on specific analytic and arithmetic properties of the set $B$ to evaluate the sum $T_B(z_1,z_2)$. 

For a general sparse set $B$ we first factor out $b_0=(b_1,b_2)$ to get, assuming for simplicity that $b_0| \Im (z_2 \overline{z_1})$,
\[
T_B(z_1,z_2) = \sum_{b_0}  \sum_{\substack{b_0b'_1,b_0b'_2 \in B \\b'_1z_2\equiv b'_2z_1 \, (\Im(z_2\overline{z_1})/b_0)}} 1.
\]
It is crucial to carefully track the dependency on $b_0$ since it is possible that a large subset of $B\times B$ has a large common factor $b_0$, for instance, if $B$ is in $q\ZZ$ for some fixed $q \leq X^{\delta}$. 

To evaluate $T_B(z_1,z_2)$ use Dirichlet characters to expand the congruence $b'_1z_2\equiv b'_2z_1 \, (\Im(z_2\overline{z_1})/b_0)$. The contribution from Dirichlet characters with a large conductor $d \geq X^{\delta+\eta}/b_0$ may be bounded by using the large sieve for multiplicative characters (Lemma \ref{largesievelemma}). The Dirichlet characters with a small conductor $d < X^{\delta +\eta}/b_0$ give the main term $M_B(z_1,z_2)$ and we can bound \eqref{musumintro} provided that for any $f=db_0 \leq X^{\delta+\eta}$ and $a \in (\ZZ/f \ZZ)^{\times}$ we have
\begin{align*} 
 \sum_{\substack{|z_1|^2, |z_2|^2 \sim N
 \\ z_2 \equiv a z_1  \, (f)}} \mu(|z_1|^2)\overline{\mu(|z_2|^2)} \ll \frac{N^2}{f^2(\log X)^C}.
\end{align*}
This follows by standard arguments (eg. using Heath-Brown's identity for $\mu$) from the assumption \eqref{zeroassumptionstrong}, provided that $N$ is sufficiently large compared to $f$, say, $N > X^\eta f^3$. We will give a more detailed sketch of the argument for the Type II sums in Section \ref{typeiisketch}. 

Without assuming the zero-free region \eqref{zeroassumptionstrong}, for the Type II sums we need to extract the contribution from the potential exceptional characters before applying Cauchy-Schwarz to $S(\alpha,\mu)$.  Denoting $\mu_z:= \mu(|z|^2)$, we write
\[
\mu_z= \mu^{\#}_z + \mu^{\flat}_z,
\]
where $\mu^{\#}_z$ is an approximation for $\mu_z$ and  $\mu^{\flat}_z$ is a function which is balanced along arithmetic progressions (cf. \eqref{eq:balanceddef} below). Importantly, $\mu^{\#}_z$ must be simple enough so that $S(\alpha,\mu^{\#})$ can be evaluated using only Type I information. 

 Let $\chi_j$ for $j \leq J=(\log X)^{O(1)}$ denote the worst of the possible exceptional characters $\chi_j$ of conductors $u_j \in \ZZ[i]$ with $\M(u_j) \leq X^{\delta+\eta}$ and denote \emph{the normalized 
correlation} of two coefficients $\alpha,\beta$ by
\[
\C(\alpha,\beta):= \bigg(\sum_{|z|^2\sim N} \alpha_z \overline{\beta_z} \bigg)\bigg(\sum_{|z|^2 \sim N} |\beta_z| \bigg)^{-1}
\]
and.
Then for our approximation we essentially choose 
\[
\mu^{\#}_z :=  \sum_{j\leq J} \overline{\chi_j}(z) \C(\mu, \overline{\chi_j} ).
\]
 The idea is similar in spirit to the dispersion method of Drappeau \cite{drappeau} which also takes into account contributions from multiple characters. This construction is also motivated by the Prime number theorem of Gallagher \cite[Theorem 7]{gallagher} and its use by Montgomery and Vaughan to get a power saving for the exceptional set in the binary Goldbach problem \cite{mvexceptional}.
For the  function $\mu^\flat := \mu_z-\mu^{\#}_z$ we can then unconditionally show  that for any $f \leq X^{\delta+\eta}$, $N > X^\eta f^3$, and $a \in (\ZZ/f \ZZ)^{\times}$ 
\begin{align} \label{eq:balanceddef} 
 \sum_{\substack{|z_1|^2, |z_2|^2 \sim N
 \\ z_2 \equiv a z_1  \, (f)}} \mu^\flat_z\overline{\mu^\flat_z} \ll \frac{N^2}{f^2(\log X)^C}.
\end{align}
Taking into account the bias $S(\alpha,\mu^\sharp)$ from the exceptional characters for the Type II sums, we obtain the quasi-explicit formula Theorem \ref{asympregulartheorem}.

The paper is structured as follows. In the Section \ref{lemmasection} we present some basic lemmas. In Section \ref{setupsection} we split the proof of Theorem \ref{maintheorem} into two cases depending on the existence of a Siegel zero (Theorems \ref{exceptionaltheorem} and \ref{regulartheorem}) and state the quasi-explicit formula Theorem \ref{asympregulartheorem}. In Section \ref{typeisection} we evaluate Type I sums and in Section \ref{exceptionalsection} we give a proof of Theorem \ref{maintheorem} under the assumption that a Siegel zero does exist. In Sections \ref{typeiipreliminarysection}, \ref{typeiimainsection}, and \ref{sharptypeiipropsection} we estimate Type II sums. In Sections \ref{asympregularproofsection}, \ref{asymptoticproofsection}, and \ref{grhsection} we deduce our main theorems.

\subsection{Notations and conventions}
\begin{itemize}
    \item $\a, \b ,\c ,\d,\f,\m, \n,\p$ -- ideals of $\ZZ[i]$, reserving $\p$ for prime ideals
    \item $u,v,w,z$ -- Gaussian integers 
    \item $(z,w)$ -- Greatest common primary divisor
    \item $\chi,\psi,u$ -- Dirichlet characters $\chi,\psi \in \widehat{(\ZZ[i]/u\ZZ[i])^\times},$ $u \in \ZZ[i]$ 
    \item $\xi_k(z)$ -- Character $\xi_k(z) = (z/|z|)^k=e^{ik\arg z}$
    \item $\N\n$ -- Norm of an ideal, $\N (z) = |z|^2$
    \item $\M(z)$ -- Smallest integer $m$ with $z | m$, $\N(z)^{1/2} \leq \M(z) \leq \N(z)$.
    \item $k,\ell,m,n,p,q,r,s$ -- Integers, reserving $p$ for a prime number
    \item $\nu_j$ -- Small power of $X$, $\nu_j= X^{-\eta_j}$
    \item $\eta$ -- A generic small constant
    \item $C$ -- A generic large constant
     \item $F$ -- A smooth compactly supported function
    \item $H_N(\n)$ -- Indicator of  $\N\n \in [N,N(1+\nu_2)]$ for $\nu_2=X^{-\eta_2}$.
    \item $W$ -- Denotes $X^{1/(\log\log X)^2}$.
    \item $P(W)$ -- Denotes $\prod_{p < W} p$
\end{itemize}

\section{Lemmas} \label{lemmasection}
\subsection{Introducing finer-than-dyadic smooth weights} \label{smoothweightsection}
Here we describe a device that allows us to partition a sum smoothly into finer-than-dyadic intervals. Let $\nu \in (0,1/10)$ be small (we will use $\nu=X^{-\eps}$ or $\nu= \log^{-C} X$), and fix a non-negative $C^\infty$-smooth function $F$ supported on $[1-\nu,1+\nu]$ and satisfying
\[
|F^{(j)}| \ll_j \nu^{-j},\, \, j \geq 0 \quad \text{and} \quad \int_{1/2}^2 F(1/t) \frac{dt}{t} = \nu.
\]
Suppose that we want to introduce a smooth partition to bound a sum of the form $\sum_{n \leq N} f_n$. We can write (using a change of variables $t \mapsto t/n$)
\[
\sum_{n \leq N} f_n = \frac{1}{\nu} \sum_{n \leq N} f_n \int_{1/2}^2 F(1/t) \frac{dt}{t} =\frac{1}{\nu}   \sum_{n \leq N} \int_{1/2}^{2N} f_n F(n/t)  \frac{dt}{t}= \frac{1}{\nu} \int_{1/2}^{2N}  \sum_{n \leq N} f_n F(n/t) \frac{dt}{t},
\]
so that 
\[
\bigg| \sum_{n \leq N} f_n \bigg| \leq \frac{1}{\nu} \int_{1/2}^{2N} \bigg| \sum_{n \leq N} f_n F(n/t) \bigg| \frac{dt}{t}.
\]
Hence, at the cost of a factor $\nu^{-1} \log N$, it suffices to consider sums of the form
\[
 \sum_{n \leq N} f_n F(n/t)
\]
for $t \leq  2N$. Naturally, if the original sum is
$ \sum_{n \sim N} f_n$ (dyadic $n$), then it suffices to consider $\sum_{n \sim N} f_n F(n/t)$ for $t \asymp N$ at a cost $\nu^{-1}$. The effect is the same as with the usual smooth partition of unity except that we did not need to explicitly construct the partition functions $F$. We will denote $F_t (n):=F(n/t)$.

We also have the following variant on $\RR/2 \pi \ZZ$. Fix $\nu_1=X^{-\eta_1}$ for some  small $\eta_1 >0$ and let $G: \RR/2 \pi \ZZ \to \CC$ be a non-negative $C^\infty$-smooth function supported on $[-\nu_1,\nu_1]$, satisfying
\[
|G^{(j)}| \ll_j \nu_1^{-j},\, \, j \geq 0 \quad \text{and} \quad \int G(\theta) d \theta = \nu_1.
\]
Then  we may write for a sum over Gaussian integers
\[
\sum_{z \in \ZZ[i]} f_z =  \frac{1}{\nu_1}\int_{\RR/2 \pi \ZZ}  \sum_{z \in \ZZ[i]} f_z G(\arg z - \theta) d\theta
\]
to obtain a smooth finer-than-dyadic partition in terms of $\arg z$.
\subsection{Elementary estimates}
\begin{lemma} \label{gcdlemma}
\[
\sum_{1 \leq a \leq A} (a,d)  \leq \tau(d) A.
\]
\end{lemma}
\begin{proof}
Denoting $c=(a,d)$ and $a=ca'$ we have
\[
\sum_{1 \leq a \leq A} (a,d)  \leq \sum_{c|d} c \sum_{1 \leq a' \leq  A/c} 1 \leq \tau(d) A.  \qedhere
\]
\end{proof}
In handling the weights $\mathbf{1}_{(n,P(W))=1}$ we can use the following standard bound for exceptionally smooth numbers, which effectively gives a version of the Fundamental lemma of the sieve
(see \cite[Chapter III.5, Theorem 1]{Ten}, for instance). 
\begin{lemma} \label{smoothlemma} For any $2 \leq Z \leq Y$ we have
\begin{align*}
\sum_{\substack{n \sim Y \\ P^+(n) < Z}} 1 \, \ll \, Y e^{-u/2},
\end{align*}
where $u:= \log Y/\log Z.$
\end{lemma}
We also require the following elementary bound (see \cite[Lemma 1]{fisieve}, for instance).
\begin{lemma} \label{divisorlemma0} For every square-free integer $n$ and every $k \geq 2$ there exists some $d|n$ such that $d \leq n^{1/k}$ and
\begin{align*}
\tau(n) \leq 2^k \tau(d)^k.
\end{align*}
\end{lemma}
From this we get the more general version.
\begin{lemma} \label{divisorlemma} For every integer $n$ and every $k \geq 2$ there exists some $d|n$ such that $d \leq n^{1/k}$ and
\begin{align*}
\tau(n) \leq 2^{k^2} \tau(d)^{k^3}.
\end{align*}
\end{lemma}
\begin{proof}
Write $n=b_1 b_2^2 \cdots b_{k-1}^{k-1} b_k^k$ with $b_1, \dots, b_{k-1}$ square-free, by letting $b_k$ be the largest integer such that $b_k^k | n$, so that $n/b_k^k$ is $k$-free and splits uniquely into $b_1 b_2^2 \cdots b_{k-1}^{k-1}$ with $b_j$ square-free. We have
\[
\tau(n) \leq \tau(b_1) \tau(b_2)^2 \cdots \tau(b_k)^k.
\]
By Lemma \ref{divisorlemma0} for all $j \leq k-1$ there is some $d_j|b_j$ with $d_j \leq b_j^{1/k}$ and $\tau(b_j) \leq 2^k\tau(d_j)^k$. Hence, for $d=d_1\cdots d_{k-1} b_k$ we have 
\[
d \leq (b_1\cdots b_{k-1})^{1/k} b_k  \leq (b_1\cdots b_{k-1}b_k^k)^{1/k}  \leq n^{1/k}
\]
and
\[
\tau(n) \leq (2\tau(d_1)\cdots \tau(d_{k-1}) \tau(b_k))^{k^2} \leq 2^{k^2}\tau(d)^{k^3}.  \qedhere
\]
\end{proof}
\subsection{Sieve bounds}
The following version of the Fundamental lemma of the sieve follows from applying \cite[Theorem 6.9]{odc} to the real and the imaginary parts of $a_n$.
\begin{lemma} \emph{(Fundamental lemma of the sieve).} \label{flsievelemma}
Let $\A=(a_n)$ be a sequence of complex coefficients and let $\kappa \geq 0, Z \geq 2,$ $D \geq Z^{9\kappa +2}$. Suppose that for some $\XX$ and for some real-valued multiplicative function $g(d)$ we have for all square-free $d$
\[
\sum_{n \equiv 0 \, (d)} a_n = g(d) \XX + r_d.
\]
and suppose that for all $p$ we have $0 \leq g(p) < 1$. Suppose that for some $K > 1$ we have for all $W < Z$
\[
\prod_{W \leq p < Z} (1-g(p))^{-1} \leq K \bigg( \frac{\log Z}{\log W}\bigg)^\kappa.
\]
Denote
\[
V(Z) := \prod_p (1-g(p)).
\]
Then for some bounded coefficients $\lambda_d$ depending only on $\kappa$ we have
\[
\sum_{(n,P(Z))=1} a_n = \XX V(Z)(1+O_{\kappa,K}(e^{ -s})) + O\bigg( \bigg|\sum_{\substack{ d < D \\d|P(Z)}} \lambda_d  r_d\bigg|\bigg).
\]
 \end{lemma}
\begin{remark} \label{flsieveremark}
    The fact that $\lambda_d$ depend only on $\kappa$ and not the sequence $a_n$ will be important for us since we will be applying the same sieve to several sequences $a_n^{(b)}$ indexed by $b \in B$ and then using the summation over $b \in B$ while bounding the remainder $\sum_{b \in B}\left|\sum_{d < D} \lambda_d  r^{(b)}_d\right|$ (cf. Proposition \ref{typeiprop}). In our set-up the functions $g^{(b)}(d)$ and $V^{(b)}(Z)$ also depend on $b$ with $g^{(b)}(d)= \mathbf{1}_{(d,2b)=1} \rho(d)/d$. Thus, we have
\[
\sum_{b \in B} V^{(b)}(Z) = V(z)  \sum_{b \in B} \prod_{\substack{p| b \\ 2 < p < Z}} (1-\rho(p)/p)^{-1},
\]
where the product over $p|b,2 < p<  Z$ may be completed to  all prime factors $p|b, p\neq 2$ with a negligible error term if $Z$ is not too small.
\end{remark}

\subsection{Poisson summation}
The following lemma gives a truncated version of the Poisson summation formula.
\begin{lemma}\emph{(Poisson summation).}  \label{poissonlemma} Let $F$ be as in Section \ref{smoothweightsection} for some $\nu \in (0,1/10)$ and denote $F_N(n):=F(n/N)$. Let $x \gg 1$ and let $q \sim Q$ be an integer. Let $\eps >0$ and denote
\[
H:= \nu^{-1} (QN)^{\eps} Q/N
\]
Then for any $A >0$
\begin{align*}
\sum_{n \equiv a \, (q)} F_N(n) = \frac{N}{q} \hat{F}(0) + \frac{N}{q} \sum_{1 \leq |h| \leq H} \widehat{F} \left( \frac{h N}{q}\right) e_q (-ah) + O_{A,\epsilon,F}((QN)^{-A}),
\end{align*}
where $\hat{f}(h):= \int f(u)e(hu) du$ is the Fourier transform.
\end{lemma}
\begin{proof}
By the usual Poisson summation formula we have
\[
\sum_{n \equiv a \, (q)} F_N(n)=\frac{N}{q} \sum_{h} \widehat{F} \left( \frac{h N}{q}\right) e_q (-ah)
\]
For $|h| > H$ we have by integration by parts $j\geq 2$ times
\[
\widehat{F} \left( \frac{h N}{q}\right) = \int F(u)e(u hN/q) du \ll_j \nu^{-j+1} (hN/q)^{-j} \ll_j \nu (QN)^{-j\eps}(h/H)^{-2},
\]
which gives the result.
\end{proof}

We will also need the two-dimensional Poisson summation formula.
\begin{lemma} \label{poisson2dlemma}
Let $F: \RR^2 \to \CC$ be a $C^\infty$-smooth compactly supported function such that
\[
|F^{(j,k)}| \ll \nu_1^{-j} \nu_2^{-k}, \quad \int_{\RR^2} |F|  \ll \nu_1\nu_2
\]
and let $ F_{\bm{N}}(\n) := F(n_1/N_1,n_2/N_2)$. Then
\[
\sum_{\bm{n} \in \ZZ^2} F_{\bm{N}}(\n) = \sum_{\bm{h} \in \ZZ^2} \widehat{F_{\bm{N}}}(\bm{h})
\]
and
\[
\widehat{F_{\bm{N}}}(\bm{h}) \ll_C \nu_1\nu_2 N_1 N_2 (1+\nu_1|h_1|/N_1)^{-C} (1+\nu_2|h_2|/N_2)^{-C}.
\]
\end{lemma}
\subsection{Fourier expansions}
We have the following lemma on Mellin transforms, where the construction of the non-negative majorant $\tilde{F}$ is from the proof of \cite[Lemma 7.1]{IK} (cf. the function $g(y)$, with $x_m=\log m$).
\begin{lemma} \label{improvedmvtlemma}
Fix a non-negative $C^\infty$-smooth function $F$ supported on $[1-\nu,1+\nu]$ and satisfying
\[
|F^{(j)}| \ll_j \nu^{-j},\, \, j \geq 0 \quad \text{and} \quad \int_{1/2}^2 F(1/t) \frac{dt}{t} = \nu.
\]
Then for any $c \in \RR$
\[
F(x) = \frac{1}{2 \pi i}\int_{c-i\infty}^{c+i\infty}  \Dot{F}(s) x^{-s} ds
\]
where the Mellin transform is
\[
\Dot{F}(s) := \int_0^\infty F(x) x^{s} \frac{dx}{x} \ll_{C,\sigma} \nu (1+\nu |s|)^{-C}.
\]
Furthermore, $|\Dot{F}(it)|$ has a smooth majorant $\tilde{F}(t)$ such that  for $m,n\sim M$
\[
\bigg|\int \tilde{F}(t) (m/n)^{it} \, dt \bigg| \ll   \, \mathbf{1}_{|m-n| \leq \nu M}.
\]
\end{lemma}
Similarly, we have the following lemma on Fourier series, which will be applied with the characters
\[
e^{ik\arg z} = \bigg( \frac{z}{|z|}\bigg)^k
\]
to expand smooth weights $G(\arg z)$.
\begin{lemma} \label{fourierserieslemma}
Let $\theta\in \RR/2\pi \ZZ$ and let $G$ be a bounded smooth function supported on $[\theta-\nu,\theta+\nu]$ and satisfying
\[
|G^{(j)}| \ll_j \nu^{-j},\, \, j \geq 0 \quad and \quad \int_{\RR/2\pi \ZZ} G(\alpha) d \alpha = \nu.
\]
Then
\[
G(\alpha) = \sum_{k} \check{G}(k) e^{ik\alpha}
\]
with
\[
\check{G}(k) := \int_{\RR/2 \pi \ZZ} F(\alpha) e^{-ik \alpha} d\alpha \ll_C \nu (1+ \nu |k|)^{-C}.
\]
Furthermore, there is a majorant  $\widetilde{G}(k) $ of $|\check{G}(k)|$ such that
\[
\bigg|\sum_{k} \widetilde{G}(k) e^{ikx} \bigg|\ll \mathbf{1}_{|x| \leq  \nu}.
\]
\end{lemma}

\subsection{Large sieve bounds}
For Type II sums we need the multiplicative large sieve inequality of Bombieri and Davenport (cf. \cite[(9.52)]{odc}, for instance).
\begin{lemma} \label{largesievelemma} For any complex numbers $\gamma_n$ we have
\[
\sum_{d\leq D} \frac{d}{\varphi(d)} \sideset{}{^\ast} \sum_{\chi\, (d)} \bigg| \sum_{n \leq N} \gamma_n \chi(n) \bigg|^2 \ll (D^2+N)  \sum_{n \leq N} |\gamma_n|^2.
\]
\end{lemma}
For Type I sums we need the large sieve inequality for roots of quadratic congruences (cf. \cite{fouvryi} and especially \cite[Lemma 14.4]{fiillusory} for this variant with the twist by $\overline{q}$).
\begin{lemma} \label{quadraticlargesievelemma} Let $q \geq 1$. For any complex numbers $\gamma_n$ we have
\[
\sum_{\substack{d\sim D \\ (d,q)=1}} \sum_{\nu^2+1 \equiv 0 \, (d)} \bigg| \sum_{n \leq N} \gamma_n e_d (\nu n \overline{q}) \bigg|^2 \ll (qD+N)  \sum_{n \leq N} |\gamma_n|^2.
\]
\end{lemma}
\subsection{Zeros of Hecke $L$-functions} \label{lfunctionsection}
We say that a Gaussian integer $z$ is  primary if $z \equiv 1 \,\, (2(1+i))$, so that every odd ideal of $\ZZ[i]$ has a unique primary generator. Note that this definition is multiplicative. We extend any  function $\psi:\ZZ[i] \to \CC$ to a  function on odd ideals by 
defining $\psi(\a):= \psi(z)$ if $z$ is the primary generator of $\a$. For $k \in \ZZ$ we let $\xi_k$ denote the  character
\[
\xi_k(z) := \bigg(\frac{z}{|z|}\bigg)^k = e^{i k \arg z},
\]
which controls the angular distribution of $z$.  For a Dirichlet character $\chi \in \widehat{(\ZZ[i]/u\ZZ[i])^\times}$ with a modulus $u \in \ZZ[i]\setminus\{0\}$ we define the  Hecke $L$-function by
\[
L(s,\xi_k \chi) := \sum_{\substack{\a \subseteq \ZZ[i] \\ (\a,2)=1}} \frac{ \xi_k(\a)\chi(\a)}{(\N \a )^{s}} =  \sum_{\substack{z \in \ZZ[i] \setminus\{0\} \\ z \, \text{primary}}} \frac{\xi_k(z)\chi(z)}{|z|^{2s}}.
\]

For a modulus $u$ we define
\[
L_u(s,\xi_k) :=  \prod_{\chi \in \widehat{(\ZZ[i]/u\ZZ[i])^\times}}L(s,\xi_k \chi)
\]

We have the following lemmas, where all the constants are effectively computable. We will not need the Deuring-Heilbronn zero repulsion as we deal with the case of a Siegel zero via a different method. The first lemma is classical (cf. \cite[Chapter 5]{IK}, for instance).
\begin{lemma} \emph{(Zero-free region, Landau-Page).}  \label{zerofreeregionlemma} There is a constant $c_1>0$ such that the function $L_u(s,\xi_k)$ has at most one zero in the region
\[
\sigma > 1- \frac{c_1}{\log (|u|(2+|t|)(2+|k|))}.
\]
If such a zero exists, then it is real and simple, $k=0$, and it is a zero of some $L(s,\chi_1)$ with a quadratic character $\chi_1$.
\end{lemma}

We let $N^\ast(\alpha,T,K,Q)$ denote the number of zeros of $L(s,\xi_k\chi)$ with primitive characters $\chi$, $|k| \leq K$, and modulus $|u|^2 \leq Q$ with $\sigma \geq \alpha,|t| \leq T$. The following lemma is a generalization of \cite[Theorem 6]{gallagher} to Gaussian integers. 
\begin{lemma}  \emph{(Log-free zero density estimate).} \label{zerodensitylemma} There is some constants $c_2,c_3>0$ such that
\[
N^\ast(\alpha,T,K,Q) \leq c_3 (Q^2KT)^{c_2(1-\alpha)}.
\] 
\end{lemma}

As a corollary to the zero-free region (Lemma \ref{zerofreeregionlemma}) we get the following lemma, by taking $\delta>0$ small enough in terms of $c_1$
and for two different moduli $u_1,u_2$ applying Lemma \ref{zerofreeregionlemma} with $u=u_1u_2$
\begin{lemma} \label{differentmodulilemma} Let $\delta > 0$ be sufficiently small in terms of $c_1$. Then there is at most one modulus $|u_1| \leq X^{2\delta}$ with a primitive character $\chi_1$ such that $L(s,\xi_k \chi_1)$ has a zero $\beta_1 \geq 1-\frac{1}{\sqrt{\delta} \log X}$. If such a zero exists, then it is real and simple, $k=0$, and it is a zero of some $L(s,\chi_1)$ with a real character $\chi_1$.
\end{lemma}

The following lemma is proved by the same argument as in \cite[(11.7)]{mv}.
\begin{lemma} \label{Lnearzerolemma} Let $\chi \in \widehat{(\ZZ[i]/u\ZZ[i])^\times}$ for some $|u|^2 \leq Q$ and let $|k| \leq Q$. If $L(s,\xi_k\chi)$ has no zeros counted by $N^\ast(\alpha,T,K,Q) $, then for all $\sigma >  (1+\alpha)/2, |t| \leq T, |k| \leq K$ we have
\[
|L(s,\xi_k\chi)|^{-1} \, \ll \log Q
\]
\end{lemma}
\subsection{Character sums}
We need the following lemma for computing sums over primitive characters.
\begin{lemma} \label{primitivecharactersum}
For any $a$ we have
\[
\bigg|\sideset{}{^\ast}\sum_{\chi \, (d)} \chi(a)  \bigg| \leq  (a-1,d),
\]
where the sum extends over primitive Dirichlet characters of $(\ZZ/d\ZZ)^\times$.
\end{lemma}
\begin{proof}
By the Chinese remainder theorem we have
\[
S(a;d):=\sideset{}{^\ast}\sum_{\chi \, (d)} \chi(a) = \prod_{p^k|| d}S(a; p^k).
\]
Thus, it suffices to show that
\[
| S(a; p^k) | \leq (a-1,p^k)
\]
This follows from
\begin{align*}
     S(a; p^k) &= \sum_{\chi \, (p^k)} \chi(a) - \sum_{\chi \, (p^{k-1})} \chi(a)  \\
     &=  \varphi(p^{k})\mathbf{1}_{a \equiv 1 \, (p^k)} -    \varphi(p^{k-1}) \mathbf{1}_{a \equiv 1 \, (p^{k-1})}.  \qedhere
\end{align*}

\end{proof}

We will need the smoothed Poly\'a-Vinogradov bound on Gaussian integers, which is a consequence of Poisson summation (Lemma \ref{poisson2dlemma}) and the Gauss sum bound on $\ZZ[i]$. 
\begin{lemma} \label{pvongaussianlemma}
Let $F:\RR^2 \to \CC$ be as in Lemma \ref{poisson2dlemma} and for $z=x+iy \in \CC$ define $F(z) := F(x,y)$. Let $\chi$ be a character of modulus $u \in \ZZ[i]$ Then
\begin{align*}
    \sum_{z} F(z)\chi(z) \ll |u|.
\end{align*}
\end{lemma}

The following lemma is required for the proof of Theorem \ref{asymptotictheorem}.
\begin{lemma}\label{gausssumlemma} Let $k\geq 1$ and $p$ be a prime number. Let $\chi \in \widehat{(\ZZ[i]/p^k\ZZ[i])^\times}$ with  conductor $p^k$  or $\pi^{k_1}\bar{\pi}^{k_2}$ with $k=\max\{k_1,k_2\}$. Then
\begin{align*} 
\sum_{r \, (p^k)}\chi(r+ip^\ell) \leq 2 p^{k/2+\ell/2}.    
\end{align*}
\end{lemma}
\begin{proof}
Let $q:=p^k$. We have by $r \mapsto r/s$
\begin{align*}
  \sum_{r \, (q)} \chi(r+i p^\ell) =  \frac{1}{\varphi(q)} \sum_{(s,q)=1}  \sum_{r \, (q)} \chi(r+ip^\ell) = \frac{1}{\varphi(q)} \sum_{r,s \, (q)} \overline{\chi}(s)\chi(r+isp^\ell).
\end{align*}
The function $ s \mapsto \overline{\chi}(s)$ defines a multiplicative character over the integers and we let $q_1|q$ denote its conductor. By expansion of $\overline{\chi}(s)$ into additive characters we get
\[
\bigg|\sum_{r\, (p^k)} \chi(r+ip^\ell) \bigg| \leq  \frac{1}{q_1 \varphi(q)} \sum_{a \, (q_1) } \bigg|\sum_{t \, (q_1)}  \overline{\chi}(t)e_{q_1}(-at) \bigg| \bigg|\sum_{r,s \, (q)} e_{q_1}(as)\chi(r+isp^\ell)\bigg|
\]
We have (writing $s \mapsto s /p^{\ell}+tp^{k-\ell}, p^\ell|s$)
\[
\sum_{r,s \, (q)} e_{q_1}(as)\chi(r+isp^\ell) = \sum_{\substack{z \, (q) \\ s \equiv 0\, (p^\ell)}} e_{q_1}(as/p^\ell) \chi(z)\sum_{t\, (p^\ell)} e_{q_1}(at p^{k-\ell}).
\]
Here
\[
\sum_{t\, (p^\ell)} e_{q_1}(at p^{k-\ell}) = p^\ell \mathbf{1}_{a \equiv 0 \, (q_1/(q_1,p^{k-\ell}))}.
\]
Hence, denoting $q_2=(q_1,p^{k-\ell})$ and making the change of variables $a\mapsto a q_1/q_2$ we get
\[
\bigg|\sum_{r\, (p^k)} \chi(r+ip^\ell)\bigg| \leq \frac{p^{\ell}}{q_1 \varphi(q)} \sum_{a \, (q_2) } \bigg|\sum_{t \, (q_1)}  \overline{\chi}(t)e_{q_2}(-at) \bigg| \bigg|\sum_{\substack{z \, (q) \\ s \equiv 0\, (p^\ell)}} e_{p^\ell q_2}(as) \chi(z) \bigg|.
\]
Since $\chi(t)$ is of conductor $q_1$ we have the Gauss sum bound \cite[Theorem 9.12]{mv}
\begin{align} \label{integergausssum}
 \bigg|\sum_{t \, (q_1)}  \overline{\chi}(t)e_{q_2}(-at)\bigg| \leq q_1^{1/2} .   
\end{align}
For the second sum we expand the condition $s \equiv 0 \, (p^\ell)$ with additive characters to get
\[
\sum_{\substack{z \, (q) \\ s \equiv 0\, (p^\ell)}} e_{p^\ell q_2}(as) \chi(z) =  \frac{1}{p^\ell}\sum_{b \, (p^\ell)} \sum_{\substack{z \, (q)}} e_{p^\ell q_2}(as) e_{p^\ell}(bs) \chi(z) 
\]
The function 
\[z \mapsto e_{p^\ell q_2}(as) e_{p^\ell}(bs) = e \bigg( \frac{(a+q_2 b)s}{q_2p^{\ell}} \bigg) =: e_{q}(cs)
\]
is an additive character on $\ZZ[i]/q\ZZ[i]$ since $q_2|p^{k-\ell}$. If $\chi$ is of conductor $\pi^{k_1}\bar{\pi}^{k_2}$, then we get by the Gauss sum bound (a direct generalization of \cite[Theorem 9.12]{mv} to $\ZZ[i]$, denoting $r:=q/(c,q)$)
\begin{align} \label{Gausssum}
  \bigg|\sum_{\substack{z \, (q)}} e_{q}(cs) \chi(z) \bigg| \leq \mathbf{1}_{\pi^{k_1}\bar{\pi}^{k_2} | r}  \frac{\varphi_{\ZZ[i]}(q)}{\varphi_{\ZZ[i]}(r)} q \leq q .  
\end{align}
since $\pi^{k_1}\bar{\pi}^{k_2} | r$ and $k=\max\{k_1,k_2\}$ implies that $r=p^k=q$. We also get the same bound if the conductor of $\chi$ is $p^k$ as then $\chi$ is primitive. Putting the bounds \eqref{integergausssum} and \eqref{Gausssum} together and using $q_2=(q_1,p^{k-\ell})$ we get (noting that $q/\varphi(q) \leq 2$ for $q=p^k$)
\begin{align*}
    \bigg|\sum_{r \, (p^k)}\chi(r+ip^\ell)  \bigg|
 &\leq \frac{p^{\ell}}{q_1 \varphi(q)} q_2 q_1^{1/2} q \leq 2 p^{\ell} q_2 q_1^{-1/2}  \\
 &\leq 2 p^\ell q_2^{1/2} \leq 2 p^{k/2+ \ell/2} . \qedhere
\end{align*}
\end{proof}
The previous lemma implies the following.
\begin{lemma} 
\label{charsumboundlemma}
Let $b \in \ZZ$ and $u\in \ZZ[i]$ and let $Y > |u|^{4}$ . Let  $\chi$ be a primitive character modulo $u$ and let $v=(u,b)$. Then for any integer $a_0$
\[
\sum_{\substack{a \in (Y,Y+Y^{1-\eta}] \\ a \equiv a_0 \, (4)\\ (a,b)=1}} \chi(b+ia) \ll_\eps \frac{Y^{1-\eta}}{|u/v|^{1/2-\eps}}.
\]
\end{lemma}
\begin{proof}
 We have
  \begin{align*}
        \sum_{\substack{a \in (Y,Y+Y^{1-\eta}] \\a \equiv a_0 \, (4) \\ (a,b)=1}} \chi(b+ia) = \sum_{\substack{c | b \\(c,4 u)=1}} \mu(c)  \sum_{\substack{ac  \in (Y,Y+Y^{1-\eta}] \\ ac \equiv  a_0 \, (4)}} \chi(b+i ac).
  \end{align*}
Let $u=m w$ where $m$ consists of all prime factors $p\equiv 3 \, (4)$. Let $n$ denote the smallest integer such that $w| n$. The contribution from $c > Y/mn$ is trivially bounded by
  \[
  \ll  \sum_{\substack{c |b \\c > Y/mn }} \frac{Y}{c} \ll mn \tau(b).
  \]
  For $c \leq Y/mn$ we have
  \[
 \sum_{\substack{ac \in (Y,Y+Y^{1-\eta}] \\ ac \equiv  a_0 \, (4)}}\chi(b+iac) = \frac{Y^{1-\eta}}{16 cm n} \sum_{\substack{r \, ( 4mn) \\ rc  \equiv a_0 \, (4)} }\chi(b+irc) + O(mn) \ll \frac{Y^{1-\eta} (b,mn)^{1/2}}{(mn)^{1/2}}
  \]
once we show that
  \[
   \sum_{\substack{r \, ( 4mn) \\ rc  \equiv a_0 \, (4)} }\chi(b+irc)=  \sum_{\substack{r \, ( 4mn) \\ r  \equiv a_0 \, (4)} }\chi(b+ir)\leq (mn)^{1/2}(b,mn)^{1/2} 
  \] 
To prove this, write
\[
\chi= \prod_{p^k||m} \chi_{p^{k}} \prod_{\pi^k|| w} \chi_{\pi^{k}}.
\]
By the Chinese remainder theorem we get  (denoting $k=\max\{k_1,k_2\}$)
\[
    \sum_{\substack{r \, ( 4mn) \\ r  \equiv a_0 \, (4)} }\chi(b+ir)  =  \bigg(\prod_{p^k|| m} \sum_{r \, (p^k)} \chi_{p^k}(b+ir)  \bigg) \bigg(\prod_{\pi^{k_1} \bar{\pi}^{k_2} || w} \sum_{r \, (p^k)} \chi_{\pi^{k_1}} \chi_{\bar{\pi}^{k_2}}(b+ir)\bigg)
\]
and the claim follows by Lemma \ref{gausssumlemma}.    
\end{proof}

\section{Set-up and statement of the quasi-explicit formula} \label{setupsection}
Let $\lambda_b$ be divisor-bounded coefficients and define
\[
\omega_2(b) := 2\prod_{\substack{ p\mid b \\ p \neq 2  }}\bigg(1- \frac{\rho(p)}{p}\bigg)^{-1}.
\]
Define the sequences over Gaussian integers $\A=(a_z)$ with
\[
a_z :=  \lambda_{\Re(z)} \mathbf{1}_{(z,\overline{z})=1}, \quad a^{\omega}_z := \lambda_{\Re(z)}  \omega_2(\Re(z)) \mathbf{1}_{(z,\overline{z})=1}.
\]
Note that $(z,\overline{z})=1$ implies that $(z,(1+i))=1$ and for $z=b+ia$ that $(a,b)=1$.

For an ideal $\n= (z)$ we define 
\[
a_{\n} = \sum_{u \in \{ \pm 1, \pm i\}} a_{uz}, \quad a^{\omega}_{\n} = \sum_{u \in \{ \pm 1, \pm i\}} a^{\omega}_{uz}
\]
so that for any function $f$ on the ideals we have
\[
\sum_{\N \n \sim X} a_\n f(\n) =  \sum_{|z|^2 \sim X} a_z f((z)) =  \sum_{\substack{a^2+b^2 \sim X \\(a,b)=1\\ (a^2+b^2,2)=1}} \lambda_b f((b+ia)).
\]

In the case that there is a Siegel zero for the character $\chi_1$, for any finite set of integers $B$ let $B_1 \subseteq B$ be the largest subset such that for all $b \in B_1$ we have
\begin{align} \label{B1condition}
 \sum_{\substack{a \sim (X-b^2)^{1/2} \\ (a,b)=1 \\ (a^2+b^2,2)=1}} \chi_1 ((b+ia)) > 0   
\end{align}
and let us denote
\[
\Omega(B) := \sum_{\N \n\sim X} a^{\omega}_\n \quad \text{for} \quad \lambda_b := \mathbf{1}_{B}(b).
\]

We split the proof of Theorem \ref{maintheorem} into two cases depending on whether there is a Siegel zero $\beta_1 > 1-\eps_1/\log X$ or not and according to the size of $\Omega(B_1)$. Theorem \ref{maintheorem} is then an immediate corollary of the following two theorems. 
\begin{theorem}\emph{(Exceptional case).} \label{exceptionaltheorem}
Let $\eps_1 \in (0,1/10)$ . Let $B \subseteq [\eta X^{1/2},(1-\eta)(2X)^{1/2}] \cap \ZZ$ with $|B| = X^{1/2-\delta}$ and let $\lambda_b=\mathbf{1}_{B}(b)$. Suppose that for some $\chi_1$ 
of modulus $\M(u_1) \leq X^{\delta+\eta}$ the $L$-function $L(s,\chi_1)$ has a zero $\beta_1 > 1-\eps_1/\log X$ and that $\Omega(B_1) \geq \Omega(B)/2$.
Suppose that $\delta$ is sufficiently small in terms of $\eps_1$. Then we have for all $\eps > 0$
\[
\sum_{\N \p\sim X} a_{\p}  \gg_\eps X^{1/2-\eps}|B|.
\]
\end{theorem}

\begin{theorem}\emph{(Regular  case).} \label{regulartheorem}
Let $\eps_1 \in (0,1/10)$. Let $B \subseteq [\eta X^{1/2},(1-\eta)(2X)^{1/2}] \cap \ZZ$ with $|B| = X^{1/2-\delta}$ and let $\lambda_b=\mathbf{1}_{B}(b)$. Suppose that the  $L$-functions $L(s,\chi)$ have no zeros $\beta > 1-\eps_1/\log X$ or that $\Omega(B_1) \leq \Omega(B)/2$ and suppose that $\delta$ is sufficiently small in terms of $\eps_1$. Then
\[
\sum_{\N \p\sim X} a_\p  \gg \eps_1  \frac{1}{\log X} \sum_{\N \n\sim X} a^{\omega}_\n.
\]
\end{theorem}
In Theorem \ref{regulartheorem} the possible exceptional zero is not an Siegel zero, the potential zeros are only somewhat close to the line $\Re(s)=1$ and we show that these zeros can have only a small influence on the main term. This means that in all error terms it suffices to save only a power of $\log X$ instead of a power of $X$. Theorem \ref{exceptionaltheorem} is proved in Section \ref{exceptionalsection}. Theorem \ref{regulartheorem} is proved in Section \ref{sec:regularthmproof} and is a quick consequence of the following result. Recall that $\M(u)$ denotes the smallest integer $m$ with $u|m$.
\begin{theorem} \emph{(Quasi-explicit formula).}\label{asympregulartheorem}
Let $\eta > 0$ be small. Let $\delta \in (0,1/10)$ and $X \gg 1$. For every $C_1> 0$ there is some $C_2>$ such that for some 
 \[
 J \leq (\log X)^{C_2}
 \]
there is a set of primitive Hecke characters $\{\xi_{k_j} \chi_j\}_{j \leq J}$ with Dirichlet characters $\chi_j$ to moduli $u_j \in\ZZ[i]$ with $\M(u_j) \leq X^{\delta+\eta}$ and $|k_j| \leq X^\eta$  such that the following holds. Let $\lambda_b$ be  coefficients with $|\lambda_b| \, \leq X^{o(1)}$, supported on  $ [\eta X^{1/2},(1-\eta)(2X)^{1/2}] \cap \ZZ$, and satisfying
\[
\sum_{b} |\lambda_b| \, \geq X^{1/2-\delta}.
\] 
Then
\begin{align*}
  \sum_{\N \n\sim X} a_\n \Lambda(\N \n)  =  \frac{4}{\pi} \sum_{\N \n\sim X} a^{\omega}_\n \bigg(1- \sum_{j\leq J} \overline{\xi_{k_j}\chi_j}(\n) \sum_{\substack{\rho_j \\ L(\rho_j,\xi_{k_j}\chi_j) = 0 \\ |\emph{\Im}(\rho_j)| \leq X^\eta }}(\N \n)^{\rho_j-1}\bigg) \\
  + O\bigg( \frac{1}{(\log X)^{C_1}} X^{1/2} \sum_{b} |\lambda_b|\bigg).  
\end{align*}

\end{theorem}

The restriction $\M(u_j) \leq X^{\delta+\eta}$ instead of a condition involving $|u_j|$ may appear unusual but in fact it occurs very naturally in the proof of the Type II estimate (Proposition \ref{typeiiprop}), where we consider the distribution of Gaussian integers in arithmetic progressions to moduli $d$ which are regular integers, so  that $u_j|d$.
We will prove Theorem \ref{asympregulartheorem} in Section \ref{asymptoticproofsection}. It is possible that the range $\delta < 1/10$ may be improved a bit with further work but we seem to hit a hard barrier at $\delta=1/6$, cf. Remark \ref{bottleneckremark} for more details. It is also plausible that one could obtain a power saving in the error term by taking into account $J \leq X^\eta$ bad characters. The factors of $X^\eta$ in the ranges for $u_j,k_j,\Im(\rho_j)$ may be replaced by $(\log X)^{C_3}$ for some large $C_3 > 0$ but this is inconsequential for our applications.  We also note that the right-hand side may be expressed as a sum over Gaussian integers by writing
\begin{align*}
    \sum_{\N \n \sim X} a^{\omega}_\n \overline{\xi_{k_j} \chi_j}(\n)&(\N\n)^{\rho_j-1} = \sum_{u \in \{\pm 1, \pm i\}} \sum_{|z|^2 \sim X} a^{\omega}_{uz} \overline{\xi_{k_j} \chi_j}(z)(|z|^2)^{\rho_j-1} \mathbf{1}_{z \equiv 1 \, (2(1+i))} \\
   & = \frac{1}{4} \sum_{\chi \in \ZZ[i]/2(1+i)\ZZ[i]} \sum_{u \in \{\pm 1, \pm i\}} \sum_{|z|^2 \sim X} a^{\omega}_{uz} \overline{\xi_{k_j} \chi_j \chi }(z)(|z|^2)^{\rho_j-1} \\
  &  =\sum_{\chi \in \ZZ[i]/2(1+i)\ZZ[i]} \mathbf{1}_{\xi_{k_j} \chi_j \chi (i)=1} \sum_{|z|^2 \sim X} a^{\omega}_{z} \overline{\xi_{k_j} \chi_j \chi }(z)(|z|^2)^{\rho_j-1},
\end{align*}
where the last equation follows from the change of variables $z \mapsto z/u$ and summing over $u$.

We conclude this section by making elementary reductions for the proof of Theorem \ref{asympregulartheorem}, to reduce to the case when $\lambda_b=\mathbf{1}_B(b)$ with $B$  in a short interval.
\subsection{Reduction to bounded $\lambda_b$}
We can reduce from coefficients satisfying $|\lambda_b| \, \leq X^{o(1)}$ to $|\lambda_b| \, \leq 1$ as follows. Let $\eps> 0$ be small enough so that $\delta+2\eps < 1/10$ and denote
\begin{align*}
    B_0 := \{b: 0<|\lambda_b|\, \leq 1\} \quad \text{and} \quad   B_j:= \{b: 2^{j-1} <|\lambda_b|\, \leq 2^{j}\}, \quad 1 \leq j \leq \log X.
\end{align*}
The contribution from $b \in B_j$ with $|B_j| \leq X^{1/2-\delta-\eps}$ is negligible, since for such $j$  by $|\lambda_b| \, \leq X^{o(1)}$
\begin{align*}
    \sum_{b \in B_j} |\lambda_b| \,\ll X^{o(1)} |B_j|\, \ll X^{1/2-\delta-\eps+o(1)}.
\end{align*}
For each $j$ such that $|B_j|\, > X^{1/2-\delta-\eps}$ we can renormalize $\lambda_b$ by $2^{-j}$ to get bounded coefficients which satisify
\begin{align*}
    \sum_{b \in B_j} \, |2^{-j}\lambda_b | \,\geq \frac{1}{2} |B_j| \geq   X^{1/2-\delta-2\eps}.
\end{align*}
Then the general case follows by applying the bounded case for each such $j$ separately with the weights $\lambda_{b}^{(j)} := 2^{-j}\lambda_b \mathbf{1}_{b \in B_j} $. Indeed, if we denote the claim in Theorem \ref{asympregulartheorem} by $S(\lambda)= \mathrm{M}(\lambda) + O(\mathrm{E}(\lambda))$, then assuming that the theorem holds for the bounded $\lambda^{(j)}_b$ we get by linearity
\begin{align*}
    S(\lambda) = \sum_j  2^j S(\lambda^{(j)}) = \sum_j  2^j  (\mathrm{M}(\lambda^{(j)}) + O(\mathrm{E}(\lambda^{(j)})))  = \mathrm{M}(\lambda) + O(\mathrm{E}(\lambda)).
\end{align*}
Thus, it suffices to prove Theorem \ref{asympregulartheorem} for $|\lambda_b| \leq 1$.
\subsection{Reduction from $\lambda_b$ to $\mathbf{1}_B$}
We can reduce the proof from general bounded weights $\lambda_b$ to the weights of the type $\mathbf{1}_B(b)$ by a finer-than-dyadic decomposition in terms of the values of $\lambda_b$.  That is, we write for $\nu=(\log X)^{-C}$
\[
I_j := \begin{cases} ((1-\nu)^{j+1},(1-\nu')^{j}]\, \quad 0 \leq j < (\log X)/\nu \\
[0,(1-\nu)^{j}], \quad j= \lfloor   (\log X)/\nu\rfloor
\end{cases}
\]
\[
\mathbf{1}_{|z|_\infty \leq 1} = \sum_{j_1,j_2} \mathbf{1}_{z \in I_{j_2}\times I_{j_2}}.
\]
Since $|\lambda_b| \leq 1$, we obtain a partition 
\[
\lambda_b = \sum_{j_1,j_2}  \left(\mathbf{1}_{\lambda_b \in I_{j_2}\times I_{j_2}} ((1-\nu)^{j_1} +i (1-\nu)^{j_2})  + O(\nu |\lambda_b| \mathbf{1}_{\lambda_b \in I_{j_2}\times I_{j_2}})\right),
\]
where the contribution from error term is negligible by crude bounds.
We consider $\mathbf{1}_B$ for
\[
B:= B(j_1,j_2)= \{b \in [\eta X^{1/2},(1-\eta)X^{1/2}]: \lambda_b \in I_{j_2}\times I_{j_2} \}
\]
and note that for $j_1,j_2$ with $|B(j_1,j_2)| < X^{1/2-\delta-\eta}$ we can bound the contribution trivially. Hence, it suffices to show Theorem \ref{asympregulartheorem} for $\lambda_b= \mathbf{1}_B(b)$.
\subsection{Reduction to $B$ in a short interval}
Let $B \subseteq [\eta X^{1/2},(2-\eta)X^{1/2}]$ and split
\begin{align*}
    B= \bigcup_j B_j, \quad B_j := B \cap [j X^{1/2-\eta},(j+1)X^{1/2-\eta}]. 
\end{align*}
The contribution from $B_j$ with $|B_j| \leq X^{-2\eta} |B|$ is negligible by a crude bound. Thus, we only need to deal with $|B_j| \geq |B|X^{-2\eta} = X^{1/2-\delta-2\eta}$. Therefore, it suffices to prove Theorem \ref{asympregulartheorem} for $\lambda_b= \mathbf{1}_B(b)$ with
\begin{align} \label{Bintervalassumption}
 B \subseteq [Y,Y + X^{1/2-\eta}]  \quad  \text{for some} \quad Y\in [\eta X^{1/2},(2-\eta)X^{1/2}].
\end{align}
 We will from now on always assume that 
 $\lambda_b= \mathbf{1}_B(b)$ for a set $B$ satisying (\ref{Bintervalassumption}).

\section{Type I information} \label{typeisection}
For for $b \in B$ and for any function $f$ on the ideals of $\ZZ[i]$ denote
\[
a_{z,f}^{(b)} := \mathbf{1}_{\Re(z)=b}\mathbf{1}_{(z,\overline{z}=1)} f((z)).
\]
Define
\[
g^{(b)}(w) := \mathbf{1}_{(w,b \overline{w})=1} \frac{1}{|w|^2}.
\]
We have Type I information provided by the following proposition. 
\begin{prop} \label{typeiprop}\emph{(Type I information).} Let $\alpha_w$ be divisor-bounded coefficients supported on $(w,2\overline{w})=1$. Let $\chi$ be a Dirichlet character to modulus $u$ and let $\xi=\xi_k$ with $
|k| \ll X^{\eta/1000}.$ Let $q := \M(u) \leq X^{1/4}$. Then 
\begin{align*}
    \sum_{b \in B} \bigg| \sum_{\substack{|w|^2 \leq X^{1-\delta-\eta}/q^2 \\ (w,u)=1}} \alpha_w  \bigg(\sum_{\substack{ |z|^2 \sim X \\ z \equiv 0 \, (w)}} a_{z,\xi\chi}^{(b)} -  g^{(b)}(w) \sum_{\substack{ a \sim (X^2-b^2)^{1/2} \\  (a,b)=1 \\ (a^2+b^2,2)=1}}  \xi\chi((b+ia))\bigg)\bigg| \\
    \ll X^{1/2-\eta/100}|B|.
\end{align*}
The same is true if $|z|^2 \sim X$ is replaced by $|z|^2 \in [X',X'(1+\nu)]$ or by a smooth weight $F(|z|^2/X')$ as in Section \ref{smoothweightsection} with $X'\sim X$ and $\nu \geq X^{-\eta/1000}$, with the same change applied to the condition $a \sim (X^2-b^2)^{1/2}$. 
\end{prop}
\begin{proof}

Let us split $d$ dyadically into $d \sim D \leq X^{1-\delta-\eta}/q^2$ and denote
\[
S^{(b)}(\alpha) :=   \sum_{\substack{|w|^2 \sim D \\ (w,u)=1}} \alpha_w \bigg( \sum_{\substack{ |z|^2 \sim X \\ z \equiv 0 \, (w)}} a_{z,\xi \chi}^{(b)} -  g^{(b)}(w) \sum_{\substack{ a \sim (X^2-b^2)^{1/2} \\  (a,b)=1 \\ (a^2+b^2,2)=1}}  \xi\chi((b+ia))\bigg)\bigg). 
\]
We apply Section \ref{smoothweightsection} with $\nu=X^{-\eta/80}$ to split $a$ into finer than dyadic ranges to get (bounding the part $A \leq X^{1/2-\eta}$  trivially using the divisor bound and dropping the condition $a^2+b^2 \sim X$)
\[
S^{(b)}(\alpha) =  \frac{1}{\nu}\int_{X^{1/2-\eta}}^{2X^{1/2}} S^{(b)}(\alpha,A) \frac{dA}{A} + O(X^{1/2-\eta/100})
\]
with
\begin{align*}
 S^{(b)}(\alpha,A) :=  \sum_{\substack{|w|^2 \sim D \\ (w,u)=1}} \alpha_w \bigg(&\sum_{\substack{a\equiv -bi \, (w) \\ (a,b)=1 \\(a^2+b^2,2)=1 }} F_A(a)\xi\chi((b+ia))  \\
 &- g^{(b)}(w)  \sum_{\substack{  (a,b)=1 \\ (a^2+b^2,2)=1}}  F_A(a)\xi\chi((b+ia)) \bigg).
\end{align*}
Since $b+ia$ is odd, we have
\[
\xi\chi((b+ia)) = \xi\chi(u) \xi\chi(b+ia)
\]
for some unit $u$ depending on $z$ modulo $2(1+i)$. Since $a$ is restricted to a short interval, for a fixed $b$ the function $\xi_k(b+ia)$ with $k \ll X^{\eta/1000}$ is equal to a constant up to a negligible error term and may therefore be dropped.  By splitting $B$ into residue classes modulo $4$ we may assume that $u$ does not depend on $b$. 
Thus, we can split $a$ into congruence classes modulo $4q$ and $d=|w|^2$ with $a \equiv a_0d  \, (4q)$ to get for some unit $u_{a_0d}$ 
\begin{align*}
 S^{(b)}(\alpha,A) = 
\sum_{\substack{|w|^2 \sim D \\ (2bu,w)=1}} \alpha_w &  \sum_{a_0 \, (4q)} \xi\chi(u_{a_0d})\chi(b+i a_0d) \\
& \times \sum_{\substack{\nu \in \ZZ/d\ZZ \\ \nu \equiv -i \, (w)}} \bigg(\sum_{\substack{a \equiv \nu b \, (d) \\ a \equiv  a_0d  \, (4q)\\ (a,b)=1}} F_A(a)  - g^{(b)}(w)  \sum_{\substack{ a \equiv a_0 d\, (4q)\\ (a,b)=1}}  F_A(a)\bigg)  ,
\end{align*}
where the sum over $\nu$ contains just the one element $\nu=-r/s$ for $w=r+is$.
Note that $(a,b)=1$ implies that $(w,b)=1$ and that $(w,\overline{w})=1$, so that $(r,s)=1$ and we may restrict to $\nu \in \ZZ/d\ZZ$ in the above.

It suffices to show that 
\begin{equation*} 
\begin{split}  \sum_{\substack{b \in B}} \bigg| \sum_{\substack{|w|^2 \sim D \\ (2bu,w)=1}} &\alpha_w   \sum_{a_0 \, (4q)} \xi\chi(u_{a_0d}) \chi(b+i a_0d) \\
& \times\sum_{\substack{\nu \, (d) \\ \nu \equiv -i \, (w)}} \bigg(\sum_{\substack{a \equiv \nu b \, (d) \\ a \equiv  a_0d  \, (4q)\\ (a,b)=1}} F_A(a)  - g^{(b)}(w)  \sum_{\substack{ a \equiv a_0 d\, (4q)\\ (a,b)=1}}  F_A(a)\bigg) \bigg| \\
  & \ll X^{1/2-\eta/40}|B|
  \end{split}
\end{equation*}

Let us first deal with the cross-condition $\chi(b+i a_0d)$ between $b$ and $d$. If $(b,q)=1$ we could just make the change of variables $a_0 \mapsto a_0 b$ and $\chi(b)$ would factor out. In general  we have for some characters $\chi_{p^k}$ modulo $p^k$
\[
 \chi(b+i a_0d) = \prod_{p^k|| q}   \chi_{p^k}(b+i a_0d).
\]
Denoting 
\[
(b,q) = q_1 = \prod_{p^\ell || (b,q)} p^\ell
\]
and making the change of variables $a_0 \mapsto a_0 b/p^\ell$ modulo $p^k$ we get
\begin{align*}
 \chi(b+i a_0d) =  \bigg(\prod_{p^\ell || q_1}  \chi_{p^k}(b/p^\ell) \bigg) \prod_{\substack{p^k|| q \\ p^\ell || q_1 }}   \chi_{p^k}(p^\ell+i a_0d) 
\end{align*}

Note that the first factor depends only on $b$ and the second factor no longer depends on $b$ but on $q_1$.  The residue classes $b/p^\ell$ and $p^\ell$ combine to unique residue classes $ \theta_{b,q}$ and $\gamma_{q_1,q}$ modulo $8q$ by the Chinese remainder theorem.
Thus, we get
\begin{align*}
    \sum_{q_1 | q}\sum_{\substack{b \in B \\ (b,q)=q_1}} \bigg| &\sum_{\substack{|w|^2 \sim D \\ (2bu,w)=1}} \alpha_w   \sum_{a_0 \, (4q)}  \xi\chi(u_{a_0d})\chi(\gamma_{q_1,q}+i a_0d)    \\
   &\times \sum_{\substack{\nu \, (d) \\ \nu \equiv -i \, (w)}} \bigg(\sum_{\substack{a \equiv \nu b \, (d) \\ a \equiv  a_0\theta_{b,q} d  \, (4q)\\ (a,b)=1}} F_A(a)  - g^{(b)}(w)  \sum_{\substack{ a \equiv a_0  \theta_{b,q} d\, (4q)\\ (a,b)=1}}  F_A(a)\bigg) \bigg|.  
\end{align*}
Dropping the condition $(b,q)=q_1$, using the divisor bound for $\sum_{q_1 | q} 1 =\tau(q)$, taking the sum  over $a_0$ to the the outside, and absorbing the factor $\chi(u_{a_0d})\chi(\gamma_{q_1,q}+i a_0d) $ into the coefficient $\alpha_w$, it suffices to show that for any $a_0$
\begin{align*}
    \sum_{\substack{b \in B \\ (b+ia_0,u)=1}} \bigg| \sum_{\substack{|w|^2 \sim D \\ (2bu,w)=1}} \alpha_w     \sum_{\substack{\nu \, (d) \\ \nu \equiv -i \, (w)}} &\bigg(\sum_{\substack{a \equiv \nu b \, (d) \\ a \equiv  a_0\theta_{b,q} d  \, (4q)\\ (a,b)=1}} F_A(a)  - g^{(b)}(w) & \sum_{\substack{ a \equiv a_0 \theta_{b,q} d\, (4q)\\ (a,b)=1}}  F_A(a)\bigg) \bigg| 
 \\
 &\ll \frac{X^{1/2-\eta/39}|B|}{q}
\end{align*}

 We expand the condition $(a,b)=1$ using the M\"obius function and use triangle inequality to get (note that  $(b+ia_0,u)=1$ implies that $(c,q)=1$)
\begin{align*}
 \sum_{(c,q)=1} \sum_{\substack{b \in B  \\ b \equiv 0\, (c)}} \bigg| \sum_{\substack{|w|^2 \sim D \\ (2bu,w)=1}} \alpha_w  \sum_{\substack{\nu \, (d) \\ \nu \equiv -i \, (w)}} \bigg(\sum_{\substack{a \equiv \nu b \, (d) \\ a \equiv  a_0 \theta_{b,q} d \, (4q)\\ a \equiv 0 \, (c)}} F_A(a)  &- g^{(b)}(w)  \sum_{\substack{ a \equiv a_0 \theta_{b,q} d \, (4q)\\ a \equiv 0 \, (c)}}  F_A(a)\bigg) \bigg|  \\
 =: S_{\leq \eta} + S_{> \eta},
\end{align*}
where we have partitioned the sum depending on whether $c \leq X^{\eta/20}$ or $c > X^{\eta/20}$. The second sum can be bounded trivially by
\begin{align*}
\sum_{a_0 \, (4q)} S_{> \eta} &\ll  \sum_{c > X^{\eta/20}} \sum_{\substack{b \in B  \\ b \equiv 0\, (c)}}  \sum_{\substack{d \sim D \\ (2bq,d)=1}} \bigg( \sum_{\nu^2+1 \equiv 0 \, (d)} \sum_{\substack{a \equiv \nu b \, (d) \\ a \equiv 0\, (c)}} F_A(a) + g^{(b)}(d)  \sum_{\substack{   a \equiv 0\, (c)}}  F_A(a) \bigg) \\
&\ll \sum_{c > X^{\eta/20}} \sum_{\substack{b \in B  \\ b \equiv 0\, (c)}}  \sum_{\substack{d \sim D\\ (2bq,d)=1}}\bigg( \sum_{\substack{a^2+b^2 \equiv 0  \, (d)\\ a \equiv 0\, (c)}} F_A(a) + g^{(b)}(d)  \sum_{\substack{    a \equiv 0\, (c)}}  F_A(a) \bigg) \\
&\ll  \sum_{c > X^{\eta/20}} \sum_{\substack{b \in B  \\ b \equiv 0\, (c)}} \bigg( \sum_{\substack{  a \equiv 0\, (c)}} F_A(a) \tau(a^2+b^2) + (\log X)^{O(1)}\sum_{\substack{  a \equiv 0\, (c)}}  F_A(a) \bigg)\bigg).
\end{align*}
By the divisor bound for $\tau(a^2+b^2)$ and $\tau(b)$ we obtain
\[
\sum_{a_0 \, (4q)}S_{> \eta} \ll   \sum_{\substack{b \in B }} \sum_{\substack{c|b  \\ c > X^{\eta/20}}}
X^{\eta/80}\bigg(\frac{X^{1/2} }{c} +1\bigg) \ll X^{1/2-\eta/20+\eta/80}  \sum_{\substack{b \in B }} \tau(b) \ll  X^{1/2-\eta/40}|B| .
\] 

Hence, it remains show that for any $a_0$
\[
S_{\leq \eta}  \ll  \frac{X^{1/2-\eta/40}|B|}{q}.
\]
By writing
\[
S_{\leq \eta}  = \sum_{c \leq X^{\eta/20} } S_c
\]
it suffices to show that for every $c \leq X^{\eta/20}$
\[
S_c \ll \frac{X^{1/2-\eta/10}|B|}{q}
\]
Applying Poisson summation (Lemma  \ref{poissonlemma}, note that $d,q,c$ are all pairwise coprime) we get
\[
S_{c}  \leq  T_c(\alpha,A) + X^{\eta/10 }U_c(\alpha,A) + O_\eps(X^{-100}),
\]
where for 
\begin{align} \label{Hdef}
  H:= \frac{qDX^{\eta/10 }}{X^{1/2}} 
\end{align}
we have (denoting $\alpha'_w := \alpha_w\frac{D}{|w|^2} $)
\begin{align*}
T_c(\alpha,A) &:=   \hat{F}(0) A \sum_{\substack{b \in B  \\ b \equiv 0\, (c)}} \bigg| \sum_{\substack{|w|^2 \sim D \\ (2bu,w)=1}} \alpha_w \bigg(\frac{1}{cdq}-  \frac{g^{(b)}(w)}{cq}\bigg) \bigg| \\
U_c(\alpha,A) &:= \frac{1}{H} \sum_{0 < |h| \leq H} \sum_{\substack{b \in B  \\ b \equiv 0\, (c)}} \bigg|   \sum_{\substack{ |w|^2\sim D\\ (2bu,w)=1}} \alpha'_w  \widehat{F} (hA/cdq) \sum_{\substack{\nu \, (d) \\ \nu \equiv -i \, (w)}} e_d(hb\overline{4cq} \nu )e_{4q}(a_0 d \theta_{b,q} h\overline{cd})\bigg|.
\end{align*}
Note that the phase $e_{4q}(a_0 d \theta_{b,q} h\overline{cd}) = e_{4q}(a_0  \theta_{b,q} h\overline{c})$ does not depend on $d$ or  $\nu$ and thus it may be replaced by $1$.
 
 By definition $g^{(b)}(w)$ matches with $1/d$ when $(b,w)=1$, so that in fact $T(\alpha,A) = 0$. Thus, it suffices to show that 
 \begin{align} \label{Uiclaim}
     U_c(\alpha,A) \ll  \frac{X^{1/2-\eta/5}|B|}{q}.
 \end{align}

\subsubsection{Bounding $U(\alpha,A)$}
The smooth cross-condition $\widehat{F} (hA/cdq)$ may be removed by using the Mellin inversion formula (Lemma \ref{improvedmvtlemma}) at a cost of  a factor $  \ll X^{\eta/40}$. 
We expand the condition $(b,d)=1$ using the M\"obius function to get
\[
\mathbf{1}_{(b,d)=1} = \sum_{e} \mu(e) \mathbf{1}_{e|b} \mathbf{1}_{e|d}.
\]
For each $c$ and $e$ let
\begin{align*}
    V_{c,e}(\alpha,A) :=  \sum_{\substack{ |w|^2 \sim D/e \\ (d,2cq)=1}} \sum_{\substack{\nu \, (d) \\ \nu \equiv -i \, (w)}} \bigg| \sum_{n \ll HX^{1/2}/(ce)} \gamma_n^{(ce)}  e_d(\overline{4cq} hb\nu )\bigg|
 \\
 =\sum_{\substack{ d \sim D/e \\ (d,2cq)=1}}  \sum_{\nu^2+1 \equiv 0 \, (d)} \bigg| \sum_{n \ll HX^{1/2}/(ce)} \gamma_n^{(ce)}  e_d(\overline{4cq}hb\nu )\bigg|
\end{align*}
for some coefficients
\[
|\gamma_n^{(ce)}| \leq \frac{1}{H } \sum_{n=hb}  \mathbf{1}_B(bce).
\]
Then by rearranging the sums we have
\[
U_c(\alpha,A) \ll X^{\eta/40} \sum_{e} V_{c,e}(\alpha,A)
\]
and for \eqref{Uiclaim} it suffices to show that
\[
V_{c,e}(\alpha,A) \ll \frac{X^{1/2-\eta/4}|B|}{e q}
\]
By a divisor bound
\[
 \sum_{n\ll HX^{1/2}/(ce)}|\gamma_n^{(ce)}|^2\ll X^{\eta/100}\frac{1}{H} \sum_{n\ll HX^{1/2}/(ce)} |\gamma_n^{(ce)}| \ll  X^{\eta/100} \frac{1}{H} |B|.
\]
By  Cauchy-Schwartz and Lemma \ref{quadraticlargesievelemma}, using (\ref{Hdef}), we have
\begin{align*}
 V_{c,e}(\alpha,A)  &\ll \bigg(\frac{ D}{e}\bigg)^{1/2} \bigg(  \sum_{\substack{ d \sim D/e \\ (d,2cq)=1}}  \sum_{\nu^2+1 \equiv 0 \, (d)} \bigg|  \sum_{n \ll HX^{1/2}/(ce)} \gamma_n^{(ce)}  e_d(hb\overline{4cq}\nu )\bigg|^2\bigg)^{1/2} \\ 
 &\ll X^{\eta/100}  \frac{ D^{1/2}}{eH^{1/2}}  ( qD + HX^{1/2})^{1/2} |B|^{1/2}  \\
 &\ll e^{-1}D^{1/2}X^{1/4+\eta/100} |B|^{1/2} \\
 & \ll  \frac{X^{1-\delta-\eta/4}}{e q}
\end{align*}
by using
\[
D \leq X^{1-\delta-\eta}/q^2. \qedhere
\]
\end{proof}
\section{Proof of Theorem \ref{exceptionaltheorem}} \label{exceptionalsection}
In this section we prove Theorem \ref{exceptionaltheorem} by following a similar strategy as in \cite[Chapter 24.2]{odc}. We first note that by restricting to $b \in B_1$ (recall \eqref{B1condition}) we may assume that for all $b$ we have
\begin{align} \label{positivityassumpt}
    \sum_{\substack{a^2 +b^2 \sim X \\ (a,b)=1 \\ (a^2+b^2,2)=1}} \chi_1 ((b+ia)) \geq 0.
\end{align}

We define the Dirichlet convolution on ideals of the Gaussian integers as
\[
(f \ast g)(\a) := \sum_{\mathfrak{c} \mathfrak{d} = \a} f(\mathfrak{c})g(\mathfrak{d}).
\]
We define the auxiliary function
\[
\lambda_1(\a) := (1\ast \chi_1) (\a),
\]
which assuming the existence of a Siegel zero for $L(1,\chi_1)$ is sparsely supported, precisely, for square-free $\a$ it is supported on $\a$ such that for all $\p|\a$ we have $\chi_1(\p)=+1$.

We set $\A'=(a'_{\n})$ with $a'_{\n}= a_{\n} \lambda_1(\n) \mathbf{1}_{(\n,u_1)=1}$ and define the multiplicative functions
\begin{align*}
g_1(\d) &:= \prod_{\p|\d} g_1(\p), \quad g_1(\p) := \frac{\mathbf{1}_{\p \neq \overline{\p}}}{\N \p}\bigg(1+\chi_1(\p)-\frac{\chi_1(\p) }{\N \p} \bigg) \bigg( 1-\frac{\chi_1(\p)}{\N \p}\bigg)\bigg( 1-\frac{1}{\N \p^2}\bigg)^{-1} \\
g'(d)&:= \sum_{\N \d = d} g_1(\d), \quad g'(p) = \frac{\rho(p)}{p}\bigg(1+\chi_1(\p)- \frac{\chi_1(\p) }{p} \bigg) \bigg( 1-\frac{\chi_1(\p)}{p}\bigg)\bigg( 1-\frac{1}{p^2}\bigg)^{-1}    
\end{align*}
denoting $p= \p \overline{\p}$, which is well defined since $\rho(p) \neq 0$ precisely if $p$ splits in $\ZZ[i]$ and $\chi_1(\p) = \chi_1(\overline{\p})$ since $\chi_1$ is real. 
We also set
\begin{align*}
V'(Z) &:= \prod_{2 < p < Z} \bigg(1-g'(p) \bigg) \\
\YY &:=  L(1,\chi_1)\sum_{\substack{ a^2+b^2 \sim X \\ b\in B \\  (a,b)=1 \\ (b+ia,2u_1)=1}}  (1 +  \chi_1((b+ia))\prod_{\p|b} \bigg( 1-\frac{\chi_1(\p) }{\N \p}\bigg) \prod_{\substack{p|b \\ p\neq 2}}\bigg(1-g'(p) \bigg)^{-1}\bigg(1-\frac{1}{p^2} \bigg).
\end{align*} 
Note that by \eqref{positivityassumpt} the contribution from the terms with $\chi_1((b+ia))$ is positive and therefore we may drop it to conclude
\[
\YY \gg L(1,\chi_1)\sum_{\substack{ a^2+b^2 \sim X\\ b\in B \\  (a,b)=1 \\ (b+ia,2u_1)=1}}  \prod_{p|b} \bigg( 1-\frac{\chi_1(\p) }{\N \p}\bigg) \prod_{\substack{p|b \\ p\neq 2}}\bigg(1-g'(p) \bigg)^{-1}\bigg(1-\frac{1}{p^2} \bigg).
\]
We define
\[
S(\A',Z) := \sum_{\substack{\N \n \sim X \\ (\N\n,P(Z))=1}} a'_\n.
\]
We note for large primes $p$ the function $g'(p)$ is essentially $p^{-1}\rho(p)(1+\chi_1(\p))$. 

Theorem \ref{exceptionaltheorem} is then a direct corollary of the following and the lower bound $L(1,\chi_1) \gg_\eps |u_1|^{-\eps}$. Note that from the assumption $\Omega(B_1) \geq  \Omega(B)/2$ it follows that $|B_1| \, \gg_\eps |B| X^{-\eps}$ since $\omega(b) = X^{\pm o(1)}$

\begin{prop} \label{exceptionalprop} Suppose that the assumptions of Theorem \ref{exceptionaltheorem} hold. Then for $Z = |u_1|^4$ we have
\[
S(\A',2X^{1/2}) =  V'(Z) \YY  \left(1+ O\left( \frac{\log Z}{\log X}  + \delta(Z,X) \right)  \right),
\]
where
\begin{align} \label{lambda1bound}
\delta(Z,X) = \sum_{Z \leq \N \p < 2X^{1/2}} \frac{\lambda_1(\p)}{\N \p} < (1-\beta) (\log X + O(\log Z)).
\end{align}
\end{prop}
\begin{remark}
Note that by introducing the weight $\lambda_1(\n)$ we have already removed all prime factors of $\n$ with $\chi_1(\p) =-1$, which by the Siegel zero assumption is most of the primes if $\eps_1$ is small. Thus, we are in a situation of a low dimensional sieve as we only need to shift out prime divisors with $\chi_1(\p)=+1$.
\end{remark}
We first gather Type I information. Note that $(a,b)=1$ implies that $\d$ is not divisible by a rational prime.
\begin{lemma} \label{exceptionaltypeilemma}
Denote $q=\M(u_1)$. Let $\alpha_\d$ be divisor-bounded coefficients supported on square-free $\d$ with $(\d,\overline{\d})=1$. Let $g_1^{(b)}(\d),\YY^{(b)}$ be defined similarly as in Section \ref{typeisection}. Then 
\[
\sum_{\N \d \leq X^{1-2\delta-4\eta}/q^{4}} \alpha_\d \bigg(\sum_{\substack{\N \n \sim X \\ \n \equiv 0 \, (\d)}} a'_{\n}  - \sum_{b \in B} g_1^{(b)}(\d) \YY^{(b)}\bigg)  \ll X^{1-\delta-\eta/100}.
\]
\end{lemma}
\begin{proof}
The function $\lambda_1(\n)$ is multiplicative and we have similarly to \cite[(24.5)]{odc} (denoting $\n=\d \m$)
\[
\lambda_1(\d \m) = \sum_{\c|(\d,\m)} \mu(\c)\chi_1(\c) \lambda_1(\d/\c) \lambda_1(\m/\c).
\]
We have
\[
\lambda_1(\m/ \c) = \sum_{\a \b  = \m/\c} \chi_1(\a) = (1+\chi_1(\m/\c))\sum_{\substack{\a \b = \m/\c \\ \N \a < \N\b}} \chi_1(\a) + O(\mathbf{1}_{\N \m/\c=\square}).
\]
The contribution from the error term where $\N \m/\c=\square$ may be bounded by crude bounds. We have
\[
\chi_1(\m/\c) = \chi_1(\c \d) \chi_1(\n).
\]
The contribution from $\N \c > X^{\eta}$ can be bounded by crude bounds. We obtain
\begin{align*}
    \sum_{\N \d \leq X^{1-2\delta-4\eta}/q^{4}} &\alpha_\d \sum_{\substack{\N \n \sim X \\ \n \equiv 0 \, (\d)}} a'_{\n}  +  O(X^{1-\delta-\eta/100})\\
   &= \sum_{\N \d \leq X^{1-2\delta-4\eta}/q^{4}} \alpha_\d \sum_{\substack{\c| \d \\ \N \c \leq X^{\eta}}} \mu (\c) \chi_1(\c) \sum_{\N \c\d \a^2 \leq 2X} \chi_1(\a) \\
   & \hspace{70pt} \times \sum_{\substack{\N \n \sim X \\ \N \n >  \N \c \d \a^2 \\ \n \equiv 0 \, (\c \d \a) }} a_{\n}\mathbf{1}_{(\n,u_1)=1}(1+\chi_1(\c\d)\chi_1(\n))  
\end{align*}
We now relax the cross-conditions  $\N \n > \N \c\d \a^2$ and $\N \n \sim X$ by introducing a finer-than-dyadic decomposition for $\N \n$ (using Section \ref{smoothweightsection} with $\nu=X^{-\eta/1000}$) to get for $X' \sim X$ sums of the type
\begin{align*}
    \sum_{\N \d \leq X^{1-2\delta-4\eta}/q^{4}} \alpha_\d \sum_{\substack{\c| \d \\ \N \c \leq X^{\eta}}} \mu (\c) \chi_1(\c) \sum_{\N \c\d \a^2 \leq X'} \chi_1(\a) \hspace{70pt} \\
    \times \sum_{\substack{\n \equiv 0 \, (\c \d \a) }} a_{\n}F_{X'}(\N\n)\mathbf{1}_{(\n,u_1)=1}(1+\chi_1(\c\d)\chi_1(\n)) \\
    =: S_1+S_2
\end{align*}
where 
\begin{align*}
   S_1 &:=  \sum_{\substack{\N \f \leq  X^{1-\delta-\eta}/q^2 \\ (\f,\overline{\f})=1}} \alpha_1(\f) \sum_{\substack{ \n \equiv 0 \, (\f) }} a_{\n}F_{X'}(\N\n)\mathbf{1}_{(\n,u_1)=1}, \\
   S_{2} &:= \sum_{\substack{\N \f \leq  X^{1-\delta-\eta}/q^2 \\ (\f,\overline{\f})=1}} \alpha_2(\f) \sum_{\substack{ \n \equiv 0 \, (\f) }} a_{\n} F_{X'}(\N\n)\chi_1(\n),
\end{align*}
with
\begin{align*}
    \alpha_1(\f) &:= \sum_{\substack{\f = \c \d \a \\ \N \d \leq X^{1-2\delta-4\eta}/q^{4}\\ \c|\d\\ \N \c \leq X^\eta \\ \N \c \d \a^2 \leq  X'}} \alpha_\d \lambda_1(\d/\c) \mu(\c) \chi_1(\c) \chi_1(\a) \\
    \alpha_2(\f) &:= \sum_{\substack{\f = \c\d\a \\\N \d \leq X^{1-2\delta-4\eta}/q^{4} \\ \c|\d \\ \N \c \leq X^\eta \\ \N \c\d\a^2 \leq  X'}} \alpha_\d \chi_1(\d) \lambda_1(\d/\c) \mu(\c)  \chi_1(\a) 
\end{align*}
Note that $\N \d \leq X^{1-2\delta-4\eta}/q^{4}$, $\N \c \leq X^\eta $, and $\N \c \d \a^2 \leq  X'$ imply that $\N \c\d\a \leq X^{1-\delta-\eta}/q^2 $. If $\n=(z)$, then in $S_2$ the character value
\[
\chi_1(\n) = \chi_1((z))
\]
depends only on the residue class of $z$ modulo $4(1+i) u_1$ (note that $(a,b)=1$ implies $2 \nmid z$). Thus, by Proposition \ref{typeiprop} we get
\begin{align*}
S_1 =\sum_{b \in B}\sum_{\substack{\N \f \leq X^{1-\delta-\eta}/q^2 \\ (\f,u_1 \overline{\f})=1 }} \alpha_1(\f)   g^{(b)}(\f) \XX_1^{(b)}  + O( X^{1/2-\eta/100}|B|) \\
S_2 =\sum_{b \in B}\sum_{\substack{\N \f \leq X^{1-\delta-\eta}/q^2 \\ (\f,u_1 \overline{\f})=1}} \alpha_2(\f)   g^{(b)}(\f) \XX_{2}^{(b)} + O( X^{1/2-\eta/100}|B|) ,
\end{align*}
with
\begin{align*}
    \XX_1^{(b)}  &:=  \sum_{\substack{  (a,b)=1}}  F_{X'}(a^2+b^2) \mathbf{1}_{(b+ia,2u_1)=1}\\
    \XX_2^{(b)}  &:=  \sum_{\substack{ (a,b)=1}} F_{X'}(a^2+b^2) \chi_1( (b+ia) )\mathbf{1}_{(b+ia,2)=1}.
\end{align*}
Note that we have picked up the condition $(\f,\overline{\f})$ from $(a,b)=1$ implicit in $a_\n$.
We have
\begin{align*}
 \sum_{\substack{\N \f \leq X^{1-\delta-2\eta}/q^2 \\ (\f,u_1\overline{\f})=1  }} \alpha_1(\f)   g^{(b)}(\f) = \sum_{\N \d \leq X^{1-2\delta-4\eta}/q^{4}} \alpha_\d \sum_{\substack{\c| \d \\ \N \c \leq X^\eta}} \lambda_1(\d/\c)\mu(\c) \chi_1(\c)   \\
 \times\sum_{\substack{\N\a \leq (X'/ \N \c\d)^{1/2} \\ (\a,\overline{\a \d})=1}} \chi_1(\a) g^{(b)}(\a)  
\end{align*}
and
\begin{align*}
   \sum_{\substack{\N \f \leq X^{1-\delta-2\eta}/q^2\\ (\f,u_1\overline{\f})=1 }} \alpha_2(\f)   g^{(b)}(\f) = \sum_{\N \d \leq X^{1-2\delta-4\eta}/q^{4}} \alpha_\d \chi_1(\d) \sum_{\substack{\c| \d \\ \N \c \leq X^\eta}}  \lambda_1(\d/\c) \mu(\c)  \\
   \times \sum_{\substack{\N\a \leq (X'/ \N \c\d)^{1/2} \\ (\a,\overline{\a \d})=1}}\chi_1(\a) g^{(b)}(\a)  
\end{align*}
To evaluate the sum over $\a$,  we have (by applying a M\"obius expansion to $(\a,\overline{\a})=1$)
\begin{align*}
   &\sum_{\substack{\N\a \leq (X'/ \N \c\d)^{1/2} \\ (\a,\overline{\a \d})=1}} \chi_1(\a) g^{(b)}(\a) 
  = \sum_{\substack{\N \a \leq (X'/\N \c\d)^{1/2} \\ (\a,\overline{\a\d}b)=1}} \chi_1(\a) \frac{1}{\N \a} \\
& = \sum_{(k,\overline{\d} b)=1} \frac{\mu(k)}{k^2}\sum_{\substack{k^2\N\a \leq (X'/\N \c\d)^{1/2} \\ (\a,\overline{\d}b)=1}} \chi_1(\a) \frac{1}{\N \a} \\
  & =  \prod_{(p, b\overline{\d})=1} \bigg( 1-\frac{1 }{p^2}\bigg)\prod_{(\p, b\overline{\d})=1} \bigg( 1-\frac{\chi_1(\p) }{\N \p}\bigg)^{-1} + O(X^{-\eta}) \\
  & = \frac{L(1,\chi_1)}{\zeta(2)} \prod_{p|b} \bigg( 1-\frac{1}{p^2}\bigg)^{-1} \prod_{\p|b} \bigg( 1-\frac{\chi_1(\p)}{\N \p}\bigg)  \prod_{\p| \d} \bigg( 1-\frac{\chi_1(\p)}{\N \p}\bigg) \bigg( 1-\frac{1}{\N \p^2}\bigg)^{-1} + O(X^{-\eta}) 
\end{align*}
since $\chi_1(\overline{\p}) = \chi_1(\p)$ and $(\d,\overline{\d})=1$.
Let us denote (noting that the contribution from $ \N \c > X^\eta $ may now be added back in at a negligible cost)
\begin{align*}
    g_1(\d)&:= \frac{1}{\N \d}\prod_{\p|\d} \bigg( 1-\frac{\chi_1(\p)}{\N \p}\bigg)\bigg( 1-\frac{1}{\N \p^2}\bigg)^{-1}\sum_{\substack{\c|\d }}  \lambda_1(\d/\c) \mu(\c) \chi_1(\c)\frac{1}{ \N \c}  \\
    g_2(\d)  &:= \frac{\chi_1(\d)}{\N \d} \prod_{\p| \d} \bigg( 1-\frac{\chi_1(\p)}{\N \p}\bigg)\bigg( 1-\frac{1}{\N \p^2}\bigg)^{-1}\sum_{\substack{\c|\d  }}  \lambda_1(\d/\c) \mu(\c)  \frac{1}{\N \c}
\end{align*}
The functions $g_j$ are multiplicative with
\begin{align*}
    g_1(\p)&= \frac{1}{ \N \p} \bigg(1+\chi_1(\p)- \frac{\chi_1(\p)}{\N \p} \bigg) \bigg( 1-\frac{\chi_1(\p)}{\N \p}\bigg)\bigg( 1-\frac{1}{\N \p^2}\bigg)^{-1}\\
    g_2(\p)  &=  \frac{\chi_1(\p)}{\N \p} \bigg(1+\chi_1(\p)-\frac{1}{\N \p} \bigg) \bigg( 1-\frac{\chi_1(\p)}{\N \p}\bigg)\bigg( 1-\frac{1}{\N \p^2}\bigg)^{-1}= g_1(\p),
\end{align*} 
where the last equality holds since $\chi_1(\p)^2=1$. Hence, we have
\[
S_1+S_2 =  \sum_{b \in B} \sum_{\N \d \leq X^{1-2\delta-4\eta}/q^{4}} \alpha_\d  g_1^{(b)}(\d) \YY^{(b)} + O(X^{-\eta})
\]
with
\begin{align*}
   \YY^{(b)}  =\frac{L(1,\chi_1)}{\zeta(2)}\sum_{\substack{ a \\  (a,b)=1 \\ (b+ia,2u_1)=1}} F_{X'}(a^2+b^2) (1 + & \chi_1((b+ia)) )\prod_{\substack{p|b \\ p\neq 2}} \bigg( 1-\frac{1}{p^2}\bigg)^{-1} \\
   &\times \prod_{\substack{\p|b \\ p\neq 2}}\bigg( 1-\frac{\chi_1(\p)}{\N \p}\bigg) .  \qedhere
\end{align*}

\end{proof}
\subsection{Proof of Proposition \ref{exceptionalprop}}
 We apply a sieve argument to the sequence $\A''= (a''_n)$ over integers defined by
\[
a''_n := \sum_{\N \n = n} a'_\n.
\]
Note that then
\[
S(\A',2X^{1/2}) = S(\A'',2X^{1/2}).
\]
The proof is essentially the same as in \cite[Proof of Proposition 24.1]{odc}, but we give it for completeness as it is short. Let $D:=X^{1-2\delta-4\eta}/q^4$ and denote $s= \log D / \log Z$. By Buchstab's identity we have
\[
S(\A'',2X^{1/2})  = S(\A'',Z) - \sum_{Z \leq p< 2X^{1/2}} S(\A''_p,p) 
\]

By the Fundamental lemma of the sieve (Lemma \ref{flsievelemma}, see also Remark \ref{flsieveremark}) and Lemma \ref{exceptionaltypeilemma} we have
\[
S(\A'',Z)  =  V'(Z) \YY (1+ O( e^{-s})) + O(X^{1-\delta-\eta/100}),
\]
where
\[
e^{-s} \ll s^{-1} \ll \frac{\log Z}{\log X}.
\]
By an upper bound sieve (eg. using Lemma \ref{flsievelemma}) with level $D/p> X^{1/4}$ and Lemma \ref{exceptionaltypeilemma} we have
\[
\sum_{Z \leq   p < 2X^{1/2}} S(\A''_p,p) \leq \sum_{Z \leq \N p < 2X^{1/2}} S(\A''_p,Z) \ll  \delta(Z,X) V'(Z) \YY  +O(X^{1-\delta-\eta/100}).
\]
Combining the two estimates we get Proposition \ref{exceptionalprop}, noting that the bound (\ref{lambda1bound}) can be proved by a similar argument as in \cite[(24.20)]{odc}.

\begin{remark} \label{epsremark}
Instead of taking a small $Z$ and using the Fundamental lemma of the sieve, we can take a larger $Z$ and use the linear sieve lower and upper bounds \cite[Theorem 11.12]{odc}. Taking $Z=X^{1/4}$ we get that the lower bound for $S(\A'',2X^{1/2})$ is proportional to (denoting the linear sieve functions by $f,F$)
\[
4 f(4(1-O(\delta))) - \eps_1 
 4 \int_{1/4}^{1/2} \frac{F(4(1-\alpha - O(\delta)))}{\alpha} d\alpha + O(\delta) = (1-\eps_1) 2 \log (3) + O(\delta).
\]
To make the right-hand side positive we can certainly take  any $\eps_1 < 1$ and $\delta = (1- \eps_1)/100$, for instance.
\end{remark}
\section{Type II information: preliminaries} 
\label{typeiipreliminarysection}
We now fix small parameters $\eta,\eta_1,\eta_2,\eta_3 > 0$ such that $\eta_1$ is small compared to $\eta$, $\eta_2$ is small compared to $\eta_1,$ and $\eta_3$ is small compared to $\eta_2$, that is,
\[
\eta_3 \ll \eta_2 \ll \eta_1 \ll \eta.
\]
For instance, for our purposes it would suffice to take any small $\eta > 0$ and let
\[
\eta_1 = \eta/100, \quad \eta_2 = \eta_1/2, \quad \eta_3 = \eta_2/100.
\]
We denote  
\[
\nu_j = X^{-\eta_j}, \quad j \in \{1,2,3\}.
\] 
We will often refer to these parameters in the course of the following sections. 

Let $G: \RR/2 \pi \ZZ \to \CC$ be a non-negative $C^\infty$-smooth function supported on $[-\nu_1,\nu_1]$ as in Lemma \ref{fourierserieslemma}, satisfying
\[
|G^{(j)}| \ll_j \nu_1^{-j},\, \, j \geq 0 \quad \text{and} \quad \int G(\theta) d \theta = \nu_1.
\]  
and recall that
\[
G(\arg z) = \sum_{k} \check{G}(k) \xi_k (z).
\]

For a set of Hecke characters $\Psi=\{ \xi \chi \}$ and $u \in \ZZ[i]\setminus \{0\}$ we let $\Psi_u$ denote the set of characters induced to modulus $u$.  We say that $\xi \chi$ is primitive if $\chi$ is a primitive Dirichlet character. The following definition depends on the choice of $\eta,\eta_j$ but since these are fixed throughout the argument we omit this dependency in  the notation.
\begin{definition}\emph{($Q$-regularity).}
Let $C_1,C_2>0$ and $Q,N,X \geq 1$. Let $\beta_\n$ be complex coefficients.
We say that $\beta_\n$ is $(Q,N,X,C_1,C_2)$-\emph{regular} if there exists a set of primitive characters $\Psi=\{\xi\chi\}$ with $|\Psi| \leq (\log X)^{C_2}$ such that the following holds. For any $\xi \chi  \in \Psi$ we have
\[
\M(\emph{\cond}(  \chi))  \leq Q \quad \text{and} \quad \xi=\xi_k,\, |k| \leq \nu_1^{-2}
\]
and for any $u \in \ZZ[i]$ with $\M(u) \leq Q$ and $N' \sim N$ we have
\begin{align*} 
  \frac{1}{\varphi_{\ZZ[i]}(u)}  \sum_{\substack{\psi \in \widehat{(\ZZ[i]/u\ZZ[i])^\times}  }} \bigg|\sum_{\substack{k \\ \xi_k \psi  \not \in \Psi_u}} \check{G}(k)  \bigg( \sum_{\substack{\N \n \in [N',N'(1+\nu_2)]  }} \beta_\n  \xi_k \psi(\n)\bigg) \bigg|^2\leq \frac{ \nu_1^2 \nu_2^2 N^2}{\varphi_{\ZZ[i]}(d) (\log X)^{C_1}}.  
\end{align*}

Given a function $Q=Q(N,X)$, we say that $\beta_{\n}$ is $Q$-\emph{regular} if for any $C_1>0$ there is some $C_2>0$ and some $X_0 >0$  such that for all  $X \geq X_0$ and for all $N \geq X^\eta$  the coefficient $\beta_{\n}$ is $(Q,N,X,C_1,C_2)$-regular.
\end{definition}
Informally speaking, coefficient $\beta_\n$ is $Q$-regular if it is equidistributed in residue classes and polar boxes apart from a set of Hecke characters of size $ \leq (\log X)^{O(1)}$, uniformly in the size of the modulus of the $\M(u)  \leq Q $. We will need the fact that the M\"obius function restricted to rough numbers is $Q$-regular for $Q$ close to $N^{1/3}$, which we will prove in Section \ref{mobiusregularitylemmasection} using the zero density estimate (Lemma \ref{zerodensitylemma}). 
\begin{lemma} \label{mobiusregularlemma}
Let $W:= X^{1/(\log\log X)^2}$. The coefficient
\[
\beta_\n :=\mu(\N\n)   \mathbf{1}_{(\n,P(W))=1}
\]
is $Q$-regular for any $Q=Q(N,X)$ with $N \geq X^\eta Q^{3}$. Furthermore, for fixed $Q,X,C_1>0$, a fixed set of characters $\Psi$  with $C_2 \ll C_1$ works for all ranges of $N > X^\eta Q^{3}$ in the definition of $(Q,N,X,C_1,C_2)$-regularity.
\end{lemma}
 The exponent $3$  in $N > X^\eta Q^{3}$ is not the best that can be obtained but it suffices for our purposes. This could be improved by using results on large values of Dirichlet polynomials (analogous to \cite{Huxley}). We have Type II information given by the following proposition. We will apply it with $\beta_\n$ as above but it applies equally well to eg. products of $k$ primes.
\begin{prop} \label{typeiiprop}\emph{(Type II information).}
Let $W:= X^{1/(\log\log X)^2}$ and let $\nu= (\log X)^{-C}$ for some $C>0$. For every $C_1>0$  there is some $C_2 \ll_{C_1} 1$ such that the following holds. Let $MN= X' \sim X$ with
\[
X^{3\delta+4\eta} < N < X^{1/2-\delta-\eta}.
\]
 Let $\alpha_\m,\beta_\n$ be bounded coefficients, supported on $(\m\n,P(W))=1$ and 
\[
\N \m \in [M,M(1+\nu)], \quad  \N \n \in [N,N(1+\nu)].
\]
Let $F$ be a smooth function as in Section \ref{smoothweightsection} with the parameter $\nu$.
Suppose that for $Q=X^{\delta + \eta}$ the coefficient $\beta_\n$ is $Q$-regular and let $\xi_{k_j} \chi_j$ denote the corresponding Hecke characters with $j \leq J\leq (\log X)^{C_2}$.  Then
\[
\sum_{\m,\n } \alpha_\m \beta_\n a_{\m\n} =   \sum_{j \leq J } \sum_{\m,\n } \frac{\alpha_\m \beta_\n \xi_{k_j}\chi_j(\m\n)}{\N \m\n}   \sum_{\a} \frac{F(\N\m\n/\N \a)}{\widehat{F}(0) } a^\omega_{\a} \overline{\xi_{k_j}\chi_j}(\a) + O \bigg( \frac{X^{1/2}|B|}{(\log X)^{C_1}}\bigg)
\]
\end{prop}
\begin{remark}
    Note that on the right-hand side of Proposition \ref{typeiiprop} we have the sequence $a_\a^\omega$ which is twisted by $\omega$, which arises from the assumption $(\m\n,P(W))=1$. In particular, Proposition \ref{typeiiprop} as stated would be false without restricting to $(\m\n,P(W))=1$.
\end{remark}

We set
\[
\alpha_w:=\alpha_{(\overline{w})} \quad \text{and} \quad \beta_z := \beta_{(z)} \mathbf{1}_{z \, \text{primary}}.
    \]
so that in the latter we may swap freely between Gaussian integers $z$ and ideals $\n$ and that
\[
\sum_{\m,\n } \alpha_\m \beta_\n a_{\m\n} = \sum_{w,z } \alpha_w \beta_z a_{\overline{w}z} .
\]
We have included the complex conjugate in $\overline{w}$ to make our notations match with those of \cite{FI}.

\subsection{Sketch of the argument} \label{typeiisketch}
As the proof of Proposition \ref{typeiiprop} is quite technical, we include here a simplified non-rigorous sketch. In Section \ref{approximationsection} we will construct an approximation $\beta^\#_\n$ for $\beta$ such that the difference
\[
\beta_\n^{\flat} := \beta_\n - \beta^\#_\n
\]
is balanced along arithmetic progressions. Write
\[
S(\alpha,\beta) = S(\alpha,\beta^\#)+S(\alpha,\beta^\flat).
\]
The approximation will be simple enough that $S(\alpha,\beta^\#)$ can be evaluated by Type I information, so let us consider 
\[
S(\alpha,\beta^\flat) = \sum_{w}  \sum_z \alpha_w \beta^\flat_z \mathbf{1}_{B}(\Re(\overline{w}z)),
\]
where the aim is to capture the oscillations from $\beta^\flat_z$. By applying Cauchy-Schwarz we get
\[
S(\alpha,\beta^\flat) \ll M^{1/2} U^{1/2}
\]
with
\[
U:= \sum_{z_1,z_2} \beta^\flat_{z_1} \overline{\beta^\flat_{z_2}} \sum_{w} \mathbf{1}_{B}(\Re(\overline{w}z_1)) \mathbf{1}_{B}(\Re(\overline{w}z_2)) F_M(|w|^2)
\]
for some smooth majorant $F_M$ of the interval $[M,2M]$.
The goal is to evaluate the sum over $w$ with a main term $M(z_1,z_2)$ and then show that
\[
\sum_{z_1,z_2} \beta^\flat_{z_1} \overline{\beta^\flat_{z_2}} M(z_1,z_2) \ll \frac{N|B|^2}{(\log X)^C}
\]
is small due to cancellations from the coefficients $\beta^\flat_z$.

Similarly as in \cite{FI}, we note that by denoting $b_j = \Re(\overline{w}z_j)$ and
\[
\Delta = \Im (\overline{z_1} z_2), \quad a \equiv z_2/z_1 \, (\Delta),
\]
we have
\begin{equation} \label{deltaskectheq}
  i \Delta w =z_2 b_1- z_1b_2.  
\end{equation}
Note that typically $|\Delta| \approx N$. The parts where $\Delta=0$ or $|\Delta| < N/ (\log X)^C$ correspond to diagonal contributions and may be bounded by crude estimates. Thus, we assume for simplicity that $|\Delta| \asymp N$. 

Let $b_0:=(b_1,b_2)$ and write $b_j=b_0b_j'$. Note that in the situation that $B \subseteq q_1 \ZZ$ we have $q_1 | b_0$, so that $b_0$ can be quite large for a large subset of $B\times B$. As usual, in most places dealing with greatest common divisors does not cause serious problems but the dependency on $b_0$ will be crucial for our argument. 

We have $b_0 | \Delta w$ by \eqref{deltaskectheq} and for simplicity let us assume that $b_0| \Delta$.
Then $w = (z_2 b_1- z_1b_2)/(i\Delta)$ is fixed once we fix $z_j,b_j$, so that (ingoring the smooth weight $F_M$) we have to bound
\[
V= \sum_{b_0 }\sum_{\substack{z_1,z_2 \\ |\Delta|  \asymp  N \\ b_0| \Delta}} \beta^\flat_{z_1} \overline{\beta^\flat_{z_2}} \sum_{b'_2 \equiv a b'_1 \, (\Delta/b_0)}  \mathbf{1}_B(b_0b_1') \mathbf{1}_B(b_0b_2').
\]
Note that
\[
a \equiv z_2/z_1 \equiv \frac{\Re (\overline{z_1}z_2)}{|z_1|^2} \, (\Delta)
\]
is congruent to an integer, so that the congruence $b_2' \equiv a b_1' \, (\Delta/b_0)$ lives in $\ZZ/(\Delta/b_0) \ZZ$. 

By expansion with Dirichlet characters 
 and sorting into primitive characters  we get (ignoring issues with greatest common divisors)
\[
V = \sum_{b_0 }\sum_{\substack{z_1,z_2 \\ |\Delta|  \asymp  N \\ b_0| \Delta}} \beta^\flat_{z_1} \overline{\beta^\flat_{z_2}} \frac{1}{\varphi(\Delta/b_0)}\sum_{d|\Delta/b_0} \sideset{}{^\ast}\sum_{\chi \, (d)} 
 \overline{\chi}(a)\bigg| \sum_{b_0 b \in B} \chi(b) \bigg|^2,
\]
where morally
\[
\frac{1}{\varphi(\Delta/b_0)} \approx \frac{b_0}{N}.
\]
We split this into two parts, $d b_0 > X^{\delta+\eta}$ and $db_0  \leq X^{ \delta+\eta}$.The contribution from the small $d$ is our main term $M(z_1,z_2)$ referred to in the above. 

For the large $db_0$ we get
\begin{align*}
 \sum_{b_0 } b_0 \sum_{d b_0 > X^{\delta+\eta}} &\sum_{\substack{D  \asymp  N \\ D \equiv 0\, (db_0) }}   \bigg( 
  \frac{1}{N}\sum_{\substack{z_1,z_2 \\ |\Delta| =D}} 1 \bigg)  \sideset{}{^\ast}\sum_{\chi \, (d)} \bigg| \sum_{b_0 b \in B} \chi(b) \bigg|^2  \\
  &\ll N \sum_{b_0 } \sum_{ X^{\delta+\eta} < db_0  \ll N}   \frac{1}{d} \sideset{}{^\ast}\sum_{\chi \, (d)} \bigg| \sum_{b_0 b \in B} \chi(b) \bigg|^2, 
\end{align*}
and applying the large sieve for multiplicative characters (Lemma \ref{largesievelemma}) we get
\[
\ll N (N + X^{1/2-\delta-\eta}) |B| \ll X^{-\eta} N |B|^2
\]
by using $N \ll X^{-\eta} |B|$. Note that we are applying the large sieve to a very sparse set $B/b_0 \cap \ZZ$, which causes a loss in the diagonal terms and we are forced to take $d b_0$ at least a bit bigger than $X^\delta$.

For the small $db_0$ we can rewrite the conditions $b_0 d| \Delta$, $a \equiv z_2/z_1 \, (\Delta)$ as $z_2 \equiv a z_1 \, (b_0 d)$ to get
\begin{align} \label{sketchtypeiiend}
   \frac{1}{N} \sum_{b_0 } b_0 \sum_{db_0  \leq X^{\delta+\eta}} \sum_{a \, (b_0 d)} \sideset{}{^\ast}\sum_{\chi \, (d)} 
 \overline{\chi}(a)\bigg| \sum_{b_0b \in B} \chi(b) \bigg|^2  \sum_{z_2 \equiv a z_1 \, (b_0 d)} \beta^\flat_{z_1} \overline{\beta^\flat_{z_2}}. 
\end{align}
We now see what precisely is required of the balanced function $\beta_z^\flat$, we need
\[
\sum_{z_2 \equiv a z_1 \, (b_0d)} \beta^\flat_{z_1} \overline{\beta^\flat_{z_2}} \ll \frac{N^2}{\varphi_{\ZZ[i]}(b_0 d) (\log X)^{C}},
\]
where the modulus may  be as large as $X^{\delta+\eta}$. This would follow if there were no zeros of $L(s,\chi)$ with a real part $>1-  C'\log \log X / \log X$. Since this is not known, we need to construct the approximation $\beta^\#_{z}$ in a way that takes into account these possible bad characters. By the zero density estimate (Lemma \ref{zerodensitylemma}) this means that the approximation needs to see $\ll (\log X)^{O(1)}$ of the characters. For technical reasons (due to the smooth weight $F_M$) the approximation also needs to see the distribution of $\beta_z$ with respect to $\arg z$ and $|z|^2$.  Note that in the case that  $B \subseteq q_1 \ZZ$ we have always $q_1|b_0$, where $q_1$ can be as large as $X^\delta$.

\subsection{An approximation for $\beta_\n$}\label{approximationsection}
For the approximation it is convenient to use a rough finer-than dyadic partition of unity instead of Section \ref{smoothweightsection}, so that the different parts do not overlap. Let $\nu_2= X^{-\eta_2}$, and let
\[
H_{N'}(\n) := \mathbf{1}_{(N',N'(1+\nu_2)]} (\N \n) 
\]
so that
\[
\mathbf{1}_{(N,2N]} (\N \n) = \sum_{N'= N(1+\nu_2)^{n} \in [N,2N) }  H_{N'}(\n).
\]
We can of course choose $\nu_2$ so that $2=(1+\nu_2)^k$ for some $k \asymp \nu_2^{-2}$.

 Let $\beta_\n$ be $Q$-regular and let $\Psi= \{\xi_{k_j}\chi_j\}$ denote the set of $J \leq (\log X)^{C_2}$ characters, and denote the moduli of the characters by $u_1,\dots,u_J$ and the primitive characters by $\chi_1,\dots,\chi_J$. For any two coefficients $\alpha,\beta$ we define their normalized $W$-rough correlation as
\[
\C_W(\beta,\alpha) := \bigg(\sum_{\substack{ (\n,P(W))=1 \\ (\n,\overline{\n})=1} } \beta_\n\overline{\alpha_\n} \bigg)\bigg(\sum_{\substack{(\n,P(W))=1 \\ (\n,\overline{\n})=1} }  |\alpha_\n| \bigg)^{-1},
\]
if the denominator is non-zero. We then define the approximation $\beta_\n^\#$ for $\beta_\n$
\begin{align*}
   \beta_\n^\#= \beta_\n^\#(\Psi):=  \sum_{N'= N(1+\nu_2)^{j} \in [N,2N]} H_{N'}(\n)&\mathbf{1}_{(\n,\, \overline{\n} P(W))=1} \bigg(\sum_{j \leq J} \overline{\xi_{k_j}\chi_j}(\n)\C_W(\beta, \overline{\xi_{k_j}\chi_j} H_{N'})  \bigg) 
\end{align*}
and the balanced function
\[
\beta_\n^\flat:= \beta_\n-\beta_\n^\#,
\]
so that we have a decomposition
\[
\beta_\n = \beta_\n^\#+\beta_\n^\flat.
\]
Morally the approximation $\beta^\#$ can be viewed as a kind of "expansion" with respect to a "basis", which is justified since the functions  $\xi  \chi H_{N'}$ are approximately orthogonal over $W$-rough number, as the following lemma shows. For the lemma recall that all functions of odd Gaussian integers $z$ are extended to $\n$ by considering the primary generator, for instance, we write $G(\arg \n) = G(\arg z)$ if $z$ is the primary generator of $\n$.
\begin{lemma}\label{flapproximationlemma}
Let $\psi,\chi$ be a characters to coprime moduli $u,u_1$ and let  $\xi=\xi_k$ with $|k| \ll (\nu_1)^{-2}$. Let $N > X^\eta |u|^2|u_1|$. Then for any $C > 0$ we have
\[
\sum_{\substack{\n }}  H_{N'}(\n)\mathbf{1}_{(\n,\overline{\n}P(W))=1} (\xi\psi\chi (\n) - \mathbf{1}_{\xi\psi\chi=1}) \ll_C \frac{\nu_2 N}{(\log N)^C}
\]
and for any $(\m,u)=1$
\[
\sum_{\substack{ (\n,u)=1}}H_{N'}(\n)\mathbf{1}_{(\n,\overline{\n}P(W))=1}\left( \mathbf{1}_{\n \equiv \m \, (u)} G(\arg \n) \xi\chi(\n) - \frac{\mathbf{1}_{\chi=1}\check{G}(-k)}{\varphi_{\ZZ[i]}(u)} \right) \ll_C \frac{\nu_1\nu_2 N}{\varphi_{\ZZ[i]}(u)(\log N)^C}.
\]
\end{lemma}
\begin{proof}
   We prove the second claim, the first is similar but easier. By applying Section \ref{smoothweightsection} with a smooth function $F$ with the parameter $\nu_1$ we split $\N\n$ smoothly into finer-than-dyadic intervals. The contribution from the edges of the support of $H_{N'}$ gives a negligible contribution by trivial bounds. It then suffices to show that for any $N_1 \sim N$
   \[
   \sum_{\substack{ (\n,u)=1}}F_{N_1}(\N\n)\mathbf{1}_{(\n,\overline{\n}P(W))=1}\left( \mathbf{1}_{\n \equiv \m \, (u)} G(\arg \n) \xi\chi(\n) - \frac{\mathbf{1}_{\chi=1}\check{G}(-k)}{\varphi_{\ZZ[i]}(u)} \right) \ll_C \frac{\nu_1^2 N}{\varphi_{\ZZ[i]}(u)(\log N)^C}.
   \]
   We let $z$ denote a primary generator of $\n$.  The condition $(z,\overline{z})=1$ may be dropped with a negligible error term since $z$ is supported on $(z,P(W))=1$. We write
   \[
   \mathbf{1}_{(z,P(W))=1} = \sum_{\substack{v | (z,P(W)) \\ |v|^2 \leq N^{\eta_1} }}\mu(v) +  \sum_{\substack{v | (z,P(W)) \\ |v|^2 > N^{\eta_1} }}\mu(v).
     \]
and write
\begin{align*}
    \sum_{\substack{ (\n,u)=1}} &F_{N_1}(\N \n)\mathbf{1}_{(\n,\overline{\n}P(W))=1}\left( \mathbf{1}_{\n \equiv \m \, (u)} G(\arg \n) \xi \chi(\n) - \frac{\mathbf{1}_{\chi=1}\check{G}(-k)}{\varphi_{\ZZ[i]}(u)} \right)  \\
    =& \sum_{z \equiv 1\, (2(1+i))} F_{N_1}(|z|^2)\mathbf{1}_{(z,P(W))=1}\left( \mathbf{1}_{z \equiv w\, (u)} G(\arg z) \xi\chi (z) - \frac{\mathbf{1}_{\chi=1}\check{G}(-k)}{\varphi_{\ZZ[i]}(u)} \right)  \\
    =& S_{\leq} + S_{>}.
\end{align*}
For the large $v$ we note that by $v|P(W)$ there is some factor $v_0|v$ such that $|v_0|^2 \in (X^{\eta_1}, X^{\eta_1} W]$. Thus, by $\check{G}(-k) \ll \nu_1$
\begin{align*}
    S_{>} \ll &\sum_{\substack{|v_0|^2 \in (X^{\eta_1}, X^{\eta_1} W] \\ v_0| P(W) \\ (v_0,u)=1}} \sum_{z \equiv 0\,(v_0)} \tau(z)^{O(1)} F_{N_1}(|z|^2) \mathbf{1}_{z \equiv w\, (u)} G(\arg z)  \\
    &+ \frac{\nu_1}{\varphi_{\ZZ[i]}(u)}\sum_{\substack{|v_0|^2 \in (X^{\eta_1}, X^{\eta_1} W] \\ v_0| P(W) \\ (v_0,u)=1}} \sum_{z \equiv 0\,(v_0)} \tau(z)^{O(1)} F_{N_1}(|z|^2)  
\end{align*}
Recall  that $\eta$ is large compared to $\eta_1$.  Hence, by counting the sum over $z \equiv 0\, (v_0)$  (using Lemma \ref{divisorlemma} to handle $\tau(z)^{O(1)}$) and applying Lemma \ref{smoothlemma}, we get
\[
S_{>} \ll \frac{\nu_1^2 N}{\varphi_{\ZZ[i]}(u) (\log N)^C}.
\]

For small $v$ we split into two cases depending on $\chi\neq 1$ and $\chi=1$.
For $\chi \neq 1$ we have by writing $z=uvz'+\alpha$
\begin{align*}
    &S_{\leq} = \sum_{\substack{v | P(W) \\ |v|^2 \leq N^{\eta_1}  \\ (v,uu_1)=1}}\mu(v)  \sum_{z \equiv 0\,(v)}  F_{N_1}(|z|^2) \mathbf{1}_{z \equiv w\, (u)} G(\arg z) \xi \chi(z)  \\ 
    &=
     \sum_{\substack{v | P(W) \\ |v|^2 \leq N^{\eta_1} }}\mu(v) \sum_{z'}  F_{N_1}(|uvz'+\alpha|^2)  G(\arg uvz + \alpha)  \xi(uv z' + \alpha) \chi (uvz'+\alpha) \\
     & =  \sum_{\substack{v | P(W) \\ |v|^2 \leq N^{\eta_1} }}\mu(v) \chi(uv) \sum_{z'}  F_{N_1}(|uvz'+\alpha|^2)  G(\arg uvz + \alpha)  \xi(uv z' + \alpha) \chi (z'+\alpha(uv)^{-1})
\end{align*}
since $(uv,u_1)=1$. Treating the weight
\[
 z' \mapsto F_{N_1}(|uvz'+\alpha|^2)  G(\arg uvz + \alpha)  \xi(uv z' + \alpha)  
\]
as a smooth weight we get by the Poly\'a-Vinogradov bound (Lemma \ref{pvongaussianlemma})
\[
S_{\leq}  \ll N^{O(\eta_1)} |u_1| \ll  N^{-\eta/2} \frac{\nu_1^2 N}{\varphi_{\ZZ[i]}(u)}
\]
by $N^{1-\eta} > |u|^2|u_1|$ since $\eta_1$ is small compared to $\eta$.

For $\chi=1$ we have by Lemma \ref{fourierserieslemma} and $\check{(G\xi_k)}(\ell) =\check{G}(\ell-k) $
\begin{align*}
    S_{\leq} =& \sum_{\substack{v | P(W) \\ |v|^2 \leq N^{\eta_1}  \\ (v,uu_1)=1}}\mu(v)  \sum_{z \equiv 0\,(v)}  F_{N_1}(|z|^2) \left( \mathbf{1}_{z \equiv w\, (u)} G(\arg z) \xi(z) - \frac{\check{G}(-k)}{\varphi_{\ZZ[i]}(u)} \right) \\
    =&  \sum_{\substack{v | P(W) \\ |v|^2 \leq N^{\eta_1}  \\ (v,uu_1)=1}}\mu(v) \sum_{\ell\neq 0} \check{G}(\ell-k) \ \sum_{z \equiv 0\,(v)}  F_{N_1}(|z|^2)\mathbf{1}_{z \equiv w\, (u)}  \xi_\ell (z)  \\
   & + \sum_{\substack{v | P(W) \\ |v|^2 \leq N^{\eta_1}  \\ (v,uu_1)=1}}\mu(v)  \check{G}(-k) \ \sum_{z \equiv 0\,(v)}  F_{N_1}(|z|^2) \left( \mathbf{1}_{z \equiv w\, (u)}  - \frac{1}{\varphi_{\ZZ[i]}(u)} \right).
\end{align*}
Estimating the contribution from $|\ell| > (\nu_1)^{-3}$ trivially (by Lemma \ref{fourierserieslemma}) and for  $|\ell| \leq (\nu_1)^{-3}$ applying the Poisson summation formula on $\ZZ[i]$ we get
\[
 S_{\leq} \ll  N^{-\eta/2} \frac{\nu_1^2 N}{\varphi_{\ZZ[i]}(u)}.  \qedhere
\]
\end{proof}

Our main lemma about the approximation is the following, which says that $\beta^\flat_\n$ is balanced over arithmetic progressions restricting to small polar boxes.
\begin{lemma} \label{stronggallagherlemma}
Let $N>  X^\eta Q^3$. For any $C_1> 0$ there is some $C_2 \ll_{C_1} 1$  such that the following holds. Let $\xi=\xi_k$ with $k=(\log N)^{O(1)}$. Let $\beta_\n$ be $Q$-regular, and let $\beta_\n^\flat$ be as above with $|\Psi|=J \leq (\log X)^{C_2}$. Suppose that $\beta_\n$ is supported on $(\n,P(W))=1$.  Let  $N'= N(1+\nu_2)^{j} \in [N,2N]$ and $\theta \in \RR/2\pi \ZZ$.  Then for  any $u \in \ZZ[i]$ with $\M(u) \leq Q$ we have 
\[
S=  \frac{1}{\varphi_{\ZZ[i]}(u)} \sum_{\psi \in \widehat{(\ZZ[i]/u\ZZ[i])^\times}} \bigg| \sum_{\substack{\n}} \beta_\n^\flat  \psi(\n)H_{N'}(\n)G(\arg \n -\theta) \bigg|^2 \ll \frac{\nu_1^2 \nu_2^2 N^2}{\varphi_{\ZZ[i]}(u) (\log X)^{C_1}}.
\]
\end{lemma}
\begin{proof}
 Let us first show that in the approximation $\beta_\n^\#$ we can replace the characters with conductor dividing $u$  by characters with modulus $u$. Suppose that $\psi$ is induced by a primitive character $\psi'$ of modulus $u' < u.$ Then
\[
\psi(\n) = \psi'(\n) \mathbf{1}_{(z,u/u')} =  \psi'(\n)  - \psi'(\n) \mathbf{1}_{(\n,u/u') > 1}.
\]
Since we have defined the correlation by sums over $(\n,P(W))=1$, the characters agree unless $\N (\n,(u/u')) > W$, so that we have
\[
\C_W(\beta, \xi_{k_j}\psi'  H_{N'}) = \C_W(\beta,  \xi_{k_j}\psi H_{N'})+ O(W^{-1} (\log X)^{O(1)}).
\]
Thus, if the approximation includes a character whose conductor is a proper divisor of $u$, we may replace it by the character with $(\n,u)=1$ at a negligible cost (by using orthogonality of characters). Let us assume that this has been done, so that the moduli of the characters $\chi_j$ satisfy either $u_j=u$ or $(u_j,u)=1$. Let $\Psi_u$ denote the characters $\xi\psi$ modulo $u$  that are equal to  $\xi_{k_j}\chi_{j}$ for some $j$. By expanding $G$ with Lemma \ref{fourierserieslemma} we get
\begin{align*}
  S=   \frac{1}{\varphi_{\ZZ[i]}(u)}  \sum_{\substack{\psi \in \widehat{(\ZZ[i]/u\ZZ[i])^\times}  }} \bigg|\sum_{\substack{k}} \check{G}(k)  \bigg( \sum_{\substack{\n   }} \beta^\flat_\n  \xi_k \psi(\n) H_{N'}(\n) \bigg) \bigg|^2  
\end{align*}
Denote
\begin{align*}
    S(\Psi_u) := & \frac{1}{\varphi_{\ZZ[i]}(u)}  \sum_{\substack{\psi \in \widehat{(\ZZ[i]/u\ZZ[i])^\times}  }} \bigg|\sum_{\substack{k \\ \xi_k \psi \in \Psi_u}} \check{G}(k)  \bigg( \sum_{\substack{ \n  }} \beta^\flat_\n  \xi_k \psi(\n) H_{N'}(\n)\bigg) \bigg|^2, \\
       S(\Psi_u^{\complement}) := & \frac{1}{\varphi_{\ZZ[i]}(u)}  \sum_{\substack{\psi \in \widehat{(\ZZ[i]/u\ZZ[i])^\times}  }} \bigg|\sum_{\substack{k \\\xi_k \psi \not \in \Psi_u}} \check{G}(k)  \bigg( \sum_{\substack{ \n   }} \beta^\flat_\n  \xi_k \psi(\n) H_{N'}(\n)\bigg) \bigg|^2,
\end{align*}
Then by Cauchy-Schwarz ($(A+B)^2 \leq 2(A^2+B^2)$) we have
\[
S \ll S(\Psi_u)+  S(\Psi_u^{\complement}).
\]
\subsubsection{Bounding $S(\Psi_u)$}
By definition of $\beta^\flat_\n$ we see that for $\xi_k \psi = \xi_{k_{j_0}} \chi_{j_0}$ we have
\begin{align*}
    \sum_{\substack{\n  }} \beta^\flat_\n  \xi_k \psi(\n)H_{N'}(\n) = \sum_{j \neq j_0}  \C_W(\beta, \overline{\xi_{k_j}\chi_j} H_{N'}) \sum_{\substack{\n }}   H_{N'}(\n)\mathbf{1}_{(\n,\overline{\n}P(W))=1} \xi_k\psi \overline{\xi_{k_j}\chi_j}(\n).
\end{align*}
Thus, by Cauchy-Schwarz on $j$ and $k$ (recall that for $\psi \xi_k \in \Psi$ we have $|k| \leq (\nu_1)^{-2}$ by $Q$-regularity) and using $|\Check{G}(k)| \ll \nu_1$ we have
\begin{align*}
    S(\Psi_u) \ll&  \frac{1 }{\varphi_{\ZZ[i]}(u)}   \sum_{\substack{\psi \in \widehat{(\ZZ[i]/u\ZZ[i])^\times}  }}  \bigg|   \sum_{j \leq J} \C_W(\beta, \overline{\xi_{k_j}\chi_j} H_{N'}) \\
    &\times
    \sum_{\substack{k \\ \xi_k \psi \in \Psi_u\\ \xi_k \psi \neq \xi_{k_{j}} \chi_j }}\check{G}(k) \sum_{\substack{\n }}  H_{N'}(\n)\mathbf{1}_{(\n,\overline{\n}P(W))=1} \xi_k\psi \overline{\xi_{k_j}\chi_j}(\n)\bigg|^2 \\
\ll& \sup_{\substack{j \leq J }}\frac{(\log X)^{3C'} }{\varphi_{\ZZ[i]}(u)} \sum_k |\check{G}(k)|^2\sum_{\substack{\psi \in \widehat{(\ZZ[i]/u\ZZ[i])^\times}\\ \xi_k \psi \in \Psi_u \\ \xi_k\psi  \neq \xi_{k_j} \chi_j  }}  \bigg|  \sum_{\substack{\n }}  H_{N'}(\n)\mathbf{1}_{(\n,\overline{\n}P(W))=1} \xi_k\psi \overline{\xi_{k_j}\chi_j}(\n)\bigg|^2 \\
\ll & \sup_{\substack{j \leq J \\ |k| \leq \nu_1^{-2} \\ \xi_k\psi \neq \xi_{k_j} \chi_j }}\frac{\nu_1^2 (\log X)^{4C'} }{\varphi_{\ZZ[i]}(u)}  \bigg|  \sum_{\substack{\n }}  H_{N'}(\n)\mathbf{1}_{(\n,\overline{\n}P(W))=1} \xi_k\psi \overline{\xi_{k_j}\chi_j}(\n)\bigg|^2.
\end{align*}
By Lemma \ref{flapproximationlemma} we get
\[
S(\Psi_u) \ll  \frac{\nu_1^2 \nu_2^2 N^2 }{\varphi_{\ZZ[i]}(u) (\log X)^C} .
\]

\subsubsection{Bounding $S(\Psi_u^{\complement})$}
By Cauchy-Schwarz ($(A+B)^2 \leq 2(A^2+B^2)$) we have
\begin{align*}
    S(\Psi_u^{\complement}) = & \frac{1}{\varphi_{\ZZ[i]}(u)}  \sum_{\substack{\psi \in \widehat{(\ZZ[i]/u\ZZ[i])^\times}  }} \bigg|\sum_{\substack{k \\\xi_k \psi \not \in \Psi_u}} \check{G}(k)  \bigg( \sum_{\substack{ \n   }} \beta^\flat_\n  \xi_k \psi(\n) H_{N'}(\n)\bigg) \bigg|^2 \\
    \ll &  S_1(\Psi_u^{\complement}) + S_2(\Psi_u^{\complement})
\end{align*}
with
\begin{align*}
  S_1(\Psi_u^{\complement}) :=& \frac{1}{\varphi_{\ZZ[i]}(u)}  \sum_{\substack{\psi \in \widehat{(\ZZ[i]/u\ZZ[i])^\times}  }} \bigg|\sum_{\substack{k \\\xi_k \psi \not \in \Psi_u}} \check{G}(k)  \bigg( \sum_{\substack{ \n   }} \beta_\n  \xi_k \psi(\n) H_{N'}(\n)\bigg) \bigg|^2 \\
  S_2(\Psi_u^{\complement}) := &\frac{1}{\varphi_{\ZZ[i]}(u)}  \sum_{\substack{\psi \in \widehat{(\ZZ[i]/u\ZZ[i])^\times}  }} \bigg|\sum_{j \leq J} \C_W(\beta, \overline{\xi_{k_j}\chi_j} H_{N'}) \\
  & \hspace{100pt}\times \sum_{\substack{k \\\xi_k \psi \not \in \Psi_u}} \check{G}(k)  \bigg( \sum_{\substack{ \n   }} H_{N'}(\n)\mathbf{1}_{(\n,\, \overline{\n} P(W))=1} \xi_k\psi\overline{\xi_{k_j}\chi_j}(\n) \bigg) 
   \bigg|^2.
\end{align*}
By the assumption that $\beta_\n$ is $Q$-regular we have
\[
S_1(\Psi_u^{\complement})  \ll   \frac{\nu_1^2 \nu_2^2 N^2 }{\varphi_{\ZZ[i]}(u) (\log X)^{C_1}}.
\]

For $S_2(\Psi_u^{\complement})$ we have by Cauchy-Schwarz on $j$
\[
  S_2(\Psi_u^{\complement}) :=  \sup_{j \leq J}\frac{(\log X)^{C'}}{\varphi_{\ZZ[i]}(u)}  \sum_{\substack{\psi \in \widehat{(\ZZ[i]/u\ZZ[i])^\times}  }} \bigg| \sum_{\substack{k \\\xi_k \psi \not \in \Psi_u }} \check{G}(k)  \bigg( \sum_{\substack{ \n   }} H_{N'}(\n)\mathbf{1}_{(\n,\, \overline{\n} P(W))=1}  \xi_k\psi\overline{\xi_{k_j}\chi_j}(\n)\bigg) 
   \bigg|^2.
\]
we write
\[
\sum_{\substack{k \\\xi_k \psi \not \in \Psi_u}} \check{G}(k)  = \sum_{\substack{k }} \check{G}(k) -  \check{G}(k_j)\mathbf{1}_{ \psi = \chi_j} -    \sum_{\substack{k  \\ \xi_k \psi \in \Psi_u \\ \xi_k \psi \neq \xi_{k_j}\chi_j}} \check{G}(k). 
\]
The contribution from the third sum may be extracted from $  S_2(\Psi_u^{\complement})$ by Cauchy-Schwarz and bounded by the same argument as with $  S(\Psi_u)$. Thus, we are left with bounding
\begin{align*}
     S_3(\Psi_u^{\complement}) := & \sup_{j \leq J}\frac{(\log X)^{C'}}{\varphi_{\ZZ[i]}(u)}  \sum_{\substack{\psi \in \widehat{(\ZZ[i]/u\ZZ[i])^\times}  }} \bigg| \sum_{\substack{k }} \check{G}(k)  \bigg( \sum_{\substack{ \n   }} H_{N'}(\n)\mathbf{1}_{(\n,\, \overline{\n} P(W))=1}  \xi_k\psi\overline{\xi_{k_j}\chi_j}(\n)\bigg)  \\ 
     &\hspace{150pt}-  \check{G}(k_j)\mathbf{1}_{ \psi = \chi_j} \sum_{\substack{ \n   }} H_{N'}(\n)\mathbf{1}_{(\n,\, \overline{\n} P(W))=1} 
   \bigg|^2 \\
   =& \sup_{j \leq J}\frac{(\log X)^{C'}}{\varphi_{\ZZ[i]}(u)}  \sum_{\substack{\psi \in \widehat{(\ZZ[i]/u\ZZ[i])^\times}  }} \bigg|   \sum_{\substack{ \n   }} H_{N'}(\n) G(\arg \n) \mathbf{1}_{(\n,\, \overline{\n} P(W))=1}  \psi\overline{\xi_{k_j}\chi_j}(\n)  \\ 
     &\hspace{150pt}-  \check{G}(k_j)\mathbf{1}_{ \psi = \chi_j} \sum_{\substack{ \n   }} H_{N'}(\n)\mathbf{1}_{(\n,\, \overline{\n} P(W))=1} 
   \bigg|^2.
\end{align*}
The claim now follows by orthogonality of characters and Lemma \ref{flapproximationlemma} since $N^{1-\eta} > Q^3$ and $|u|,|u_j| \leq Q$.
\end{proof}

\begin{remark}
When constructing the approximation $\beta^\#_\n$ we have a choice of using either the physical space or the Fourier space. To approximate $\beta_{\n}$ with respect to arithmetic progressions and sectors of $\ZZ[i]$ we use the Fourier space (ie. characters $\chi,\xi$), where as to approximate $\beta_{\n}$ with respect to the size of $\N \n$ we use the physical space (ie. smooth partition $H_{N'}$). These are the most convenient choice for using existing zero density estimates and information about exceptional characters.
\end{remark}

\subsection{Partitioning the Type II sum}
With the approximation for $\beta_\n$ defined as in Section \ref{approximationsection}, we can extract the main term from the Type II sum  by writing
\[
S(\alpha,\beta) = S(\alpha,\beta^\#)+S(\alpha,\beta^\flat).
\]
 By $(a,b)=1$ we may restrict to $(w,\overline{w})=1$. Since we are working with rough numbers the condition $(a,b)=1$ may be dropped with a negligible error term By definition $a_\n = \sum_{u \in \{\pm 1,\pm i\}} a_{uz}$ so that we have
\[
S(\alpha,\beta^\flat) = \sum_{\substack{|w|^2 \sim M \\ (w,\overline{w})=1}}\sum_{|z|^2 \sim N} \alpha_w \beta^\flat_z \mathbf{1}_{B}(\Re(\overline{w}z)) + O_C\bigg(\frac{X^{1/2} |B|}{(\log X)^C}\bigg).
\]

The two contributions are bounded by the following two propositions, which together imply Proposition \ref{typeiiprop}.
\begin{prop} \label{flattypeiiprop}
Suppose that the assumptions of Proposition \ref{typeiiprop} hold and let $\beta_\n^\flat$ be as in Section \ref{approximationsection}.
Then for every $C_1>0$ there is some $C_2>0$
\[
\sum_{\substack{|w|^2 \sim M \\ (w,\overline{w})=1}} \sum_{|z|^2 \sim N} \alpha_w \beta^\flat_z \mathbf{1}_B (\emph{\Re} (\overline{w} z ) )  \ll \frac{X^{1/2}|B|}{ (\log X)^{C_1}}.
\]
\end{prop}
\begin{prop} \label{sharptypeiiprop}  Suppose that the assumptions of Proposition \ref{typeiiprop} hold and let $\beta^\#_\n$ be as in  Section \ref{approximationsection}.
Then for any $C>0$
\begin{align*}
  \sum_{\m,\n }  \alpha_\m \beta^{\#}_\n a_{\m\n} =    \sum_{j \leq J } \sum_{\m,\n } \frac{\alpha_\m \beta_\n \xi_{k_j}\chi_j(\m\n)}{\N \m\n}   \sum_{\a}  \frac{F(\N\m\n/\N \a)}{\widehat{F}(0) }  a^\omega_{\a} \overline{\xi_{k_j}\chi_j}(\a) + O \bigg( \frac{X^{1/2}|B|}{(\log X)^C}\bigg).
  \end{align*}
\end{prop}

\begin{remark} \label{multiplicity remark}
We want to carry the condition $(w,\overline{w})=1$ through the application of Cauchy-Schwarz. To see why, consider a situation where $B \subseteq d_0 \ZZ$ with $d_0 > X^\eta$ being very smooth so that $\tau(d_0)$ is larger than any fixed power of $\log X$. Let us split the Type II sum according to the gcd of $w$ and $d_0$, which gives us
\[
\sum_{e|d_0} \sum_{w \equiv 0\, (e)} \sum_{z} \mathbf{1}_B (\Re (\overline{w} z ) ).
\]
Now in the inner sum $e| w$ means that $e|b$ is automatic,  so that the density on the inside is bumped up, that is, we morally have
\[
\sum_{w \equiv 0\, (e)} 1 \approx \frac{M}{e^2}, \quad \sum_{z} \mathbf{1}_B (\Re (\overline{w} z ) ) \approx e \frac{X^{1/2}|B|}{M} 
\]
Prior to Cauchy-Schwarz this is not an issue since we still get  converging sum $\sum_{e|d_0} e^{-1}$. However, after applying Cauchy-Schwarz to $w$ we get
\[
\sum_{e|d_0} \sum_{w \equiv 0\, (e)} \bigg| \sum_{z}\mathbf{1}_B  (\Re (\overline{w} z ) ) \bigg|^2 \approx N |B|^2 \sum_{e|d_0} 1,
\]
which means that we have picked up a large divisor function $\tau(d_0)$. This would be problematic since we can only save a fixed power of $\log X$ from Lemma \ref{stronggallagherlemma}.
We resolve this issue by keeping the condition $(w,\overline{w})=1$ so that $w$ has no non-trivial integer divisors, but there are also other ways to deal with this.
\end{remark}
\subsection{Proof of Lemma \ref{mobiusregularlemma}} \label{mobiusregularitylemmasection}
 Note that $\M(u) \leq Q$ implies that $|u| \,\leq Q$, so that by $N > X^\eta Q^{3}$ we have $N > X^{\eta} |u|^{3}$. By Lemma \ref{zerodensitylemma} we can take for $\sigma_Q := 1- \frac{C_2' \log \log Q}{ \log Q}$ with some large $C_2'> 0$ to get
\[
J\leq N^\ast(\sigma_Q,X^\eta,X^{\eta},Q^2) \leq (\log X)^{C_2}
\]
and let $\Psi$ be the set of primitive characters $\xi\psi$ such that $L(s,\xi\psi)$ has a zero counted in the above with $\M(u) \leq Q$.
Recall that we now specify
\[
\beta_\n := \mathbf{1}_{\N \n \sim N} \mu(\N\n) \mathbf{1}_{(\n,P(W))=1}.
\]
Then for Lemma \ref{mobiusregularlemma} we need to show that if $\Psi_u$ denotes the set of characters modulo $u$ which are induced by $\Psi$, then for any $N' \in [N,2N]$ for any $C_1>0$ there is some $C_2>0$ such that
  \begin{align} \label{mobiusclaim}
       \frac{1}{\varphi_{\ZZ[i]}(u)}  \sum_{\psi \in \widehat{(\ZZ[i]/u\ZZ[i])^\times}} \bigg| \sum_{\substack{k \\ \xi_k \psi \not \in \Psi_u}} \check{G}(k) \bigg(\sum_{\N \n \in (N',N'(1+\nu_2)]} \beta_{ \n}  \xi_k\psi(\n)\bigg) \bigg|^2 \ll \frac{\nu_1^2\nu_2^2 N^2}{\varphi_{\ZZ[i]}(d) (\log X)^{C_1}}.
  \end{align}
The contribution from the trivial character $\psi=\psi_0$ is bounded by a similar but easier argument as below, using Heath-Brown's identity and the Vinogradov strength zero-free region of Coleman \cite{coleman_1990}. We then restrict to $\psi \neq \psi_0.$ We apply Section \ref{smoothweightsection} to $\N\n$ with $\nu_1=X^{-\eta_1}$, using the assumption that $\eta_2$ is small compared to $\eta_1$ to replace $\N \n \in ( N',N'(1+\nu_2)]$ with a smooth weight. It then suffices to show that for any $N_1 \sim N$ we have
\[
 \frac{1}{\varphi_{\ZZ[i]}(u)}  \sum_{\substack{\psi \in \widehat{(\ZZ[i]/u\ZZ[i])^\times}\\ \psi\neq \psi_0}} \bigg| \sum_{\substack{k \\ \xi_k \psi \not \in \Psi_u}} \check{G}(k)  \bigg(\sum_{\n } \beta_{ \n}  \xi_k\psi(\n) F_{N_1}(\N\n) \bigg) \bigg|^2 \ll \frac{\nu_1^4 N^2}{\varphi_{\ZZ[i]}(d) (\log X)^{C_1}}
\]
 
The proof strategy is classical so we will be brief. We use the Heath-Brown identity \cite[(13.58)]{IK} with $K=3$
\[
\mu(n) =\sum_{k=1}^3 (-1)^{k+1}\binom{3}{k}\sum_{\substack{n=m_1m_2m_2n_1n_2n_3 \\ m_j \leq 2 N^{1/3}}} \mu(m_1) \mu(m_2) \mu(m_3)
\]
Let $F_1$ be as in Section \ref{smoothweightsection} with $\nu=1/2$. Using the Heath-Brown identity and a dyadic decomposition (smooth for the free variable $\a$) we get sums of Type I 
\[
S_I :=  \frac{1}{\varphi_{\ZZ[i]}(u)} \sum_{\substack{\psi \in \widehat{(\ZZ[i]/u\ZZ[i])^\times}\\ \psi\neq \psi_0}} \bigg| \sum_{\substack{k \\ \xi_k \psi \not \in \Psi_u}} \check{G}(k) \bigg(    \sum_{\substack{\a \\ \N \m \sim M \\  \\ (\m\a,P(W))=1}}F_{N_1}(\N\m \a)\alpha(\m) F_{1,A}(\N\a)  \xi_k\psi(\m \a)\bigg)\bigg|^2 
\]
with $M  \ll 2N^{2/3}$ and $AM \sim N$ and sums of Type II
\begin{align*}
  S_{II}  := \frac{1}{\varphi_{\ZZ[i]}(u)}  \sum_{\substack{\psi \in \widehat{(\ZZ[i]/u\ZZ[i])^\times}\\ \psi\neq \psi_0}} \bigg| \sum_{\substack{k \\ \xi_k \psi \not \in \Psi_u}} \check{G}(k) \bigg( 
  \sum_{\substack{  \N \m_j \sim M_j \\ (\m_j,P(W))=1}} F_{N_1}(\N\m_1\m_2\m_3)  \hspace{40pt}  \\ \times \alpha_1(\m_1)  \alpha_2(\m_2) \alpha_3(\m_3)\xi_k\psi(\m_1 \m_2 \m_3)\bigg) \bigg|^2   
\end{align*}
with $\alpha_1(\m) = 1$ or $\alpha_1(\m) = \mu(\N \m)$, $M_1M_2M_3 \sim M$ and
\[
N^{1/6} \ll M_1 \ll N^{1/3}.
\]
To see this note that we are in the Type I case unless all of the variables $n_j$ are $\ll N^{1/3}$, and in that case we can take  the $M_1$ for the Type II sum to be the range of largest variable $m_j,n_j$, which must be $\gg N^{1/6}$.

To show \eqref{mobiusclaim} it then suffices to show that
\[
S_I, \,  S_{II} \,\ll \frac{\nu_1^4 N^2}{\varphi_{\ZZ[i]}(d) (\log X)^{C_1}}.
\]
For $S_I$ we let $D= X^{\eta_1}$ and write 
\[
\mathbf{1}_{(\a,P(W))=1} = \sum_{\d | (\a,P(W))} \mu(\d) =  \sum_{\substack{\d | (\a,P(W)) \\ \N\d \leq D}} \mu(\N\d) +\sum_{\substack{\d | (\a,P(W)) \\ \N\d > D}} \mu(\N\d).
\]
The contribution from $\N\d > D$ is bounded by using Lemma \ref{smoothlemma}, after using orthogonality of characters and Lemma \ref{fourierserieslemma}.
For $S_I$ we then need to bound
\[
S_I' :=  \frac{1}{\varphi_{\ZZ[i]}(u)} \sum_{\substack{\psi \in \widehat{(\ZZ[i]/u\ZZ[i])^\times}\\ \psi\neq \psi_0}} \bigg| \sum_{\substack{k \\ \xi_k \psi \not \in \Psi_u}} \check{G}(k) \bigg(    \sum_{\substack{ \a \\ \N \m \sim M  \\ (\m,P(W))=1}} \sum_{\substack{\d | (\n,P(W)) \\ \N\d \sim D}} \mu(\d) F_{1,A}(\N \d \a)\alpha(\m)   \xi_k\psi(\d\m \a)\bigg)\bigg|^2 
\]
with  $D \ll X^{\eta_1}$ and $M \ll N^{2/3}$.

 Denote
\begin{align*}
    M(t,\xi\psi) &:= \sum_{\N \m \sim M} \mathbf{1}_{(\m,P(W))=1} \alpha(\m)  (\N \m)^{-it}\xi\psi(\m),  \\
     M_j(t,\xi\psi) &:= \sum_{\N \m \sim M_j} \mathbf{1}_{(\m,P(W))=1} \alpha_j(\m)  (\N \m)^{-it}\xi\psi(\m), \\
     A(t,\xi\psi) &:= \sum_{\a} \sum_{\substack{\d|(\a,P(W)) \\ \N\d \leq D}}\mu(\d) F_{1,A}(\N\d\a) (\N \d \a)^{-it}\xi\psi(\a).
\end{align*}
Then for $\psi \neq \psi_0$  we have the standard point-wise bounds
\begin{align} \label{convexitybound}
    A(t,\xi_k\psi) \ll D(1+|t|) (1+|k|)|u|
\end{align}
and for $\alpha_1=1$ or $\alpha_1=\mu$ with $M_1 \gg X^{1/6}$ once $C_2$ is large compared to $C_1$ with $\xi\psi \not \in \Psi_u$
\begin{align} \label{vingradovbound}
   M_1(t,\xi\psi) \ll \frac{M_1}{(\log X)^{C_1}}
\end{align}
The bound \eqref{convexitybound} follows by the Poly\'a-Vinogradov  bound (ie. the convexity bound for $L(s,\xi\psi)$ in the $u$ aspect, Lemma \ref{pvongaussianlemma}). The bound \eqref{vingradovbound} follows by the truncated Perron's formula and shifting the contour to $(1+\sigma_Q)/2$ (justified by $\xi \psi \not \in \Psi_u$), using the bound Lemma \ref{Lnearzerolemma} for $1/L(s,\xi \psi)$, and taking $C_2'>0$ in the definition of $\sigma_Q$ sufficiently large.

By Mellin inversion (Lemma \ref{improvedmvtlemma})  we get
\[
F_{N_1}(x) = \frac{1}{2\pi i}\int \Dot{F}(s) N_{1}^{s} x^{-s} ds.
\]
with
\[
|\Dot{F}(s)| \ll_C \nu_1 (1+\nu_1|t|)^{-C}.
\]
Hence, we have
\[
S_I' \ll J_I \quad \text{and} \quad S_{II} \ll   J_{II}
\]
with
\begin{align*}
  J_I &:=  \frac{1}{\varphi_{\ZZ[i]}(u)} \sum_{\substack{\psi \in \widehat{(\ZZ[i]/u\ZZ[i])^\times}\\ \psi\neq \psi_0}} \bigg| \sum_{\substack{k \\ \xi_k \psi \not \in \Psi_u}} |\check{G}(k)|\int |\Dot{F}(it)||MA(t,\xi_k \psi)| \,dt \bigg|^2 
 \\
 J_{II}& := \frac{1}{\varphi_{\ZZ[i]}(u)}\sum_{\substack{\psi \in \widehat{(\ZZ[i]/u\ZZ[i])^\times}\\ \psi\neq \psi_0}} \bigg| \sum_{\substack{k \\ \xi_k \psi \not \in \Psi_u}} |\check{G}(k)|\int |\Dot{F}(it)||M_1M_2 M_3(t,\xi_k \psi) |\,dt \bigg|^2 
\end{align*}

To bound $J_I$  we apply  Cauchy-Scwarz on $t,k$, \eqref{convexitybound} and orthogonality of characters to get for some coefficients $\gamma_\n$ and for some $t,\xi$
\begin{align*}
  J_I& \ll \nu_1^{O(1)}    \frac{D^2|u|^2}{\varphi_{\ZZ[i]}(u)} \sum_{\psi \in \widehat{(\ZZ[i]/u\ZZ[i])^\times}} \bigg|   M (t,\xi \psi)\bigg|^2 \\
   & \ll \nu_1^{O(1)}  D^2 |u|^2 \sum_{\substack{\n_1 \equiv \n_2 \, (u) \\ \N \n_1, \N \n2 \ll M }} |\gamma_{\n_1} \overline{\gamma_{\n_2}}| \\
  & \ll \nu_1^{O(1)} D^2 M^2  + \nu_1^{O(1)} |u| D^2 M.
\end{align*}
By $N > N^{\eta} |u|^{3} $, $M \ll N^{2/3}$, and $D =X^{\eta_1}$ we get 
\[
J_I \ll N^{4/3 +  O(\eta_1)} + |u| N^{2/3 + O(\eta_1)}   \ll  N^{-\eta}\frac{N^2}{\varphi_{\ZZ[i]}(u)}
\]
since $\eta_1$ is small compared to $\eta$, which is sufficient for bounding $J_I$.

For $J_{II}$ we apply the bound \eqref{vingradovbound} for $M_1$ and Cauchy-Schwarz in the $t,k$ variables to get
\begin{align*}
 J_{II} \ll &\frac{M_1^{2}}{(\log X)^{C_1}}  \iint |\Dot{F}(it_1)\Dot{F}(it_2)| \sum_{k_1,k_2} |\check{G}(k_1)\check{G}(k_2) | \\
& \times \frac{1}{\varphi_{\ZZ[i]}(u)}  \sum_{\psi \in \widehat{(\ZZ[i]/u\ZZ[i])^\times}}    |M_2(t_1,\xi_{k_1} \psi)M_3(t_2,\xi_{k_2} \psi)|^2   dt_1 d t_2.  
\end{align*}

By orthogonality of characters and Lemmas \ref{fourierserieslemma} and \ref{improvedmvtlemma} we get
\begin{align*}
    J_{II} &\ll \frac{M_1^{2}}{(\log X)^{C_1}} T^2 \sum_{\substack{\m_{21},\m_{22},\m_{31} \m_{32} \\ \N \m_{jk} \sim M_j \\  \m_{21} \m_{31} \equiv   \m_{22} \m_{32} \, (u)   \\ |\N \m_{21} - \N \m_{22}  |\ll \nu_1 M_2 \\|\arg \m_{21} - \arg \m_{22}  |\ll \nu_1 \\  | \N\m_{31} - \N\m_{32}  | \ll \nu_1 M_3 \\ | \arg\m_{31} - \arg \m_{32}  | \ll \nu_1}} 1 
\end{align*}
Let us denote $\m_{2j} = (w_j)$ and $\m_{3j}=(z_{j})$
so that
\[
  J_{II} \ll \frac{M_1^{2}}{(\log X)^{C_1}} T^2 \sum_{\substack{w_1,w_2,z_1,z_2  \\  |w_j|^2 \sim M_2 \\ |z_j|^2 \sim M_3 \\ w_1z_1 \equiv w_2z_2 \, (u)   \\ | w_1 - w_2  |^2 \ll \nu_1 M_2 \\ |z_1 - z_2 |^2 \ll \nu_1 M_3 }} 1. 
\]
Writing $w_2=w_1+u$, $z_1=z_2+v$ we get (using Lemma \ref{divisorlemma} to handle $\tau(w) \tau(z) $)
\begin{align*}
    J_{II} &\ll \frac{M_1^{2}}{(\log X)^{C_1}} T^2 \sum_{\substack{w_1,u,v,z_2  \\  |w_1|^2 \sim M_2 \\ |u|^2 \ll \nu_1 M_2  \\ |z_2|^2 \sim M_3 \\ v \ll \nu_1 M_3 \\ w_1v \equiv uz_2 \, (u) \\  }}  1 \\
 & \ll \frac{M_1^{2}}{(\log X)^{C_1}} T^2 \sum_{\substack{ |w|^2,|z|^2 \ll \nu_1 M_2 M_3 \\ w \equiv z \, (u) \\  }}  \tau(w) \tau(z) 
\\
 &\ll  \frac{ \nu^4 N^2}{\varphi_{\ZZ[i]}(u) (\log N)^{C_1-O(1)}},
\end{align*}
using $M_1 \ll N^{1/3}$, $M_1M_2 M_3 \asymp N$ to get $M_2M_3 \gg N^{2/3} \gg X^{\eta} |u|^2$. \qed

\section{Type II information: proof of Proposition \ref{flattypeiiprop}} \label{typeiimainsection}
\subsection{Cauchy-Schwarz}
   Let $F_M(m)=F(m/M)$ with a fixed smooth majorant $F$ for the interval $[1,2]$, supported on $[1/2,3]$.  By applying Cauchy-Schwarz we get
\begin{align*}
\sum_{\substack{|w|^2 \sim M \\ (w,\overline{w})=1}} \sum_{|z|^2 \sim N} \alpha_w \beta^\flat_z \mathbf{1}_B (\Re (\overline{w} z ))  \ll M^{1/2} U(\beta)^{1/2},
\end{align*}
where
\[
U(\beta) := \sum_{|z_1|^2, |z_2|^2 \sim N} \beta^\flat_{z_1} \overline{\beta^\flat_{z_2}} \sum_{(w,\overline{w})=1} F_M(|w|^2)\mathbf{1}_B (\Re (\overline{w} z_1 ) ) .
\]
It then suffices to show that
\begin{align} \label{Uclaim}
U(\beta) \ll \frac{N |B|^2}{(\log X)^{C_1}}.
\end{align}
Define 
\[
\Delta=\Delta(z_1,z_2) := \Im (\overline{z_1} z_2) = |z_1z_2| \sin (\arg z_2 - \arg z_1).
\]
Note that typically $|\Delta| \asymp N$.
We partition the sum into a main term and diagonal terms by writing
\[
U(\beta) = V(\beta) + O(U_0(\beta) + U_1(\beta) ),
\]
where
\begin{align*}
V(\beta)& :=  \sum_{\substack{|z_1|^2, |z_2|^2 \sim N \\ \Delta \neq 0 \\ (z_1,z_2) = 1}} \beta^\flat_{z_1} \overline{\beta^\flat_{z_2}} \sum_{(w,\overline{w})=1} F_M(|w|^2)\mathbf{1}_B (\Re  (\overline{w} z_1 ) ) \mathbf{1}_B (\Re  (\overline{w} z_2 ) ) \\
U_0(\beta) &:=  \sum_{\substack{|z_1|^2, |z_2|^2 \sim N \\ \Delta = 0}} |\beta^\flat_{z_1} \beta^\flat_{z_2} |\sum_{(w,\overline{w})=1} F_M(|w|^2)\mathbf{1}_B (\Re  (\overline{w} z_1 ) ) \mathbf{1}_B (\Re (\overline{w} z_2 ) )\\
U_1(\beta) &:=  \sum_{\substack{|z_1|^2, |z_2|^2 \sim N \\ \Delta \neq 0 \\ (z_1,z_2) > 1}} |\beta^\flat_{z_1} \beta^\flat_{z_2}|\sum_{(w,\overline{w})=1} F_M(|w|^2)\mathbf{1}_B (\Re  (\overline{w} z_1 ) ) \mathbf{1}_B (\Re (\overline{w} z_2 ) .
\end{align*}
  For $V(\beta)$ we apply Section \ref{smoothweightsection} with $G:\RR/2\pi\ZZ \to \RR$ being a non-negative smooth function with the parameter $\nu_1=X^{-\eta_1}$ to the variables $\arg z_1$ and $\arg z_2$  to get
\[
V(\beta) = \nu_1^{-2} \int_{(\RR/2\pi\ZZ)^2} V(\beta,\bm{\theta}) d \theta_1 d \theta_2
\]
with
\[
V(\beta,\bm{\theta}) =  \sum_{\substack{|z_1|^2, |z_2|^2 \sim N  \\ \Delta \neq 0 \\(z_1,z_2) = 1}} \beta^\flat_{z_1,\theta_1} \overline{\beta^\flat_{z_2,\theta_2}} \sum_{(w,\overline{w})=1} F_M(|w|^2)\mathbf{1}_B (\Re  (\overline{w} z_1 ) ) \mathbf{1}_B (\Re  (\overline{w} z_2 ) ) 
\]
and 
\[
\beta^\flat_{z,\theta} = \beta^\flat_{z} G( \arg z -\theta ).
\]
We now further partition $V(\beta)$ according to the size of $|\Delta|$. Let $\nu_3=X^{-\eta_3}$ and write 
\[
V(\beta) =  V_{> \nu_3}(\beta) + V_{\leq \nu_3}(\beta)
\]
where $V_{> \nu_3}(\beta)$ is the part where $|\sin (\theta_1-\theta_2) | > \nu_3$, which implies (since $\eta_3$ is small compared to $\eta_1$)
\[
|\Delta| = |z_1z_2|\, |\sin(\arg z_2-\arg z_1)| \gg \nu_3 N.
\]
We then have the following lemmas, which together imply Proposition \ref{flattypeiiprop}.
\begin{lemma} \label{offdiaglemma}\emph{(Off-diagonal contribution).} For $|\sin (\theta_1-\theta_2) | > \nu_3$
\[
V(\beta,\bm{\theta}) \ll   \frac{\nu_1^2 N |B|^2}{|\sin (\theta_1-\theta_2) | (\log X)^{C_1}}.
\]
\end{lemma}
\begin{lemma} \label{diag1lemma} \emph{(Diagonal contribution).}
\[
U_0(\beta) \ll_\eps  X^{1/2+\eps} |B| 
\]
\end{lemma}
\begin{lemma} \label{diag2lemma} \emph{(Pseudo-diagonal contribution I).}
\[
U_1(\beta) \ll   \frac{N |B|^2 }{W^{1-\eta}}
\]
\end{lemma}
\begin{lemma} \label{pseudodiaglemma} \emph{(Pseudo-diagonal contribution II).}
\[
V_{\leq \nu_3}(\beta) \ll_\eps X^\eps   \nu_3 N |B|^2  
\]
\end{lemma}
\subsection{Proof of Lemma \ref{offdiaglemma}}  \label{offdiagsection}
It suffices to show that for $|\theta_1 - \theta_2\pmod{\pi}| > \nu_3$ we have
\begin{equation} \label{vclaim1}
\begin{split}
          V(\beta,\bm{\theta}) =&  \sum_{\substack{|z_1|^2, |z_2|^2 \sim N   \\ (z_1,z_2) = 1}} \beta^\flat_{z_1,\theta_1} \overline{\beta^\flat_{z_2,\theta_2}}  \\
          & \times \sum_{(w,\overline{w})=1} F_M(|w|^2)\mathbf{1}_B (\Re  (\overline{w} z_1 ) ) \mathbf{1}_B (\Re  (\overline{w} z_2 ) ) \\
   & \ll \frac{\nu_1^2 N |B|^2}{|\sin (\theta_1-\theta_2) |(\log X)^C}
\end{split}
\end{equation}

We note that $(z_1,z_2)=(z_1,\overline{z_1}) = (z_2,\overline{z_2}) = 1$ implies $(\Delta,|z_1|^2|z_2|^2)=1.$ By symmetry this can be seen from
\[
\overline{(\Delta,\overline{z_1})} = (\Delta,z_1) = (\Im(z_2 \overline{z_1}),z_1) = (z_2 \overline{z_1},z_1) = 1.
\]

Denoting $\Re  (\overline{w} z_1 ) = b_j$, we have 
\begin{align*} 
    i \Delta w = z_2 b_1 -z_1b_2.
\end{align*}
Let $b_0:= (b_1,b_2)$. Since $(w,\overline{w})=1$, we know that $b_0|\Delta$.  Thus,
 \begin{align} \label{deltaeq1}
  i (\Delta/b_0) w = z_2 b_1/b_0 -z_1b_2/b_0.    
 \end{align}
 Denoting  $b_j= b_0b_j'$ we have
\[
 V(\beta,\bm{\theta}) =  \sum_{\substack{b_0\geq 1 }} \sum_{\substack{|z_1|^2, |z_2|^2 \sim N   \\ (z_1,z_2) = 1 \\ \Delta \equiv 0 \, (b_0)}} \beta^\flat_{z_1,\theta_1} \overline{\beta^\flat_{z_2,\theta_2}} \sum_{\substack{b'_2 \equiv a b'_1 \, ( \Delta/b_0) \\ (b'_1b'_2, \Delta/b_0)=1 \\ (b'_1,b'_2)=1 \\ (w,\overline{w})=1}} \mathbf{1}_B(b_0b'_1)\mathbf{1}_B(b_0b'_2)F_M\bigg(\bigg|b_0 \frac{z_2b'_1-z_1b'_2}{\Delta}\bigg|^2\bigg)
\]
We now wish to remove the smooth  weight $F_M$. Recall that already $b_0b'_j \in [Y,Y + X^{1/2-\eta}]$ by \eqref{Bintervalassumption}. We introduce a rough finer-than-dyadic partition for $|z_1|^2,|z_2|^2$ by using $H_{N'}$ as in Section \ref{approximationsection} with $\nu_2=X^{-\eta_2}$.  Let $N_{1},N_{1}\sim N$, and denote 
\[
\beta^\flat_{z,i} := \beta^\flat_{z}H_{N_{i}}(z)G(\arg z-\theta_i).
\] 
To prove \eqref{vclaim1} it then suffices to prove that for $N_{1},N_{1}\sim N$ and for $|\sin (\theta_1-\theta_2) | \geq \nu_3$ we have
\begin{equation} \label{vclaim2}
\begin{split}
          V'(\beta,\bm{N},\bm{\theta}) :=&  \sum_{\substack{|z_1|^2, |z_2|^2 \sim N    \\ (z_1,z_2) = 1}} \beta^\flat_{z_1,1} \overline{\beta^\flat_{z_2,2}}  \\
          & \times \sum_{(w,\overline{w})=1} F_M(|w|^2)\mathbf{1}_B (\Re  (\overline{w} z_1 ) ) \mathbf{1}_B (\Re  (\overline{w} z_2 ) ) \\
   & \ll \frac{\nu_1^2 \nu_2^2 N |B|^2}{|\sin (\theta_1-\theta_2) |(\log X)^C}.
\end{split}
\end{equation}
Using $|\sin (\theta_1-\theta_2) | > \nu_3$ we have $\Delta = |z_1z_2| \sin (\arg_2-\arg z_1) \gg \nu_3 N$,  so that for some constant $F_M(\bm{N},\bm{\theta})$
\begin{equation} \label{ccremoval}
F_{M} \bigg( \bigg| b_0 \frac{z_2b'_1 -z_1b'_2}{\Delta} \bigg|^2 \bigg) = F_M(\bm{N},\bm{\theta}) + O(\nu_2\nu_3^{-2} ).
\end{equation}
Therefore, we have
\[
  V'(\beta,\bm{N},\bm{\theta}) =  F_M(\bm{N},\bm{\theta})V(\beta,\bm{N},\bm{\theta}) +  \nu_2 \nu_3^{-2}  O(W(\beta,\bm{N},\bm{\theta}) )
\]
where
\begin{align*}
V(\beta,\bm{N},\bm{\theta}) :=   \sum_{b_0}\sum_{\substack{|z_1|^2, |z_2|^2 \sim N  \\ (z_1,z_2) = 1 \\   \Delta \equiv 0 \, (b_0)}} \beta^\flat_{z_1,1} \overline{\beta^\flat_{z_2,2}}  \sum_{\substack{b'_2 \equiv a b'_1 \, (\Delta/b_0) \\ (b'_1b'_2, \Delta/b_0)=1\\ (b'_1,b'_2)=1 \\ (w,\overline{w})=1}} \mathbf{1}_B(b_0b'_1)\mathbf{1}_B(b_0b'_2) \\
W(\beta,\bm{N},\bm{\theta}) :=  \sum_{b_0}\sum_{\substack{(|z_1|^2, |z_2|^2 \sim N \\ (z_1,z_2) = 1 \\   \Delta \equiv 0 \, (b_0)}} |\beta^\flat_{z_1,1}| |\beta^\flat_{z_2,2} |\sum_{\substack{b'_2 \equiv a b'_1 \, (\Delta/b_0) \\ (b'_1b'_2,  \Delta/b_0)=1\\ (b'_1,b'_2)=1 \\ (w,\overline{w})=1}} \mathbf{1}_B(b_0b'_1)\mathbf{1}_B(b_0b'_2)  
\end{align*} 
Then \eqref{vclaim2} follows from the following two lemmas once $\eta_3$ is small compared to $\eta_2$ so that $\nu_2 \nu_3^{-3}=X^{-\eta_2+3\eta_3} < X^{-\eta_2/2}$, say.
\begin{lemma} \label{maintypeiilemma}
For any $C_1 >0$ there is some $C_2>0$ for $\beta^\flat_\n$ as in Section \ref{approximationsection} such that for $|\sin (\theta_1-\theta_2)|> \nu_3$ we have
\[
V(\beta,\bm{N},\bm{\theta}) \ll  \nu_1^2\nu_2^2\frac{N |B|^2}{|\sin (\theta_1-\theta_2)| (\log X)^{C_1}}
\]
\end{lemma}
\begin{lemma} \label{typeiitrivialboundlemma}
For $|\sin (\theta_1-\theta_2)| > \nu_3$ we have
\[
W(\beta,\bm{N},\bm{\theta})  \ll \frac{\nu_1^2 \nu_2^2 N |B|^2  (\log  X)^{O(1)}}{|\sin (\theta_1-\theta_2)|}.
\]
\end{lemma}

\subsubsection{Proof of Lemma \ref{maintypeiilemma}} \label{maintypeiilemmasection}
The condition $(w,\overline{w})=1$ has served its purpose so we remove it by expanding
 \[
 \mathbf{1}_{(w,\overline{w})=1} = \sum_{\ell|w} \mu (\ell),
 \]
 where $\ell$ runs over integers.  Writing $w=\ell w'$ we get from \eqref{deltaeq1}
 \begin{align} \label{deltaeq2}
    i (\Delta/b_0)  \ell w' = z_2 b_1' -z_1b_2'.   
 \end{align}

We have
\begin{align} \label{coprimeclaim}
\left(\frac{\ell \Delta}{b_0},b_j' \right) = 1
\end{align}
To see this, by using $(b_1',b_2')=1$ and (\ref{deltaeq2}) we get that for $j \in \{1,2\}$
\[
\left(\frac{\ell \Delta}{b_0},b_j' \right) 
 \, |\,\left(\frac{\ell \Delta}{b_0},z_j\right),
 \]
and \eqref{coprimeclaim} follows since the left-hand side is an integer and by $(z_j,\overline{z_j})=1$ the only integer dividing $z_j$ is $1$. Similarly we also get $(z_1z_2,\ell) =1$ by \eqref{deltaeq2}. Thus,
 \[
a:=  z_2/z_1 \equiv b_2'/b_1' \, (\ell)
 \]
 is congruent to a rational integer, which is equivalent to saying that $\Delta \equiv 0 \, (\ell)$. Thus, we get
 \[
 V(\beta,\bm{N},\bm{\theta}) =   \sum_{b_0,\ell }\mu(\ell)\sum_{\substack{|z_1|^2, |z_2|^2 \sim N \\ (z_1,z_2) = 1 \\   \Delta \equiv 0 \, ([b_0,\ell])}} \beta^\flat_{z_1,1} \overline{\beta^\flat_{z_2,2}}  \sum_{\substack{b'_2 \equiv a b'_1 \, (\ell\Delta/b_0) \\ (b'_1b'_2,\ell \Delta/b_0)=1\\ (b'_1,b'_2)=1}} \mathbf{1}_B(b_0b'_1)\mathbf{1}_B(b_0b'_2) 
 \]
 By expanding the congruence $b'_2 \equiv a b'_1 \, (\ell\Delta/b_0) $ into Dirichlet characters and splitting into primitive characters we get
\begin{align*}
   V(\beta,\bm{N},\bm{\theta}) = \sum_{b_0, \ell }\mu(\ell)&\sum_{\substack{|z_1|^2, |z_2|^2 \sim N \\ (z_1,z_2) = 1 \\ \Delta \equiv 0 \, ([b_0,\ell])}} \beta^\flat_{z_1,1} \overline{\beta^\flat_{z_2,2}} \\
   &\times\frac{1}{\varphi (\ell \Delta/b_0)}  \sum_{d| \ell \Delta/b_0} \sideset{}{^\ast}\sum_{\chi\, (d)} \chi(a)  \bigg(  \sum_{\substack{b_0b'_1,b_0b'_2 \in B \\  (b'_1,b'_2)=1 \\ (b_1'b_2',\ell \Delta/(db_0))=1}} \chi(b'_1) \overline{\chi}(b'_2) \bigg).
\end{align*}
We separate $db_0 > X^{\delta+\eta/2}$ and $ db_0 \leq X^{ \delta +\eta/2}$, that is, write
\[
V(\beta,\bm{N},\bm{\theta}) = V_{ \leq}(\beta,\bm{N},\bm{\theta})+ V_{>}(\beta,\bm{N},\bm{\theta}).
\]

For the large $db_0$ we use the expansion
\[
\mathbf{1}_{(b_1',b_2')=1} = \sum_{c|(b_1',b_2')} \mu(c)
\]
to get
\begin{align*}
    V_{>}(\beta,\bm{N},\bm{\theta}) =   \sum_{b_0,\ell }  & \mu (\ell)\sum_{\substack{|z_1|^2, |z_2|^2 \sim N \\ (z_1,z_2) = 1 \\ \Delta \equiv 0 \, ([b_0,\ell])}} \beta^\flat_{z_1,1} \overline{\beta^\flat_{z_2,2} }  \\
& \times \frac{1}{\varphi (\ell \Delta/b_0)}  \sum_{\substack{d| \ell\Delta/b_0  \\ db_0 > X^{\delta+\eta/2}}}\sum_{(c,d)=1} \mu(c) \sideset{}{^\ast}\sum_{\chi\, (d)} \chi(a)  \bigg|  \sum_{\substack{cb_0b \in B \\ (cb,\ell \Delta/(db_0))=1}} \chi(b)  \bigg|^2.
\end{align*}
We split the sum according to $a_0=(b_0,\ell)$ which gives us
\begin{align*}
    V_{>}(\beta,\bm{N},\bm{\theta}) = &  \sum_{\substack{b_0  \\ a_0 | b_0}}  \sum_{(\ell,b_0)=a_0}   \mu (\ell)\sum_{\substack{|z_1|^2, |z_2|^2 \sim N \\ (z_1,z_2) = 1 \\ \Delta \equiv 0 \, (b_0\ell/a)}} \beta^\flat_{z_1,1} \overline{\beta^\flat_{z_2,2} }  \\
& \times \frac{1}{\varphi (\ell \Delta/b_0)}  \sum_{\substack{d| \ell\Delta/b_0  \\ db_0 > X^{\delta+\eta/2}}}\sum_{(c,d)=1} \mu(c) \sideset{}{^\ast}\sum_{\chi\, (d)} \chi(a)  \bigg|  \sum_{\substack{cb_0b \in B \\ (cb,\ell \Delta/(db_0))=1}} \chi(b)  \bigg|^2.
\end{align*}
Recall that by $|\sin (\theta_1-\theta_2)| \gg \nu_3$ we have $|\Delta| \gg \nu_3 N$. For any fixed $D$ with $\nu_3 N \ll D \ll N$ we have (combining the variables $z_2\overline{z_1}=z$, using a divisor bound, and recalling that $\Delta= \Im (z_2\overline{z_1})$)
\[
 \sum_{\substack{|z_1|^2,|z_2|^2 \sim N \\ |\Delta|=D \gg \nu_3 N  }} \frac{1}{\varphi (\ell \Delta/b_0)}  \ll X^{\eta/200}\sum_{\substack{|z|^2 \sim N^2\\ |\Im (z)| = D \gg \nu_3 N  }} \frac{b_0 }{\ell D}  \ll \frac{  X^{\eta/200} b_0 }{\nu_3 \ell },
\]
which gives us (since $\eta_3$ is small compared to $\eta$)
\begin{align*}
    V_{>}(\beta,\bm{N},\bm{\theta})  \ll \sup_{L \ll X} \frac{X^{\eta/100}}{L} \sum_{\substack{b_0 \\ a_0| b_0  }} \sum_{\substack{\ell \sim L \\ \ell \equiv 0 \, (a_0)}}  b_0  \sum_{\substack{D \ll N \\ D \equiv 0 \, (b_0 \ell/a_0)}} 
 \sum_{\substack{d| \ell D/b_0 \\ db_0 > X^{\delta+\eta/2}}}\sum_{(c,d)=1} \\
\times \sideset{}{^\ast}\sum_{\chi\, (d)}   \bigg|  \sum_{\substack{cb_0b \in B \\ (cb,\ell D/(db_0))=1}} \chi(b)  \bigg|^2.
\end{align*}
By a dyadic partition we get (denoting $D'=\ell D/(db_0) $ 
\begin{align*}
  V_{>}(\beta,\bm{N},\bm{\theta})  \ll    & \sup_{\substack{B_0,D_0,L  \\ X^{\delta+\eta/2} \ll D_0 B_0 \ll LN}} \frac{X^{\eta/50}  B_0}{L}  \sum_{\substack{b_0 \sim B_0 \\ a_0| b_0  }} \sum_{\substack{\ell \sim L \\ \ell \equiv 0 \, (a_0)}}  \sum_{\substack{D' \sim L N / B_0 D_0   \\ D'  \equiv 0 \, (\ell^2/a_0)}}  
 \\
& \hspace{100pt} \times\sum_{d \sim D_0} \sum_{(c,d)=1} \sideset{}{^\ast}\sum_{\chi\, (d)}   \bigg|  \sum_{\substack{cb_0b \in B \\ (cb,D')=1}} \chi(b)  \bigg|^2 \\
  \ll  &   \sup_{\substack{B_0,D_0,L  \\ X^{\delta+\eta/2} \ll D_0 B_0 \ll LN}} \frac{X^{\eta/50}  B_0}{L}  \sum_{\substack{b_0 \sim B_0 \\ a_0| b_0  }} \sum_{\substack{\ell \sim L \\ \ell \equiv 0 \, (a_0)}}  \sum_{\substack{D' \sim L N / B_0 D_0   \\ D'  \equiv 0 \, (\ell^2/a_0)}} \\
  & \hspace{100pt} \times \sup_{D''}\sum_{d \sim D_0} \sum_{(c,d)=1} \sideset{}{^\ast}\sum_{\chi\, (d)}   \bigg|  \sum_{\substack{cb_0b \in B \\ (cb,D'')=1}} \chi(b)  \bigg|^2 \\\\ 
  \ll & X^{\eta/40} N \sup_{\substack{ B_0,D_0,L  \\ X^{\delta+\eta/2} \ll D_0 B_0 \ll LN}} \frac{1}{D_0 L} \sum_{b_0 \sim B_0 } \sum_c \sup_{D''} \sum_{d \sim D_0}   \sideset{}{^\ast}\sum_{\chi\, (d)}   \bigg|  \sum_{\substack{cb_0b \in B \\ (cb,D'')=1}} \chi(b)  \bigg|^2.
\end{align*}
By the large sieve for multiplicative characters  (Lemma \ref{largesievelemma}) we obtain
\begin{align*}
    V_{>}(\beta,\bm{N},\bm{\theta}) \ll X^{\eta/30}  N (N +  X^{1/2-\delta-\eta/2} ) |B| \,  \ll X^{-\eta/4}  N |B|^2
\end{align*}
since $N \leq X^{-\eta} |B|$. This is sufficient for Lemma \ref{maintypeiilemma} since $\eta_1,\eta_2$ are small compared to $\eta$.

To bound the contribution from the small $d b_0$ recall that
\begin{align*}
 V_{\leq}(\beta,\bm{N},\bm{\theta})  =   \sum_{b_0,\ell } \mu(\ell) & \sum_{\substack{|z_1|^2, |z_2|^2 \sim N \\ (z_1,z_2) = 1 \\ \Delta \equiv 0 \, ([b_0,\ell])}} \beta^\flat_{z_1,1}  \overline{\beta^\flat_{z_2,2} }  \\
& \times \frac{1}{\varphi (\ell \Delta/b_0)} \sum_{\substack{ d| \ell \Delta/b_0 \\ db_0 \leq X^{\delta+\eta/2}}} \sideset{}{^\ast}\sum_{\chi\, (d)} \chi(a)  \bigg(  \sum_{\substack{b_0b'_1,b_0b'_2 \in B \\  (b'_1,b'_2)=1\\ (b_1'b_2', \ell \Delta/(db_0))=1}} \chi(b'_1) \overline{\chi}(b'_2) \bigg).  
\end{align*}
We expand the conditions $(b_1'b_2',\ell \Delta/(db_0))=1$ by using the M\"obius function to get
\begin{align*}
 V_{\leq}(\beta,\bm{N},\bm{\theta}) =   \sum_{(e_1,e_2)=1} &\mu (e_1) \mu(e_2)\sum_{b_0,\ell } \mu(\ell)\sum_{\substack{|z_1|^2, |z_2|^2 \sim N \\ (z_1,z_2) = 1 \\ \Delta \equiv 0 \, ([b_0,\ell])}} \beta^\flat_{z_1,1}  \overline{\beta^\flat_{z_2,2} }  \\
&  \times \frac{1}{\varphi (\ell \Delta/b_0)} \sum_{\substack{de_1e_2|\ell  \Delta/b_0 \\ d b_0\leq X^{\delta+\eta/2}}} \sideset{}{^\ast}\sum_{\chi\, (d)} \chi(a)  \bigg(  \sum_{\substack{b_0e_1 b'_1,b_0e_2 b'_2 \in B \\  (b'_1,b'_2)=1 \\ e_j|b_j'}} \chi(e_1b''_1) \overline{\chi}(e_2b''_2) \bigg).  
\end{align*}
We have
\begin{align*}
  \frac{1}{\varphi (\ell \Delta/b_0)} &= \frac{b_0}{\ell |\Delta|}  \prod_{p| \ell \Delta/b_0} \bigg(1+\frac{1}{p-1}\bigg) = \frac{b_0}{\ell |\Delta|}  \sum_{\substack{ d_1 | \ell\Delta/b_0}} \frac{|\mu(d_1)|}{\varphi(d_1)} 
\end{align*}
Since $z_1,z_2$ are restricted to polar boxes, we have
\begin{align} \label{Ddefinition}
  |\Delta| = N_1^{1/2} N_2^{1/2} |\sin (\theta_2-\theta_2)| (1+O(\nu_2))  := D(1+O(\nu_2)).
\end{align}
Pluggin this in we get
\begin{align} \label{werrorterm}
 V_{\leq}(\beta,\bm{N},\bm{\theta}) =    V'_{\leq}(\beta,\bm{N},\bm{\theta})+O(W_{\leq }(\beta,\bm{N},\bm{\theta})) 
\end{align}
with
\begin{align*}
    V'_{\leq}(\beta,\bm{N},\bm{\theta}) := &\frac{1}{D} \sum_{\substack{ d_1 }} \frac{|\mu(d_1)|}{\varphi(d_1)} \sum_{(e_1,e_2)=1} \mu (e_1) \mu(e_2)\sum_{b_0,\ell } \frac{b_0\mu(\ell)}{\ell}\sum_{\substack{|z_1|^2, |z_2|^2 \sim N \\ (z_1,z_2) = 1 \\ \Delta \equiv 0 \, ([b_0,\ell])}} \beta^\flat_{z_1,1}  \overline{\beta^\flat_{z_2,2} }  \\
&  \times \sum_{\substack{ db_0 \leq X^{\delta+\eta/2} \\ [d_1,de_1e_2]|\ell  \Delta/b_0}} \sideset{}{^\ast}\sum_{\chi\, (d)} \chi(a)  \bigg(  \sum_{\substack{b_0e_1 b'_1,b_0e_2 b'_2 \in B \\  (b'_1,b'_2)=1 \\ e_j|b_j'}} \chi(e_1b''_1) \overline{\chi}(e_2b''_2) \bigg),
\end{align*}
\begin{align*}
  W_{\leq }(\beta,\bm{N},\bm{\theta}): =& \frac{\nu_2}{D} \sum_{\substack{ d_1 }} \frac{1}{\varphi(d_1)} \sum_{(e_1,e_2)=1}\sum_{b_0,\ell } \frac{b_0}{\ell}\sum_{\substack{|z_1|^2, |z_2|^2 \sim N \\ (z_1,z_2) = 1 \\ \Delta \equiv 0 \, ([b_0,\ell])}} |\beta^\flat_{z_1,1}  \overline{\beta^\flat_{z_2,2} }|  \\
&  \times \sum_{\substack{ db_0 \leq X^{\delta+\eta/2} \\ [d_1,de_1e_2]|\ell  \Delta/b_0}} \bigg|\sideset{}{^\ast}\sum_{\chi\, (d)} \chi(a)  \bigg(  \sum_{\substack{b_0e_1 b'_1,b_0e_2 b'_2 \in B \\  (b'_1,b'_2)=1 \\ e_j|b_j'}} \chi(e_1b''_1) \overline{\chi}(e_2b''_2) \bigg)\bigg|.   
\end{align*}

Let us first consider $V'_{\leq}(\beta,\bm{N},\bm{\theta})$. We note that the sums over $d_1,e_2,e_2,\ell$ converge quickly, so that we expect to be able to bound the contribution from large ranges of $d_1,e_2,e_2$ by crude estimates.  We have
\[
\frac{\ell}{(b_0,\ell)} \,  | \,\Delta/b_0 \quad \text{and} \quad \frac{[d_1,de_1e_2]}{([d_1,de_1e_2],\ell)} \, |\, \Delta/b_0.
\]
Denoting
\[
f:=f(\ell,d_1,d,e_1,e_2) :=  b_0 
 \cdot \bigg[\frac{[d_1,de_1e_2]}{([d_1,de_1e_2],\ell)} , \frac{\ell}{(b_0,\ell)}\bigg] 
\]
we get $f | \Delta$, which is equivalent to saying that
\[
z_2 \equiv a z_1 \, (f)
\]
for some $a \, (f)$. Note that $|f| \ll N$. Thus,
\begin{align*}
   V'_{\leq}(\beta,\bm{N},\bm{\theta}) = &\frac{1}{D}\sum_{\substack{ d_1 }} \frac{|\mu(d_1)|}{\varphi(d_1)} \sum_{\substack{(e_1,e_2)=1 }} \mu (e_1) \mu(e_2)\sum_{b_0,\ell} \frac{b_0\mu(\ell)}{\ell} \\
    & \times  \sum_{d b_0 \leq X^{\delta+\eta/2}} \sum_{a \, (f)}    \sideset{}{^\ast}\sum_{\chi\, (d)} \chi(a)   \bigg(  \sum_{\substack{b_0b'_1,b_0b'_2 \in B \\  (b'_1,b'_2)=1 \\  e_j|b_j'}} \chi(b'_1) \overline{\chi}(b'_2) \bigg)  \sum_{\substack{z_2 \equiv a z_1 \, (f) \\ (z_1,z_2)=1}} \beta^\flat_{z_1,1} \overline{\beta^\flat_{z_2,2} }. 
\end{align*}
We write
\begin{align} \label{Eerrorterm}
    V'_{ \leq}(\beta,\bm{N}, \bm{\theta}) = V''_{\leq}(\beta,\bm{N}, \bm{\theta}) + E_{ \leq}(\beta,\bm{N}, \bm{\theta}),
\end{align}
where $E_{\leq}(\beta,\bm{N}, \bm{\theta})$
is the part where $f > X^{\delta +\eta}$. 

For $f \leq X^{\delta + \eta}$  we note by $(z_1,z_2)=1$ we have $(z_1z_2,f)=1$. By dropping the condition $(z_1,z_2)=1$ we get
\begin{align} \label{zerrorterm}
   V''_{ \leq}(\beta,\bm{N}, \bm{\theta})  = V'''_{ \leq}(\beta,\bm{N}, \bm{\theta}) + O(Z_{ \leq}(\beta,\bm{N}, \bm{\theta}))
\end{align}
with
\begin{align*}
   V'''_{\leq}(\beta,\bm{N},\bm{\theta}) = &\frac{1}{D}\sum_{\substack{ d_1 }} \frac{|\mu(d_1)|}{\varphi(d_1)} \sum_{\substack{(e_1,e_2)=1 }} \mu (e_1) \mu(e_2)\sum_{b_0,\ell} \frac{b_0\mu(\ell)}{\ell} \\
    & \times  \sum_{\substack{ d b_0 \leq X^{\delta+\eta/2} \\ f \leq X^{\delta+\eta}}} \sum_{a \, (f)}    \sideset{}{^\ast}\sum_{\chi\, (d)} \chi(a)   \bigg(  \sum_{\substack{b_0b'_1,b_0b'_2 \in B \\  (b'_1,b'_2)=1 \\  e_j|b_j'}} \chi(b'_1) \overline{\chi}(b'_2) \bigg)  \sum_{\substack{z_2 \equiv a z_1 \, (f) \\ (z_1z_2,f)=1}} \beta^\flat_{z_1,1} \overline{\beta^\flat_{z_2,2} }. 
\end{align*}
and
\begin{align*}
    Z_{ \leq}(\beta,\bm{N}, \bm{\theta}) := &\frac{1}{D}\sum_{\substack{ d_1 }} \frac{1}{\varphi(d_1)} \sum_{\substack{(e_1,e_2)=1 }} \sum_{b_0,\ell} \frac{b_0}{\ell} \\
    & \times  \sum_{\substack{ d b_0 \leq X^{\delta+\eta/2} \\ f \leq X^{\delta+\eta}}} \sum_{a \, (f)}    \bigg|\sideset{}{^\ast} \sum_{\chi\, (d)} \chi(a)   \bigg(  \sum_{\substack{b_0b'_1,b_0b'_2 \in B \\  (b'_1,b'_2)=1 \\  e_j|b_j'}} \chi(b'_1) \overline{\chi}(b'_2) \bigg)\bigg|  \sum_{\substack{z_2 \equiv a z_1 \, (f) \\ (z_1z_2,f)=1\\ |(z_1,z_2)|^2 \geq W }} |\beta^\flat_{z_1,1} \overline{\beta^\flat_{z_2,2} } |
\end{align*}

By expanding $z_2 \equiv a z_1 \, (f)$ with Dirichlet characters we have
\begin{align*}
V'''_{\leq}(\beta,\bm{N})  \leq & \frac{1}{D} \sum_{\substack{ d_1 }} \frac{1}{d_1} \sum_{\substack{(e_1,e_2)=1}} \sum_{b_0, \ell } \frac{b_0}{\ell}\sum_{d b_0\leq X^{\delta+\eta/2}} \sum_{\substack{a \, (f)}}  \bigg| \sideset{}{^\ast}\sum_{\chi\, (d)} \chi(a)   \bigg(  \sum_{\substack{b_0b_1',b_0b_2' \in B \\  (b_1',b_2')=1 \\ e_j| b_j'}} \chi(b_1') \overline{\chi}(b_2') \bigg)   \bigg|\\
&\times\frac{\mathbf{1}_{f \leq X^{\delta+\eta}}}{\varphi_{\ZZ[i]}(f)}\sum_{\psi \in \widehat{(\ZZ[i]/f \ZZ[i])^\times}}  \bigg|\sum_{z_1,z_2} \beta^\flat_{z_1,1} \psi (z_1) \overline{\beta^\flat_{z_2,2}   \psi (z_2)} \bigg|
\end{align*}
By applying Cauchy-Schwarz on $\psi$ and using Lemma \ref{stronggallagherlemma} we get
\begin{align*}
|V'''_{\leq}(\beta,\bm{N}) | \ll \frac{\nu_1^2\nu_2^2 N^2}{D (\log X)^{C_1}}&\sum_{\substack{d_1,e_1,e_2 }} \sum_{b_0 ,\ell }\sum_{db_0 \leq X^{\delta+\eta/2}} \frac{b_0}{d_1 \ell f^2}  \\
&\times\sum_{\substack{a \, (f) }}    \bigg|\sideset{}{^\ast}\sum_{\chi\, (d)} \chi(a)  \bigg(  \sum_{\substack{b_0b_1',b_0b_2' \in B \\  (b'_1,b'_2)=1 \\ e_j | b_j'}} \chi(b'_1) \overline{\chi}(b'_2) \bigg)  \bigg|.
\end{align*}
By Lemma \ref{primitivecharactersum} (with the argument $ab'_1/b'_2$) and Lemma \ref{gcdlemma} this is bounded by 
\begin{align*}
 \ll \frac{\nu_1^2\nu_2^2 N^2}{D (\log X)^{C_1}} &\sum_{d_1,b_0,\ell } \sum_{\substack{b_0 b'_1,b_0 b'_2 \in B  \\(b'_1,b'_2)=1}}\sum_{\substack{e_1|b'_1 \\ e_2|b'_2}}  \sum_{\substack{ d b_0\leq X^{\delta+\eta/2} \\ (b'_1b'_2,d)=1 }}\frac{b_0}{d_1\ell f^2} \sum_{\substack{ a \, (f)\\(a,d)=1}} (d,b'_2-ab'_1) \\
 & \ll  \frac{ \nu_1^2 \nu_2^2 N^2}{D(\log X)^{C_1}} \sum_{d_1,b_0,\ell } \sum_{\substack{b_0 b'_1,b_0 b'_2 \in B  \\(b'_1,b'_2)=1}}\sum_{\substack{e_1|b'_1 \\ e_2|b'_2}}  \sum_{\substack{ d b_0\leq X^{\delta+\eta/2} \\ (b_1b_2,d)=1 }}\frac{\tau(d)b_0}{d_1\ell f}  \\
\end{align*}
Using
\[
f =   b_0 
 \cdot \bigg[\frac{[d_1,de_1e_2]}{([d_1,de_1e_2],\ell)} , \frac{\ell}{(b_0,\ell)}\bigg] \geq \frac{b_0 d}{(d,\ell)}
\]
we get by Lemma \ref{gcdlemma} and \eqref{Ddefinition}
\begin{align*}
 |V'''_{\leq}(\beta,\bm{N}) | & \ll  \frac{\nu_1^2\nu_2^2 N^2}{D(\log X)^{C_1}}\sum_{\substack{b_1,b_2 \in B }} \sum_{\substack{ d,\ell \leq X }}\frac{\tau(d) (\ell,d)}{d \ell}  \\
& \ll   \frac{\nu_1^2\nu_2^2 N|B|^2}{|\sin(\theta_2-\theta_1)|(\log X)^{C_1}}.
\end{align*}
which is sufficient.

The argument for bounding $ Z_{ \leq}(\beta,\bm{N}, \bm{\theta}) $ from \eqref{zerrorterm} is the same except that instead of expanding into Hecke characters and Lemma \ref{stronggallagherlemma} we use the trivial  bound
\begin{align*}
    \sum_{\substack{z_2 \equiv a z_1 \, (f) \\ (z_1z_2,f)=1 \\ |(z_1,z_2)|^2 \geq W }} |\beta^\flat_{z_1,1} \overline{\beta^\flat_{z_2,2} } | =& \sum_{|z_0|^2 \geq  W} \sum_{\substack{z'_2 \equiv a z'_1 \, (f) \\ (z_1z_2,f)=1  }} |\beta^\flat_{z_0z_1,1} \overline{\beta^\flat_{z_0z_2,2} } |  \\
    \ll& (\log X)^{O(1)} \sum_{ W < |z_0|^2  \ll N} \bigg( \frac{\nu_1 \nu_2 N}{|z_0|^2} + 1\bigg)\bigg( \frac{\nu_1 \nu_2 N}{|z_0|^2 f^2} + 1\bigg) \\
    \ll& (\log X)^{O(1)} \bigg(\frac{\nu_1^2 \nu_2^2 N^2}{W f^2}+ \nu_1 \nu_2 N (\log X)  + N \bigg)\\
    \ll &  (\log X)^{O(1)}  \frac{\nu_1^2 \nu_2^2 N^2}{W f^2},
\end{align*}
where the last bound holds since $N > X^{3\delta+3\eta}$ and $f^2 \leq X^{\delta+2\eta}$. This gives us
\[
Z_{ \leq}(\beta,\bm{N}, \bm{\theta})  \ll  (\log X)^{O(1)}  \frac{\nu_1^2\nu_2^2 N |B|^2}{|\sin \theta_2-\theta_1| W },
\]
which suffices for Lemma \ref{maintypeiilemma}.

For $E_{\leq}(\beta,\bm{N},\bm{\theta})$  from \eqref{Eerrorterm} with large $f$ we'll need to use a slightly different argument since the modulus $f$ can be as large as $N$ which would make $f^2$ much bigger than $N$. We write
\begin{align*}
     E_{\leq}(\beta,\bm{N},\bm{\theta}) = &\frac{1}{D}\sum_{\substack{ d_1 }} \frac{|\mu(d_1)|}{\varphi(d_1)} \sum_{\substack{(e_1,e_2)=1 }} \mu (e_1) \mu(e_2)\sum_{b_0,\ell} \frac{b_0\mu(\ell)}{\ell} \\
    & \times  \sum_{\substack{d b_0 \leq X^{\delta+\eta/2} \\ f > X^{\delta + \eta}}} \sum_{a \, (f)}    \sideset{}{^\ast}\sum_{\chi\, (d)} \chi(a)   \bigg(  \sum_{\substack{b_0b'_1,b_0b'_2 \in B \\  (b'_1,b'_2)=1 \\  e_j|b_j'}} \chi(b'_1) \overline{\chi}(b'_2) \bigg)  \sum_{z_2 \equiv a z_1 \, (f)} \beta^\flat_{z_1,1} \overline{\beta^\flat_{z_2,2} } \\
    \ll&  \frac{1}{D}\sum_{\substack{ d_1 }} \frac{1}{\varphi(d_1)} \sum_{\substack{e_1,e_2 }} \sum_{b_0,\ell} \frac{b_0}{\ell} \\
    & \times  \sum_{\substack{d b_0 \leq X^{\delta+\eta/2} \\ f > X^{\delta + \eta}}}  \sideset{}{^\ast}\sum_{\chi\, (d)}\bigg|   \sum_{\substack{b_0b'_1,b_0b'_2 \in B \\  (b'_1,b'_2)=1 \\  e_j|b_j'}} \chi(b'_1) \overline{\chi}(b'_2)  \bigg|   \sum_{\substack{z_1,z_2 \\ f| \Delta }} |\beta^\flat_{z_1,1} \overline{\beta^\flat_{z_2,2} } | \\
    \ll & \frac{N^2}{D} \sum_{\substack{ d_1 }} \frac{1}{\varphi(d_1)} \sum_{\substack{e_1,e_2}} \sum_{b_0,\ell} \frac{b_0}{\ell} \\
    & \times   \sum_{\substack{d b_0 \leq X^{\delta+\eta/2} \\ f > X^{\delta + \eta}}} \frac{1}{|f|} \sideset{}{^\ast}\sum_{\chi\, (d)}\bigg|   \sum_{\substack{b_0b'_1,b_0b'_2 \in B \\  (b'_1,b'_2)=1 \\  e_j|b_j'}} \chi(b'_1) \overline{\chi}(b'_2) \bigg|.
\end{align*}
By using $f > X^{\delta + \eta}$ and $db_0 \leq X^{\delta + \eta/2}$ we get $f^{-1} \leq X^{-\eta/2} (b_0 d)^{-1}$, which gives us
\begin{align*}
E_{\leq}(\beta,\bm{N},\bm{\theta})  \ll& X^{-\eta/2}  \frac{N^2}{D} \sum_{\substack{ d_1  \ll X}} \frac{1}{\varphi(d_1)} \sum_{\substack{e_1,e_2 }}\sum_{b_0,\ell \ll X} \frac{1}{\ell} \\
    & \times   \sum_{\substack{db_0 \leq X^{\delta+\eta} }} \frac{1}{d} \sideset{}{^\ast}\sum_{\chi\, (d)}\bigg|   \sum_{\substack{b_0b'_1,b_0b'_2 \in B \\  (b'_1,b'_2)=1 \\  e_j|b_j'}} \chi(b'_1) \overline{\chi}(b'_2) \bigg|
\end{align*}
By using the expansion
\[
\mathbf{1}_{(b_1',b_2')=1} = \sum_{c|(b_1',b_2')} \mu(c),
\] 
Cauchy-Schwarz,  orthogonality of characters, and a divisor bound we have
\begin{align*}
\sideset{}{^\ast}\sum_{\chi\, (d)}& \sum_{e_1, e_2} \bigg|   \sum_{\substack{b_0b'_1,b_0b'_2 \in B \\  (b'_1,b'_2)=1 \\  e_j|b_j'}} \chi(b'_1) \overline{\chi}(b'_2) \bigg| \leq  \sum_{c}  \sum_{\chi\, (d)} \sum_{e_1 ,e_2} \bigg|   \sum_{\substack{b_0b'_1,b_0b'_2 \in B  \\  e_j|b_j' \\ c|b_j'}} \chi(b'_1) \overline{\chi}(b'_2) \bigg| \\
   \ll_\eps & X^{\eps}   \bigg(  \sum_{e_1 ,e_1'} \sum_{\substack{b_0 b_1,b_0b_1' \in B \\ e_1| b_1 \\ e_1'|b_1' }} d\mathbf{1}_{b_1 \equiv b_1' \, (d)}  \bigg)^{1/2}  \bigg(  \sum_{e_2 ,e_2'} \sum_{\substack{b_0 b_2,b_0 b_2' \in B \\ e_2| b_2 \\ e_2'|b_2' }}d \mathbf{1}_{b_2 \equiv b_2' \, (d)}  \bigg)^{1/2} \\
   \ll&  X^{\eta/4}    \sum_{e_1 ,e_1'} \sum_{\substack{b_0 b_1,b_0b_1' \in B \\ e_1| b_1 \\ e_1'|b_1' }} d\mathbf{1}_{b_1 \equiv b_1' \, (d)} ,
\end{align*}
where the last bound follows by symmetry.
Hence, we have
\begin{align*}
    E_{\leq}(\beta,\bm{N},\bm{\theta}) \ll & X^{-\eta/4}  \frac{N^2}{D}
    \sum_{\substack{b_0,d,e_1,e_1' \\ db_0 \leq X^{\delta+\eta/2}   }} \sum_{\substack{ b_0b_1,b_0b_1' \in B \\ e_1| b_1 \\ e_1' |b_1' }} \mathbf{1}_{b_1 \equiv b_1' \, (d)}.
    \end{align*}
The contribution from $b_1'=b_1$ is bounded by
\[
\ll  X^{-\eta/8}  \frac{N^2}{D}|B| X^{\delta+\eta/2} \ll X^{-\eta/8}  \frac{N^2}{D}|B|^2.
\]
The contribution from $b_1'\neq b_1$ is bounded by
\[
\ll X^{-\eta/4}  \frac{N^2}{D}
    \sum_{\substack{b_0,e_1,e_1'   }} \sum_{\substack{ b_0b_1,b_0b_1' \in B \\ e_1| b_1 \\ e_1' |b_1' }} \mathbf{1}_{b_1 \neq b_1'} \tau(b_1'-b_1)  \ll X^{-\eta/8}  \frac{N^2}{D}|B|^2
\]
by applying the divisor bound $\tau(n) \ll_\eps n^\eps$. Therefore, by \eqref{Ddefinition} we get (recall that $\nu_3$ is small compared to $\nu$)
\[
 E_{\leq}(\beta,\bm{N},\bm{\theta})  \ll X^{-\eta/10} N|B|^2.
\]

Finally, the error term $W_{\leq }(\beta,\bm{N},\bm{\theta})$  from \eqref{werrorterm} is bounded by exactly the same argument as above except that in the part $f \leq X^{\delta+\eta}$ we use the trivial estimate
\[
\sum_{\substack{|z_1|^2, |z_2|^2 \sim N \\ z_2 \equiv a z_1 \, (f)}}  |\beta^\flat_{z_1,1}| |\beta^\flat_{z_2,2} | \ll (\log X)^{O(1)}\frac{\nu_1^4 N^2}{|f|^2}
\]
instead of expanding into Hecke characters and Lemma \ref{stronggallagherlemma}.
\subsubsection{Proof of Lemma \ref{typeiitrivialboundlemma}}
The argument is exactly the same as in Section \ref{maintypeiilemmasection}, except that in the part $f \leq X^{\delta+\eta}$ we use the trivial estimate
\[
\sum_{\substack{|z_1|^2, |z_2|^2 \sim N \\ z_2 \equiv a z_1 \, (f)}}  |\beta^\flat_{z_1,1}| |\beta^\flat_{z_2,2} | \ll (\log X)^{O(1)}\frac{\nu_1^4 N^2}{|f|^2}
\]
instead of expanding into Hecke characters and Lemma \ref{stronggallagherlemma}.
\subsection{Proof of Lemma \ref{diag1lemma}}
Since $\Delta=0$ we have
\[
z_1b_2 =  z_2b_1.
\]
Multiplying both sides by $\overline{w}$ and taking the imaginary parts we find
\[
a_1b_2=a_2b_1.
\]
Hence, we get
\[
U_{0}(\gamma) \ll  \sum_{\substack{a_1,a_2 \ll X^{1/2} \\ b_1,b_2 \ll X^{1/2} \\ a_1b_2=a_2b_1}} \mathbf{1}_B(b_1)\mathbf{1}_B(b_2)\ll_\eps X^{1/2+\eps} |B|.
\]
by using the divisor bound $\tau(\ell) \ll_\eps \ell^\eps$. 

\subsection{Proof of Lemma \ref{diag2lemma}} \label{diag2section}
Since $(z_1 z_2, P(W)) =1$, having $(z_1,z_2) > 1$ implies $|(z_1,z_2)|^2 \geq W$. Let $z_0 = (z_1,z_2)$, and $z_j=z_0z_j' $. Denoting $w_0=\overline{z_0} w$ and $\Delta'=\Im(\overline{z_1'} z_2')$, we have 
\[
b_j = \Re (\overline{w_0} z_j'),
\]
which implies that
\begin{align} \label{deltaprime}
    i\Delta' w_0 = z_2' b_1-z_1'b_2
\end{align}
so that $w_0$ is fixed once we fix $z_j',b_j$. Furthermore, we have
\[
\sum_{(z_0,P(W))=1} \mathbf{1}_{\overline{z_0}| w_0} \ll_\eps W^{\eps}. 
\]
Note also that $(w_0,\overline{w_0})=1$ implies that $(b_1,b_2)=b_0 | \Delta'$.
Then denoting $b_j=b_0b_j'$ and $a'=z_2'/z_1'$ we have
\begin{align*}
 &U_1(\beta) =  \sum_{\substack{|z_1|^2, |z_2|^2 \sim N \\ \Delta \neq 0 \\ (z_1,z_2) > 1 \\ (z_j',\overline{z_j'})=1}} \beta^\flat_{z_1} \overline{\beta^\flat_{z_2}} \sum_{(w,\overline{w})=1} F_M(|w|^2)\mathbf{1}_B (\Re  (\overline{w} z_1 ) ) \mathbf{1}_B (\Re (\overline{w} z_2 ) ) \\
   & \ll \sup_{W \ll Z \ll N}  \sum_{\substack{|z_1'|^2, |z_2'|^2 \asymp N/Z \\ \Delta \neq 0 \\ (z_1',z_2') = 1\\ (z_j',\overline{z_j'}P(W))=1 }} \sum_{\substack{z_0,w_0 \\(w_0,\overline{w_0})=1 \\ (z_0,P(W))=1 \\ \overline{z_0}|w_0 \\ |z_0|^2 \sim Z }}|\beta^\flat_{z_0z_1'} \overline{\beta^\flat_{z_0z_2'}} |\mathbf{1}_B (\Re  (\overline{w_0} z_1' ) ) \mathbf{1}_B (\Re (\overline{w_0} z_2' ) )  \\
&\ll_\eps W^\eps
\sup_{\substack{W \ll Z \ll N  }}  \sum_{b_0}\sum_{\substack{|z_1'|^2, |z_2'|^2 \sim N/Z \\ \Delta' \neq 0 \\ (z_1',z_2') = 1 \\  (z_j',\overline{z_j'}P(W))=1 \\ \Delta' \equiv 0 \, (b_0)}} \sum_{\substack{b_2' \equiv a b_1'  \, (\Delta'/b_0) \\ (b_1'b_2',\Delta'/b_0)=1\\(b_1',b_2')=1  }}\mathbf{1}_B(b_0b_1') \mathbf{1}_B(b_0b_2').
   \end{align*}
Thus, denoting
\[
V_Z(\beta) := \sum_{b_0}\sum_{\substack{|z_1'|^2, |z_2'|^2 \sim N/Z \\ \Delta' \neq 0 \\ (z_1',z_2') = 1\\  (z_j',\overline{z_j'}P(W))=1 \\  \Delta' \equiv 0 \, (b_0)}} \sum_{\substack{b_2' \equiv a b_1'  \, (\Delta'/b_0) \\ (b_1'b_2',\Delta'/b_0)=1\\(b_1',b_2')=1  }}\mathbf{1}_B(b_0b_1') \mathbf{1}_B(b_0b_2'),
\]
it suffices to show that
\begin{align} \label{Wboundclaim}
    V_Z(\beta)  \ll_\eps W^\eps \frac{N |B|^2}{ W}.
\end{align}
We apply a similar argument as in Section \ref{maintypeiilemma} except that certain parts will be easier by positivity. By expanding into Dirichlet characters we get
\begin{align*}
    V_Z(\beta)  \ll  (\log X) \sum_{b_0,  }&\sum_{\substack{|z'_1|^2, |z'_2|^2 \sim N/Z \\ \Delta'\neq 0\\ (z'_1,z'_2) = 1  \\ (z_j',\overline{z_j'}P(W))=1 \\ \Delta' \equiv 0 \, (b_0)}} \frac{b_0}{|\Delta|}  \sum_{d|\Delta'/b_0} \bigg|\sideset{}{^\ast}\sum_{\chi\, (d)} \chi(a)  \bigg(  \sum_{\substack{b_0b'_1,b_0b'_2 \in B \\  (b'_1,b'_2)=1 \\ (b_1'b_2', \Delta'/(db_0))=1}} \chi(b'_1) \overline{\chi}(b'_2) \bigg) \bigg|.
\end{align*}
We split into the parts $db_0 \leq X^{\delta+\eta/2}$ and $db_0 > X^{\delta+\eta/2}$
\[
V_Z(\beta)  \ll (\log X) (V_{Z,\leq}(\beta)  +  V_{Z,>}(\beta)).
\]

For large $db_0$ we expand the condition $(b_1',b_2')$ to get
\[ 
 V_{Z,>}(\beta) \ll  \sum_c \sum_{b_0,  }\sum_{\substack{|z'_1|^2, |z'_2|^2 \sim N/Z\\ \Delta'\neq 0 \\ (z'_1,z'_2) = 1  \\  (z_j',\overline{z_j'}P(W))=1\\ \Delta' \equiv 0 \, (b_0)}} \frac{b_0}{|\Delta|}  \sum_{\substack{d|\Delta'/b_0 \\ db_0 > X^{\delta+\eta}}} \sideset{}{^\ast}\sum_{\chi\, (d)} \bigg|  \sum_{\substack{cb_0b \in B  \\ (b, \Delta'/(db_0))=1}} \chi(b) \overline{\chi}(b'_2) \bigg|^2
\]
By using the estimate
\begin{align} \label{z1estimate}
 \sum_{\substack{|z'_1|^2, |z'_2|^2 \sim N/Z \\ (z'_1,z'_2) = 1  \\ (z'_j,\overline{z'_j})=1 \\ \Delta' =D}}  1  \ll_\eps X^\eps \frac{N}{Z}   
\end{align}
we get (denoting $D=D' db_0$)
\begin{align*}
  V_{Z,>}(\beta) \ll_\eps & X^\eps  \frac{N}{Z} \sum_c \sum_{b_0,  } b_0 \sum_{db_0 > X^{\delta+\eta}} \sum_{\substack{D \ll N/Z \\ db_0|D}} \frac{1}{D} \sideset{}{^\ast}\sum_{\chi\, (d)} \bigg|  \sum_{\substack{cb_0b \in B  \\ (b, D/(db_0))=1}} \chi(b) \overline{\chi}(b'_2) \bigg|^2    \\
 \ll_\eps &  X^\eps  \frac{N}{Z} \sum_c \sum_{b_0,  } b_0 \sum_{db_0 > X^{\delta+\eta}} \sum_{\substack{D' \ll N/(Zdb_0) }} \frac{1}{D'db_0} \sideset{}{^\ast}\sum_{\chi\, (d)} \bigg|  \sum_{\substack{cb_0b \in B  \\ (b, D')=1}} \chi(b) \overline{\chi}(b'_2) \bigg|^2 \\
 \ll_\eps & X^\eps \frac{N}{Z}  \sum_c \sum_{b_0,  } \sup_{D'} \sum_{X^{\delta+\eta}<db_0 \ll N/Z}\frac{1}{d}\sideset{}{^\ast}\sum_{\chi\, (d)} \bigg|  \sum_{\substack{cb_0b \in B  \\ (b, D')=1}} \chi(b) \overline{\chi}(b'_2) \bigg|^2.
 \end{align*}
By applying the multiplicative large sieve (Lemma \ref{largesievelemma}) similarly as in Section \ref{maintypeiilemmasection} we get
\[
V_{Z,>}(\beta)  \ll  (\log X)^{O(1)}\frac{N |B|^2}{ W},
\]
which suffices for \eqref{Wboundclaim}.

For small $db_0$ we write by Lemma \ref{primitivecharactersum}
\begin{align*}
    V_{Z,\leq}(\beta) &\ll   \sum_{db_0 \leq X^{\delta+\eta}} b_0  \sum_{a \, (db_0)} \sum_{\substack{|z'_1|^2, |z'_2|^2 \sim N/Z\\ \Delta'\neq 0 \\ (z'_1,z'_2) = 1 \\  (z_j',\overline{z_j'}P(W))=1  \\  z'_2 \equiv a z'_1 \, (d b_0)}}  \frac{1}{|\Delta'|} \bigg| \sideset{}{^\ast}  \sum_{\chi\, (d)} \chi(a)  \bigg(  \sum_{\substack{b_0b'_1,b_0b'_2 \in B \\  (b'_1,b'_2)=1 \\ (b_1'b_2', \Delta'/(db_0))=1}} \chi(b'_1) \overline{\chi}(b'_2) \bigg) \bigg|  \\
    &\ll  \sum_{db_0 \leq X^{\delta+\eta}} b_0  \sum_{a \, (db_0)}  \sum_{\substack{b_0b'_1,b_0b'_2 \in B \\  (b'_1,b'_2)=1 \\ (b_1'b_2',d)=1 }} (d,b_2'-ab_1') \sum_{\substack{|z'_1|^2, |z'_2|^2 \sim N/Z\\ \Delta'\neq 0 \\ (z'_1,z'_2) = 1 \\  (z_j',\overline{z_j'}P(W))=1  \\  z'_2 \equiv a z'_1 \, (d b_0)}}  \frac{1}{|\Delta'|}  
\end{align*}
We note that $z'_2 \equiv a z'_1 \, (db_0)$ implies that for $z=z_2\overline{z_1}$ we have $z \equiv \overline{z} \, (db_0)$, that is, denoting $z=r+is$ we get $s\equiv 0 \, (db_0)$. Thus, using the divisor bound
\[
\sum_{\substack{z= z_2 \overline{z_1} \\ (z_j,P(W))=1}} 1 \ll_\eps W^{\eps}
\]
we get
\begin{align}
 \sum_{\substack{|z'_1|^2, |z'_2|^2 \sim N/Z\\ \Delta'\neq 0 \\ (z'_1,z'_2) = 1 \\  (z_j',\overline{z_j'}P(W))=1  \\  z'_2 \equiv a z'_1 \, (d b_0)}}  \frac{1}{|\Delta'|} \ll_\eps &W^\eps \sum_{\substack{|r_1|, |r_2| \ll (N/Z)^{1/2} \\r_2 \equiv a r_1 \, (db_0)}} \sum_{\substack{0 < |s| \ll N/Z  \\ s \equiv 0\, (db_0)}} \frac{1}{s}   \nonumber
 \\   \nonumber
 \ll_\eps& W^{\eps}\bigg(\frac{N}{Z d^2 b_0^2}  +  \frac{N^{1/2}}{Z^{1/2} db_0}\bigg)  \\ \label{z2estimate}
 \ll_\eps&  W^{\eps}\frac{N}{W d^2 b_0^2}, 
\end{align}
where the last step follows from $N > X^{3\delta+3\eta}$ and $d b_0 \leq X^{\delta+\eta}$. Therefore, we get by Lemma \ref{gcdlemma}
\begin{align*}
  V_{Z,\leq}(\beta) &\ll   W^{\eps}\frac{N}{W } \sum_{db_0 \leq X^{\delta+\eta}}   \sum_{\substack{b_0b'_1,b_0b'_2 \in B \\  (b'_1,b'_2)=1 \\ (b_1'b_2',d)=1}} \frac{1}{d^2 b_0}\sum_{a \, (db_0)} 
 (d,b_2'-ab_1')    \\
 & \ll_\eps W^{\eps}\frac{N |B|}{W} ,
\end{align*}
which completes the proof of \eqref{Wboundclaim}.

\subsection{Proof of
 Lemma \ref{pseudodiaglemma}}
By recombining the finer-than-dyadic partitions for  $\theta_1,\theta_2$, we get 
\[
V_{\leq \nu_3}(\beta) \leq U_{\ll \nu_3}(\beta)
\]
with
\[
U_{\ll \nu_3}(\beta) :=  \sum_{\substack{|z_1|^2, |z_2|^2 \sim N \\ 0 <|\Delta| \ll \nu_3 N \\ (z_1,z_2) = 1}} |\beta^\flat_{z_1} \overline{\beta^\flat_{z_2}}| \sum_{(w,\overline{w})=1} F_M(|w|^2)\mathbf{1}_B (\Re (\overline{w} z_1 ) ) \mathbf{1}_B (\Re  (\overline{w} z_2 ) ) 
\]
Similarly as in Section \ref{offdiagsection}, we write
\[
U_{\ll\nu_3}(\beta) \ll (\log X)^{O(1)}\sum_{b_0}\sum_{\substack{|z_1|^2, |z_2|^2 \sim N \\ 0 < |\Delta| \ll \nu_3 N \\ (z_1,z_2) = 1 \\(z_j',\overline{z_j'})=1\\ \Delta \equiv 0\, (b_0)}} \sum_{\substack{b'_2 \equiv a b'_1 \, (\Delta/b_0) \\ (b_1'b_2',\Delta/b_0)=1\\(b_1',b_2')=1}} \mathbf{1}_B(b_0 b_1')\mathbf{1}_B(b_0 b_2')F_M\bigg(\bigg|b_0 \frac{z_2b_1'-z_1b_2'}{\Delta}\bigg|^2\bigg).
\]
Using $\Delta \ll \nu_2 N$ and $MN \sim X$ we see that the smooth weight $F_M$ restricts $z_2b_1-z_1b_2$ to a small disc, that is, 
\[
U_{\ll \nu_3}(\beta) \ll   
 \sum_{b_0}\sum_{\substack{|z_1|^2, |z_2|^2 \sim N \\ 0 < |\Delta| \ll  \nu_3 N \\ (z_1,z_2) = 1 \\(z_j',\overline{z_j'})=1\\ \Delta \equiv 0\, (b_0)}} \sum_{\substack{b'_2 \equiv a b'_1 \, (\Delta/b_0) \\ (b_1'b_2',\Delta/b_0)=1\\(b_1',b_2')=1}} \mathbf{1}_B(b_0 b_1')\mathbf{1}_B(b_0 b_2')\mathbf{1}_{|z_2b_0b_1'-z_1 b_0 b_2'| \ll \nu_3 N^{1/2} X^{1/2}}.
\]
Recall now that by \eqref{Bintervalassumption} we have $B \subseteq [Y,Y+X^{1/2-\eta}]$ for some $Y \asymp  X^{1/2}$. Hence, we obtain
\[
 U_{\ll \nu_3}(\beta) \ll \sum_{b_0}\sum_{\substack{|z_1|^2, |z_2|^2 \sim N \\ |\Delta| \neq 0 \\ (z_1,z_2) = 1 \\ (z_j',\overline{z_j'})=1\\ \Delta \equiv 0\, (b_0)}} \sum_{\substack{b'_2 \equiv a b'_1 \, (\Delta/b_0) \\ (b_1'b_2',\Delta/b_0)=1\\(b_1',b_2')=1}} \mathbf{1}_B(b_0 b_1')\mathbf{1}_B(b_0 b_2')\mathbf{1}_{|z_2-z_1| \ll \nu_3 N^{1/2} }.
\]
This can now be bounded by the same argument as in Section \ref{diag2section}, 
replacing the bounds \eqref{z1estimate} and \eqref{z2estimate}
respectively by
\begin{align*}
 \sum_{\substack{|z_1|^2, |z_2|^2 \sim N \\ (z_1,z_2) = 1  \\ (z_j,\overline{z_j})=1 \\ \Delta =D}}  \mathbf{1}_{|z_2-z_1| \ll \nu_3 N^{1/2}}  \ll_\eps X^\eps \nu_3 N  
\end{align*}
and
\begin{align*}
     \sum_{\substack{|z_1|^2, |z_2|^2 \sim N\\ \Delta\neq 0 \\ (z_1,z_2) = 1  \\  z_2 \equiv a z_1 \, (d b_0)}}  \frac{ \mathbf{1}_{|z_2-z_1| \ll \nu_3 N^{1/2}} }{|\Delta|} \ll_\eps &X^\eps \sum_{\substack{|r_1|, |r_2| \ll (N/Z)^{1/2}\\ |r_2-r_1| \ll \nu_3 N^{1/2} \\r_2 \equiv a r_1 \, (db_0)}} \sum_{\substack{0 < |s| \ll \nu_3 N  \\ s \equiv 0\, (db_0)}} \frac{1}{s}   
 \\  
 \ll_\eps& X^{\eps}\bigg(\frac{\nu_3 N}{ d^2 b_0^2}  +  \frac{\nu_3 N^{1/2}}{db_0}\bigg)  \\ \label{z2estimate}
 \ll_\eps&  X^{\eps}\frac{\nu_3 N}{d^2 b_0^2}.
\end{align*}
We get
\[
U_{\ll \nu_3}(\beta) \ll_\eps  X^\eps \nu_3 N |B|^2,
\]
which gives us Lemma \ref{pseudodiaglemma}.
\begin{remark} \label{smallBremark}
 Without the assumption $B \subset [\eta_1 X^{1/2},(2-\eta)X^{1/2}]$ we would need to take $\nu_3=X^{-\delta-\eps}$ to get savings in this argument, since it is possible that $Y \asymp X^{1/2-\delta}$. This means that in the approximation for $\beta_\n$ we need to track the distribution of $\beta_\n$ in sectors with an angle $X^{-\delta-\eta}$. It is possible to do so by a more careful argument but we do not pursue this issue here. For this we also note that the smooth weight
 \[
 F_M\bigg( \bigg|\frac{b_1z_2-b_2z_1}{\Delta}\bigg|^2\bigg)  =F_{M/Y^2}\bigg( \bigg|\frac{z_2-z_1}{\Delta}\bigg|^2\bigg) + O(X^{-\eta'})
 \]
 could be handled more efficiently in terms of the dependency on $\arg z_j$ since it is a function of the difference $\arg z_2 -\arg z_1$, so that we need an expansion to $\xi_k$ only once instead of twice.  Note that $B$ cannot simultaneously be multiples of a large fixed $q$ and restricted to a narrow interval. Thus, for this extension the condition $\M(u) \leq X^{\delta+\eta}$ ought to be replaced by $|k| \M(u) \leq X^{\delta+\eta}$.
\end{remark}

\section{Type II information: proof of Proposition \ref{sharptypeiiprop}} \label{sharptypeiipropsection}
Recall that we are trying to evaluate
\begin{align*}
    &\sum_{\N \m\sim M} \sum_{\N \n \sim N} \alpha_\m \beta^\sharp_\n a_{\m\n} \\
    &= \sum_{N'} \sum_{\N \m \sim M}\sum_{\n} H_{N'}(\n) \alpha_\m \mathbf{1}_{(\n,\overline{\n} P(W))=1} \bigg(\sum_{j \leq J} \overline{\xi_{k_j}\chi_j}(\n)\C_W(\beta, \overline{\xi_{k_j}\chi_j}  H_{N'})  \bigg)a_{\m\n} \\
    &=: \sum_{ j \leq J} S_j.
\end{align*}
The condition $(\n, \overline{\n})$ may be dropped since we have $(z,\overline{z})=1$ in the definition of $a_z$. 
We have
\begin{align*}
  S_j &=  \sum_{N'}   \C_W(\beta, \overline{\xi_{k_j}\chi_j} H_{N'}) \sum_{\N \m \sim M}\sum_{\n} H_{N'}(\n) \alpha_\m \mathbf{1}_{(\n, P(W))=1} \overline{\xi_{k_j}\chi_j}(\n)  a_{\m\n} \\
 & = \sum_{N'}   \C_W(\beta, \overline{\xi_{k_j}\chi_j}  H_{N'})\sum_{\N \m \sim M}\sum_{\n} H_{N'}(\n) \alpha_\m \xi_{k_j}\chi_j(\m)\mathbf{1}_{(\n, P(W))=1} \overline{\xi_{k_j}\chi_j}(\m\n)a_{\m\n}.
\end{align*}
Here for any $C>0$
\begin{align*}
  \C_W(\beta, \overline{\xi_{k_j}\chi_j} H_{N'}) &= \bigg(\sum_{(\n,\overline{\n})=1 } H_{N'}(\n) \beta_\n \xi_{k_j}\chi_j (\n)\bigg) \bigg( \sum_{(\n,\overline{\n} P(W))=1 } H_{N'}(\n) |\xi_{k_j}\chi_j (\n)|\bigg)^{-1} \\
  &= \frac{1}{\nu_2} \prod_{p \leq W} \bigg( 1-\frac{\rho(p)}{p}\bigg)^{-1} \sum_{\n } H_{N'}(\n)\frac{ \beta_\n \xi_{k_j}\chi_j (\n)}{\N \n} + O_C((\log X)^{-C})
\end{align*}
Then Proposition \ref{typeiprop} follows from applying the Fundamental lemma of the sieve (Lemma \ref{flsievelemma}, see also Remark \ref{flsieveremark}) and Proposition \ref{typeiprop} to handle  $\mathbf{1}_{(\n, P(W))=1}$ in $S_j$. Note that by $N > X^{3\delta+3\eta}$ and $\M(u_j) \leq Q \leq X^{\delta+\eta}$ we get that $\N\m \ll X^{1-\delta-\eta}/\M(u_j) ^2$ which is required for Proposition \ref{typeiprop}. This gives us a main term of the form
\[
  \sum_{N'} \sum_{j \leq J } \sum_{\m,\n } \frac{\alpha_\m \beta_\n H_{N'} (\n)\xi_{k_j}\chi_j(\m\n)}{\N \m\n}   \sum_{\a} \frac{H_{N'} (\a/\m)}{\nu_2 } a^\omega_{\a} \overline{\xi_{k_j}\chi_j}(\a). 
\]
The weight $\frac{H_{N'} (\a/\m)}{\nu_2 }$ may be replaced (with a negligible error term) by $\frac{F(\N\m\n/\N \a)}{\widehat{F}(0) } $ by a further application of Poisson summation (Lemma \ref{poissonlemma}) on the free variable $a$ in $z=b+ia$, completing the proof.

\begin{remark} \label{bottleneckremark} There are two potential bottle-necks for improving the range of $\delta  < 1/10$ in Theorem \ref{asympregulartheorem}, namely, the exponent $3$ in Lemma \ref{mobiusregularlemma} and the exponent $2$ in $X^{1-\delta-\eta}/q^2$ in Proposition \ref{typeiprop}.  It is plausible that with more work these exponents may be improved to $2$ and $1$, respectively, which would suffice to prove Theorem \ref{asympregulartheorem} for $\delta < 1/8$. Both of these improvements run into quite delicate issues and we have decided not to pursue this here. It is not clear if  the boundary $1/8$ can be improved, but we certainly hit a hard barrier at $\delta =1/6$ as this when even the most optimistic the Type II range $[N^{2\delta}, N^{1/2-\delta}]$ becomes empty.
\end{remark}

\section{Proof of Theorem \ref{asympregulartheorem}} \label{asympregularproofsection} 
We apply a sieve argument to the sequence $\A=(a_n)$ over integers
\[
a_n := F(n/X')\sum_{\N \n = n} a_\n,
\]
where for convenience we have split $n$ into finer-than-dyadic intervals (as in Section \ref{smoothweightsection}) with $\nu=(\log X)^{-C}$ and $X'\sim X$. We also define an auxiliary sequence $\B_j$ by
\[
b^{(j)}_n :=  \sum_{\N \n = n}  \sum_{\m} \frac{F(\N \n /\N \m) \xi_{k_j} \chi_j(\n)}{ \hat{F}(0) \N \n } F(\N \m / X') a^{\omega}_\m \overline{\xi_{k_j}\chi_j}(\m),
\]
denoting $\xi_{k_0}\chi_0 = 1$ for $j=0.$
Then Theorem \ref{asympregulartheorem} follows by using the explicit formula to evaluate the sums
\[
\sum_{\n}  \Lambda(\N \n) \frac{F(\N \n /\N \m) \xi_{k_j} \chi_j(\n)}{\hat{F}(0) \N \n },
\]
 once we prove that for any $C_1>0$ there is some $C_2 >0$ such that
\begin{align} \label{sieveclaim}
   S(\A,\Lambda) = \sum_{n} a_n \Lambda(n) =  \sum_{0\leq j \leq J_1} S(\B_j,\Lambda)  + O\bigg( \frac{X^{1/2}|B|}{(\log X )^{C_1}}\bigg).
\end{align}

Let $Y=X^{3 \delta + 4\eta}$ and $Z=X^{1/2 + \delta+\eta}$. Then $YZ \leq X^{1-\delta-\eta}$ for some $\eta >0$ by $\delta <1/10$. By Vaughan's identity \cite{vaughan} for $n > Y $  we have
\begin{align} \label{vaughanid}
   \Lambda(n) = \sum_{\substack{b|n \\ b \leq Y}}  \mu(b) \log \frac{n}{b} - \sum_{\substack{bc|n  \\ b \leq Y \\c \leq Z}} \mu(b)\Lambda(c) +  \sum_{\substack{bc|n  \\ b > Y \\c > Z}} \mu(b)\Lambda(c). 
\end{align}
Applying this with both sides multiplied by $(n,P(W))=1$  we get
\[
S(\A,\Lambda) = S_1(\A) +S_2(\A) + S_3(\A) +  O\bigg( \frac{X^{1/2}|B|}{(\log X )^{C_1}}\bigg).
\]
and similarly for $j \leq J$ write
\[
S(\B_j,\Lambda) = S_1(\B_j) +S_2(\B_j) + S_3(\B_j) +  O\bigg( \frac{X^{1/2}|B|}{(\log X )^{C_1}}\bigg)
\]
The sums $S_1,S_2$ correspond to Type I sums and $S_3$ is a Type II sum, with all variables coprime to $P(W)$.

By Fundamental lemma of the sieve (Lemma \ref{flsievelemma}) and Type I information (Proposition \ref{typeiprop})  we get
for $k=1,2$
\[
S_k(\A) = S_k(\B_0) +O\bigg( \frac{X^{1/2}|B|}{(\log X )^{C_1}}\bigg) =  \sum_{0\leq j \leq J_1} S_k(\B_j)  + O\bigg( \frac{X^{1/2}|B|}{(\log X )^{C_1}}\bigg).
\]
 since for $j \geq 1$, $k=1,2$ for any $C > 0$
 \[
S_k(\B_j)  \ll_{C} X^{1/2}|B|(\log X)^{-C}.
 \]
 By Type II information (Proposition \ref{typeiiprop}, note that $Y < b \ll X/Z$) we get
 \[
 S_3(\A) = \sum_{1\leq j \leq J_j} S_3(\B_j)  + O\bigg( \frac{X^{1/2}|B|}{(\log X )^{C_1}}\bigg) = \sum_{0\leq j \leq J_j} S_3(\B_j) + O\bigg( \frac{X^{1/2}|B|}{(\log X )^{C_1}}\bigg)
 \]
since for any $C > 0$
 \[
S_3(\B_0)  \ll_{C} X^{1/2}|B|(\log X)^{-C}.
 \]
By recombining the Vaughan's identity for the $\B_j$ we get (\ref{sieveclaim}).
\section{Proof of Theorem \ref{regulartheorem}}  \label{sec:regularthmproof}
We have two cases, no zeros $\beta > 1-\eps_1/\log X$ or that in the case of a zero   $\beta > 1-\eps_1/\log X$ we have $\Omega(B_1) \leq \Omega(B)/2$.
\subsection{No zeros  $\beta > 1-\eps_1/\log X$} \label{sec:nozerocase}
Let us denote  by $a_\n, a^{\omega}_\n$ the sequences corresponding to $\lambda=\mathbf{1}_{B}$. By Theorem \ref{asympregulartheorem} we have 
 \begin{align} \label{thm33} \begin{split}
    \sum_{\N \n\sim X} a_\n  \Lambda(\N \n)  = & \frac{4}{\pi} \sum_{\N \n\sim X} a^{\omega}_\n \bigg(1- \sum_{j\leq J} \overline{\xi_{k_j}\chi_j}(\n) \sum_{\substack{\rho_j \\ L(\rho_j,\xi_{k_j}\chi_j) = 0 \\ |\Im(\rho_j)| \leq X^\eta }}(\N \n)^{\rho_j-1}\bigg) 
  \\
  & \hspace{100pt}+ O\bigg( \frac{1}{(\log X)^{C_1}} X^{1/2} \sum_{b} |\lambda_b|\bigg).   
 \end{split}
\end{align}
If there is a zero $\beta_1 \geq 1-\frac{1}{\sqrt{\delta} \log X}$ as in Lemma \ref{differentmodulilemma} corresponding to $\chi_1 $ real and $\xi_{k_1}=1$, then  $\beta_1 \leq 1-\eps_1/\log X$  and the contribution from that zero is
\begin{align*}
 -   \frac{4}{\pi} \sum_{\N \n\sim X} a^{\omega}_\n   \chi_1(\n) (\N \n)^{\beta_1-1}  \geq & - \frac{4}{\pi} \sum_{\N \n\sim X} a^{\omega}_\n (\exp(-\eps_1)  + o(1)) \\
 \geq&  - \frac{4}{\pi} \sum_{\N \n\sim X} a^{\omega}_\n (1-\eps_1/2) 
\end{align*}
since $\eps_1 < 1/10$. Therefore, the contribution from the first two terms in \eqref{thm33} is
\begin{align*}
 \frac{4}{\pi} \sum_{\N \n\sim X} a^{\omega}_\n  \bigg(1-\chi_1(\n) (\N \n)^{\beta_1-1} \bigg)   \gg \eps_1  \sum_{\N \n\sim X} a^{\omega}_\n 
\end{align*}
Denoting $\sigma_0=1-\frac{1}{\sqrt{\delta} \log X}$, by Lemma \ref{differentmodulilemma} the remaining zeros satisfy $\beta_j 
\leq  \sigma_0$ and they contribute at most
\begin{align*}
    & \ll  \sum_{\N \n\sim X} a^{\omega}_\n \sum_{j\leq J} \sum_{\substack{\rho_j \\ L(\rho_j,\xi_{k_j}\chi_j) = 0 \\ |\Im(\rho_j)| \leq X^\eta \\  \beta_j 
\leq \sigma_0 }} X^{\beta_j-1}  \\
&= \Omega(B)\sum_{j\leq J} \sum_{\substack{\rho_j \\ L(\rho_j,\xi_{k_j}\chi_j) = 0 \\ |\Im(\rho_j)| \leq X^\eta \\  \beta_j 
\leq \sigma_0 }} \bigg( \int_{1/2}^{\beta_j} X^{\sigma-1} \log X d \sigma  +  X^{-1/2}  \bigg) \\
 &\ll \Omega(B) \int_{1/2}^{\sigma_0} N^\ast(\sigma, X^\eta,X^\eta, X^{2\delta + 2\eta}) X^{\sigma-1} \log X d \sigma + \Omega(B) X^{-1/2+ O(\delta+\eta)}.
\end{align*}
The last term is negligible once $\delta,\eta$ are sufficiently small. By Lemma \ref{zerodensitylemma} the integral is bounded by
 \begin{align*}
     &\int_{1/2}^{\sigma_0} N^\ast(\sigma, X^\eta,X^\eta, X^{2\delta + 2\eta}) X^{\sigma-1} \log X d \sigma \ll      \int_{1/2}^{\sigma_0} X^{ c_2(4\delta+6\eta)(1-\sigma)+\sigma-1} \log X d \sigma  \\
     &\ll X^{(\sigma_0-1)/2}= \exp(-\tfrac{1}{2 \sqrt{\delta}} )
 \end{align*}
once $\delta,\eta$ are sufficiently small compared to the constant $c_2$ in  Lemma \ref{zerodensitylemma}. Combining all of the above estimates we have
\begin{align*}
     \sum_{\N \p\sim X} a_\p  \gg  \frac{1}{\log X} (\eps_1 - O(\exp(-\tfrac{1}{2 \sqrt{\delta}} ) ))  \sum_{\N \n\sim X} a^{\omega}_\n  \gg  \frac{1}{\log X} \eps_1  \sum_{\N \n\sim X} a^{\omega}_\n 
\end{align*}
once $\delta$ is small enough compared to $\eps_1$.

\subsection{Zero  $\beta > 1-\eps_1/\log X$ and $\Omega(B_1) \leq \Omega(B)/2$}
Let us call $B_2= B \setminus B_1 $ so that $\Omega(B_2) \geq \Omega(B)/2$ and for all $b \in B_2$ we have
\begin{align} \label{B2condition}
 \sum_{\substack{a \sim (X-b^2)^{1/2} \\ (a,b)=1 \\ (a^2+b^2,2)=1}} \chi_1 ((b+ia)) \leq 0.  
\end{align}
Let us denote  by $a_\n, a^{\omega}_\n$ the sequences corresponding to $\lambda=\mathbf{1}_{B}$ and  $a^{(2)}_\n,a^{(2) \omega}_\n$ the sequences corresponding to $\lambda=\mathbf{1}_{B_2}$. Then by Theorem \ref{asympregulartheorem} we have
\begin{align*}
    & (\log X)\sum_{\N \p\sim X} a_\p \geq      \log(X)\sum_{\N \p\sim X} a_\p^{(2)}  \\
    &\gg    \sum_{\N \n\sim X} a_\n^{(2)}  \Lambda(\N \n)  =  \frac{4}{\pi} \sum_{\N \n\sim X} a^{(2)\omega}_\n \bigg(1- \sum_{j\leq J} \overline{\xi_{k_j}\chi_j}(\n) \sum_{\substack{\rho_j \\ L(\rho_j,\xi_{k_j}\chi_j) = 0 \\ |\Im(\rho_j)| \leq X^\eta }}(\N \n)^{\rho_j-1}\bigg) 
  \\
  & \hspace{200pt}+ O\bigg( \frac{1}{(\log X)^{C_1}} X^{1/2} \sum_{b} |\lambda_b|\bigg).
\end{align*}
The first term contributes 
\begin{align*}
   \sum_{\N \n\sim X} a^{(2)\omega}_\n = \Omega(B_2) \geq \frac{1}{2}\Omega(B) =    \frac{1}{2}\sum_{\N \n\sim X} a^{\omega}_\n. 
\end{align*}
The contribution from $2 \leq j \leq J$ is handled similarly as in Section \ref{sec:nozerocase} and similarly for $j=1$ the zeros $\beta \leq 1-\eps_1/\log X$. The contribution from the Siegel zero $\beta > 1-\eps_1/\log X$  for $j=1$ is essentially positive, since $\xi_{k_1}=1$, $\chi_1$ is real, and by 
\eqref{B2condition}
\begin{align*}
   - \sum_{\N \n\sim X} a^{(2)\omega}_\n  \chi_1(\n) (\N \n)^{\beta-1} &= -\sum_{b \in B_2}  \sum_{\substack{a \sim (X-b^2)^{1/2} \\ (a,b)=1 \\ (a^2+b^2,2)=1}} \chi_1 ((b+ia)) (a^2+b^2)^{\beta -1}  \\
  & \geq -\sum_{b \in B_2}  \sum_{\substack{a \sim (X-b^2)^{1/2} \\ (a,b)=1 \\ (a^2+b^2,2)=1}} \big(\chi_1 ((b+ia)) +2\eps_1 \big) \\
  & \geq  - 2\eps_1 \Omega(B_2) \geq  - 2\eps_1 \sum_{\N \n\sim X} a^{\omega}_\n 
\end{align*}
since $\eps_1 < 1/10$. Therefore, we conclude that also in the second case
\begin{align*}
    \sum_{\N \p\sim X} a_\p \gg  \frac{1}{\log X} (1-2 \eps_1 O(\exp(-\tfrac{1}{2 \sqrt{\delta}} ) ))  \sum_{\N \n\sim X} a^{\omega}_\n \gg  \sum_{\N \n\sim X} a^{\omega}_\n 
\end{align*} 
once $\delta$ is sufficiently small. \qed
\section{Proof of Theorem \ref{asymptotictheorem}} 
\label{asymptoticproofsection}
By similar reductions as in Section \ref{setupsection} it suffices to consider $\lambda_b= \mathbf{1}_B(b)$. The goal is to show that if $u_j$ is a modulus of one of the characters $\xi_{k_j} \chi_j$ and $b$ does not have a large common factor with $u_j$, then the sum over the free variable $a$ of $\xi_{k_j} \chi_j(b+ia)$ exhibits cancellation.   By Theorem \ref{asympregulartheorem} and the Siegel-Walfisz bound \cite[Lemma 
 16.1]{FI} for small moduli $u_j$, it suffices to show that for any $j\leq J_1$ and $Y \in [X^{1/2-\eta},2X^{1/2}]$ and any $|u_j|^2 \gg (\log X)^{C''}$ with $C''$ large compared to $C'$ we have
\begin{align*}
    S'_j := \sum_{b \in B}\bigg|\sum_{\substack{a \in (Y,Y(1+X^{-\eta})] \\(a,b)=1  \\ a \equiv a_0 \, (4)}}   \chi_j(b+ia) \bigg| \ll \frac{X^{-\eta}Y|B|}{(\log X)^C}.
\end{align*}
Note that the weight $\xi_j((b+ia))$ has been removed by splitting $a$ into finer-than-dyadic intervals and using \eqref{Bintervalassumption} to note that then $b+ia$ lives in a small box.

We write
\[
S'_j = \sum_{v | u_j} \sum_{\substack{b \in B \\(b,u_1)=v}}\bigg|\sum_{\substack{a (Y,Y(1+X^{-\eta})] \\ (a,b)=1  \\ a \equiv a_0 \, (4)}}   \chi_j(b+ia) \bigg| .
\]
For $|v| > |u_j|/(\log X)^{C'/2}$ we use the assumption (\ref{lambdaassumption})  to get
\[
 \ll \frac{X^{-\eta}Y |B|}{(\log X)^{C'}}\sum_{\substack{v | u_j \\ |v| > |u_1|/(\log X)^{C'/2}}} 1 =  \frac{X^{-\eta}Y |B|}{(\log X)^{C'} }\sum_{\substack{v | u_j \\ |v| \leq (\log X)^{C'/2}}} 1\ll  \frac{X^{-\eta}Y|B|}{(\log X)^{C+C_2} }
\]
once $C'$ is large compared to $C_1$ and $C$.
For  $|v| \leq |u_j|/(\log X)^{C'/2}$ we use Lemma \ref{charsumboundlemma} to get
\[
S_{>} \ll X^{-\eta} Y \sum_{\substack{v|u_j \\  |v| \leq |u_1|/(\log X)^{C'/2}}} \frac{1}{|u/v|^{1/3}} \sum_{\substack{b\in B \\ (b,u_1)=v}} 1 \ll  \frac{X^{-\eta}Y |B|}{(\log X)^{C+C_2} }.
\]
by taking $C'$ is large compared to $C_1$ and $C$.
To evaluate the main term we note that
\[
\sum_{\substack{a^2+b^2 \sim X \\ (a,b)=1 \\ (b+ia,2)=1}} 2 \omega_2(b) = (1+O(X^{-\eta})) \sum_{\substack{a^2+b^2 \sim X \\ (a,b)=1}} \omega(b). \qed
\]
\section{Proof of Theorem \ref{GRHtheorem}} \label{grhsection}
This follows immediately from Theorem \ref{asympregulartheorem} with the zero-density estimate Lemma \ref{zerodensitylemma}  once $C'$ is sufficiently large, via similar arguments as in Section \ref{sec:nozerocase} \qed

\subsection*{Acknowledgements}
I am grateful to Akshat Mudgal for numerous discussions and to Lasse Grimmelt for suggestions on constructing an approximation for primes, as well as to James Maynard for encouragement and helpful comments. I also wish to thank the anonymous referee for comments. This project has
received funding from the European Research Council (ERC) under the European Union's Horizon 2020 research and innovation programme (grant agreement No 851318).

\bibliography{gaussianprimes}

\begin{thebibliography}{10}

\bibitem{coleman_1990}
M.~D. Coleman.
\newblock A zero-free region for the hecke l-functions.
\newblock {\em Mathematika}, 37(2):287–304, 1990.

\bibitem{drappeau}
S.~Drappeau.
\newblock Théorèmes de type fouvry–iwaniec pour les entiers friables.
\newblock {\em Compositio Mathematica}, 151(5):828–862, 2015.

\bibitem{fouvryi}
E.~Fouvry and H.~Iwaniec.
\newblock Gaussian primes.
\newblock {\em Acta Arith.}, 79(3):249--287, 1997.

\bibitem{fisieve}
J.~Friedlander and H.~Iwaniec.
\newblock Asymptotic sieve for primes.
\newblock {\em Ann. of Math. (2)}, 148(3):1041--1065, 1998.

\bibitem{FI}
J.~Friedlander and H.~Iwaniec.
\newblock The polynomial {$X^2+Y^4$} captures its primes.
\newblock {\em Ann. of Math. (2)}, 148(3):945--1040, 1998.

\bibitem{odc}
J.~Friedlander and H.~Iwaniec.
\newblock {\em Opera de cribro}, volume~57 of {\em American Mathematical
  Society Colloquium Publications}.
\newblock American Mathematical Society, Providence, RI, 2010.

\bibitem{fiillusory}
J.~B. Friedlander and H.~Iwaniec.
\newblock The illusory sieve.
\newblock {\em Int. J. Number Theory}, 1(4):459--494, 2005.

\bibitem{gallagher}
P.~X. Gallagher.
\newblock A large sieve density estimate near {$\sigma =1$}.
\newblock {\em Invent. Math.}, 11:329--339, 1970.

\bibitem{harman}
G.~Harman.
\newblock {\em Prime-detecting sieves}, volume~33 of {\em London Mathematical
  Society Monographs Series}.
\newblock Princeton University Press, Princeton, NJ, 2007.

\bibitem{hblinnik}
D.~R. Heath-Brown.
\newblock Zero-free regions for dirichlet l-functions, and the least prime in
  an arithmetic progression.
\newblock {\em Proceedings of the London Mathematical Society},
  s3-64(2):265--338, 1992.

\bibitem{hb}
D.~R. Heath-Brown.
\newblock Primes represented by {$x^3+2y^3$}.
\newblock {\em Acta Math.}, 186(1):1--84, 2001.

\bibitem{hbli}
D.~R. Heath-Brown and X.~Li.
\newblock Prime values of {$a^2 + p^4$}.
\newblock {\em Invent. Math.}, 208(2):441--499, 2017.

\bibitem{Huxley}
M.~N. Huxley.
\newblock Large values of {D}irichlet polynomials. {III}.
\newblock {\em Acta Arith.}, 26(4):435--444, 1974/75.

\bibitem{IK}
H.~Iwaniec and E.~Kowalski.
\newblock {\em Analytic number theory}, volume~53 of {\em American Mathematical
  Society Colloquium Publications}.
\newblock American Mathematical Society, Providence, RI, 2004.

\bibitem{lsx}
P.~C.-H. Lam, D.~Schindler, and S.~Y. Xiao.
\newblock On prime values of binary quadratic forms with a thin variable.
\newblock {\em Journal of the London Mathematical Society}, 102(2):749--772,
  2020.

\bibitem{li}
X.~Li.
\newblock Prime values of a sparse polynomial sequence.
\newblock {\em to appear in Duke Math. J.}, 2021.

\bibitem{linnik}
U.~V. Linnik.
\newblock On the least prime in an arithmetic progression. {I}. {T}he basic
  theorem.
\newblock {\em Rec. Math. [Mat. Sbornik] N.S.}, 15(57):139--178, 1944.

\bibitem{maynard}
J.~Maynard.
\newblock Primes represented by incomplete norm forms.
\newblock {\em Forum Math. Pi}, 8:e3, 2020.

\bibitem{merikoskipoly}
J.~Merikoski.
\newblock The polynomials x2+(y2+1)2 and x2+(y3+z3)2 also capture their primes.
\newblock {\em Preprint, arXiv:2112.03617}, 2022.

\bibitem{mvexceptional}
H.~L. Montgomery and R.~C. Vaughan.
\newblock The exceptional set in {G}oldbach's problem.
\newblock {\em Acta Arith.}, 27:353--370, 1975.

\bibitem{mv}
H.~L. Montgomery and R.~C. Vaughan.
\newblock {\em Multiplicative number theory. {I}. {C}lassical theory},
  volume~97 of {\em Cambridge Studies in Advanced Mathematics}.
\newblock Cambridge University Press, Cambridge, 2007.

\bibitem{pratt}
K.~Pratt.
\newblock Primes from sums of two squares and missing digits.
\newblock {\em Proc. Lond. Math. Soc. (3)}, 120(6):770--830, 2020.

\bibitem{Ten}
G.~Tenenbaum.
\newblock {\em Introduction to analytic and probabilistic number theory},
  volume 163 of {\em Graduate Studies in Mathematics}.
\newblock American Mathematical Society, Providence, RI, third edition, 2015.
\newblock Translated from the 2008 French edition by Patrick D. F. Ion.

\bibitem{vaughan}
R.~C. Vaughan.
\newblock {Mean Value Theorems in Prime Number Theory}.
\newblock {\em Journal of the London Mathematical Society}, s2-10(2):153--162,
  05 1975.

\end{thebibliography}
\bibliographystyle{abbrv}
\end{document}